\documentclass[10pt,a4paper]{amsart}
\usepackage[T1]{fontenc}
\usepackage{lmodern}
\usepackage{microtype}
\usepackage[a4paper,textwidth=152mm,textheight=232mm,
  centering,headheight=14pt,headsep=12pt,footskip=24pt]{geometry}
\usepackage{amsmath,amssymb,mathtools,bm}
\usepackage{booktabs,tabularx,array}
\usepackage{enumitem}
\usepackage{flafter}
\makeatletter
\def\fps@figure{t}
\def\fps@table{t}
\makeatother
\usepackage{xcolor}
\usepackage{tikz}
\usetikzlibrary{calc}
\usepackage{pgfplots}
\pgfplotsset{compat=1.18}
\usepgfplotslibrary{fillbetween}
\usepackage[hypertexnames=false]{hyperref}
\usepackage[nameinlink,noabbrev]{cleveref}
\hypersetup{colorlinks=true,linkcolor=blue!45!black,
  citecolor=blue!45!black,urlcolor=blue!45!black,
  pdfauthor={Thijs Laarhoven},
  pdftitle={Spherical Statistics and Phase Transitions in High-Dimensional Random Lattices}}
\numberwithin{equation}{section}

\newtheorem{theorem}{Theorem}[section]
\newtheorem{proposition}[theorem]{Proposition}
\newtheorem{lemma}[theorem]{Lemma}
\newtheorem{corollary}[theorem]{Corollary}

\theoremstyle{definition}
\newtheorem{definition}[theorem]{Definition}
\theoremstyle{remark}
\newtheorem{remark}[theorem]{Remark}

\newcommand{\R}{\mathbb R}
\newcommand{\Z}{\mathbb Z}
\newcommand{\E}{\mathbb E}
\newcommand{\Prb}{\mathbb P}
\newcommand{\Var}{\operatorname{Var}}
\newcommand{\vol}{\operatorname{vol}}
\newcommand{\cL}{\mathcal L}
\newcommand{\Rstar}{R^*}
\newcommand{\Rell}{R_\ell}
\newcommand{\1}{\mathbf 1}
\newcommand{\wt}{\omega}
\newcommand{\Rcrit}{R_{\mathrm{crit}}}
\newcommand{\TV}{d_{\mathrm{TV}}}
\newcommand{\Span}{\operatorname{span}}

\definecolor{figblue}{RGB}{0,92,153}
\definecolor{figorange}{RGB}{213,94,0}
\definecolor{figgreen}{RGB}{0,138,98}
\definecolor{figpurple}{RGB}{117,112,179}

\title[Spherical statistics and phase transitions]
{Spherical Statistics and Phase Transitions \\ in
High-Dimensional Random Lattices}
\author{Thijs Laarhoven}
\date{12 September 2026; version 26}

\begin{document}
\begin{abstract}
We study the spherical statistics of all the points in a thin shell of a
high-dimensional random lattice. The exact relations between lattice
points make it unclear when predictions based only on spherical geometry
should hold. We answer this question for several statistics, including the
number of shell points, the balance of their directions, and the occurrence,
repetition, and distribution of differences between them. We identify sharp
thresholds as the shell radius grows and show that these statistics undergo
several distinct phase transitions. A shell can already agree with one
geometric prediction while still differing strongly from another. Our
results hold for a single sampled lattice, with probability tending to one
as the dimension grows. They include estimates that hold across a complete
shell, bounds close to the transition thresholds, and extensions to randomly
shifted lattices. A Lean formalization verifies the main results, assuming
the classical formulas and probability model stated in the code.
\end{abstract}

\maketitle
\enlargethispage{5pt}

\section{Introduction and overview}\label{sec:introduction}

\subsection{Random lattices between structure and randomness}

A lattice $\cL\subset\R^n$ is a periodic set of points in dimension $n$.
Its points are linked by exact rules: sums and differences of lattice
points are again lattice points. Even so, some statistics of a
high-dimensional random lattice are close to those of a random point cloud.
We use the Haar-random unimodular model: the lattice has determinant one
and is sampled from the natural invariant probability distribution on such
lattices.

Let $\Rstar$ be the radius of a Euclidean ball of volume one. This is the
natural length scale at which the first nonzero lattice points appear.
We study a thin spherical shell of radius $R=c\Rstar$, where $c>1$ is a
fixed radius factor. Write $S_R$ for the Euclidean shell consisting of
vectors with lengths between $R(1-n^{-2})$ and $R(1+n^{-2})$, and let
$A_R:=\cL\cap S_R$ be the set of all lattice points in this shell. We write
$N_R:=|A_R|$ for the number of these points and $P_R:=\vol(S_R)$ for the
Euclidean shell volume. Both have exponential scale $c^{n+o(n)}$; for the
random count, this holds with high probability. A shell with so many points
may seem similar to a large sample of independent uniform points on a
sphere. Its lattice relations, however, still matter. The central question
is:

\begin{quote}
\emph{Which geometric statistics of the complete shell of one typical
high-dimensional random lattice agree with predictions from a continuous
spherical shell, and how large must the shell radius be for each agreement
to hold?}
\end{quote}

The word \emph{complete} matters: we include every lattice point in the
shell. Classical mean-value formulas average over lattices. We instead
seek statements that hold for all relevant points or differences of one
sampled lattice at once. Such a statement is called \emph{quenched}; an
average over lattices is called \emph{annealed}.

Our main statistic counts pairs of shell points with a given difference.
We call the two points the \emph{parents} of that difference. For a fixed
lattice vector $\mathbf d$, let $r_{\cL,R}(\mathbf d)$ count the possible
first parents:
\begin{equation}\label{eq:intro-representation-count}
 r_{\cL,R}(\mathbf d):=
 \#\{\mathbf x\in A_R:\mathbf x-\mathbf d\in A_R\}.
\end{equation}
Each first parent $\mathbf x$ determines exactly one second parent
$\mathbf y=\mathbf x-\mathbf d$. Thus $r_{\cL,R}(\mathbf d)$ also counts
the ordered pairs of shell points whose difference is $\mathbf d$.
The geometric prediction is the volume where the shell overlaps a
translated copy of itself. We call this overlap a spherical lens and write
\[
 K_R(\mathbf d):=\vol\bigl(S_R\cap(S_R+\mathbf d)\bigr).
\]
Here $S_R+\mathbf d$ is the shell translated by $\mathbf d$.
We ask when $r_{\cL,R}(\mathbf d)$ is close to $K_R(\mathbf d)$ for every
relevant target $\mathbf d$ at once. This tells us which differences occur,
how often they occur, and how they are distributed. A first question is
whether
\[
 A_R\subseteq A_R-A_R:
\]
can every vector in the shell itself be written as a difference of two
points from that same shell? The lens volume predicts the number of
possible parents, and $r_{\cL,R}(\mathbf d)$ counts the lattice points in
that lens.

\subsection{The main answers at a glance}

Saying that a lattice shell ``looks random'' is not precise enough.
Different statistics can give different answers. The results below concern
specific statistics; they do not say that all shell points are jointly
independent.

The following informal statement gives the main results in terms of the
counts just defined. It leaves out error constants and the technical
conditions for more general ranges of difference lengths. The threshold
claims concern fixed radius factors on either side of the stated boundaries.

\begin{theorem}[Main results, informal]
\label{thm:main-overview}
Sample a Haar-random lattice and consider its complete shell $A_R$ at radius
$R=c\Rstar$ for a fixed radius factor $c>1$. As the dimension
tends to infinity, the following conclusions hold with high probability.
The items follow the order of the technical sections.
\begin{enumerate}[label=\textup{(\roman*)},leftmargin=*]
\item \textbf{Threshold $c=1$: shell size} \textup{(see Lemma~\ref{lem:thin-shell-count} for details).}
The number of shell points is close to the Euclidean volume of the shell:
\[
 N_R=(1+o(1))P_R.
\]
\item \textbf{Threshold $c=2/\sqrt3$: parents of every shell vector} \textup{(see Theorem~\ref{thm:coverage} for details).}
Below the threshold, for $1<c<2/\sqrt3$, only an exponentially small
fraction of shell vectors are differences of two shell points. Above it,
for $c>2/\sqrt3$, every shell vector is such a difference, and all parent
counts agree with their lens predictions up to a relative error tending
to zero:
\[
 r_{\cL,R}(\mathbf d)=(1+o(1))K_R(\mathbf d)
 \qquad\text{for every }\mathbf d\in A_R.
\]
\item \textbf{Threshold $c=2/\sqrt3$: parents in a random translate} \textup{(see Corollary~\ref{cor:affine-same-shell} for details).}
For $c>2/\sqrt3$, the same parent-count law holds when the target and
first parent lie in the radius-$R$ shell of one independently and uniformly
translated lattice, while the second parent lies in $A_R$. The numbers of
allowed partners of all translated shell points follow the same lens
prediction. Probability here is over both the lattice and this one random
translate.
\item \textbf{Threshold $c=3\sqrt3/4$: collisions between pairs} \textup{(see Theorem~\ref{thm:energy} for details).}
Let $E$ count the ordered quadruples of shell points for which the first
pair and the second pair have the same difference. Equivalently,
$E=\sum_{\mathbf d\in\cL}r_{\cL,R}(\mathbf d)^2$.
Its exponential growth rate is
\[
 \frac{\log E}{n}
 =\max\left\{2\log c,\,
        3\log c+\log\frac4{3\sqrt3}\right\}+o(1).
\]
The first term is larger below the threshold and comes from unavoidable
coincidences, such as choosing the same pair twice. The second is larger
above the threshold and is the prediction from lens volumes. At the
threshold the two terms are equal.
\item \textbf{Threshold $c=\sqrt2$: the full difference distribution} \textup{(see Theorem~\ref{thm:convolution} for details).}
Let $p(\mathbf d)$ be the probability that two independent uniform shell
points have difference $\mathbf d$. Let $q(\mathbf d)$ assign a
probability to that same lattice vector in proportion to its lens volume:
\[
 p(\mathbf d)=\frac{r_{\cL,R}(\mathbf d)}{N_R^2},
 \qquad
 q(\mathbf d)=\frac{K_R(\mathbf d)}
                   {\sum_{\mathbf z\in\cL}K_R(\mathbf z)}.
\]
Their total-variation (TV) distance is
$\tfrac12\sum_{\mathbf d\in\cL}|p(\mathbf d)-q(\mathbf d)|$. It tends to
one below the threshold, for $1<c<\sqrt2$, and to zero above it, for
$c>\sqrt2$. Both distributions are supported on the same lattice: the
geometric prediction assigns weights to those lattice points.
\item \textbf{Threshold $c=1$: directional balance} \textup{(see Theorems~\ref{thm:dense-shell-isometry} and~\ref{thm:linear-degree-tensor-moments} for details).}
For every fixed $c>1$, project the normalized shell points onto any
direction. Their average squared projection agrees with the prediction
from a uniform sphere, in all directions at once. Higher even moments
agree in the sense described below, through degrees proportional to the
dimension.
\end{enumerate}
\end{theorem}

The references in each item give the precise assumptions and error bounds.
The main tool is a bound on high moments of parent and partner counts that
holds uniformly over the required parameter ranges
(Propositions~\ref{thm:intrinsic-marked-moments}
and~\ref{thm:affine-marked-moments} in Section~\ref{sec:shell-moments}).
At a fixed density above the coverage threshold, a sufficiently large fixed
moment order is enough. To study a shrinking window around the threshold,
we need orders proportional to $\log n$. The energy, total-variation, and
angular results use only low-order moments. The rest of the introduction
explains the geometry and meaning of these results.

\subsection{Why pair counts in a shell matter}

In high dimensions, volume changes quickly with distance from the origin.
Restricting points to a thin shell keeps their lengths nearly fixed and
lets us study their directions and sums or differences. Spherical-cap and
lens volumes give natural predictions for pair counts. Such predictions
also appear in random-spherical models used to estimate cryptographic
security. Here we ask whether they hold for the complete shell of one
sampled lattice.

With an independent random translate, pairs with one point in the
translated lattice and one in the original lattice describe residual
vectors for a random target in the closest vector problem. The scope below
explains how these geometric results differ from claims about algorithms
or structured lattices.

\subsection{What the difference statistics measure}

We write the shell radius as $R=c\Rstar$. As $c$ grows, the shell contains
more lattice points. Different statistics, however, start agreeing with
spherical predictions at different values of $c$.

Coverage asks whether each relevant difference occurs at least once.
Collision energy counts repeated differences. After normalization, it is
the probability that two independently chosen ordered pairs have the same
difference. Total variation compares the full difference distributions and
bounds their disagreement for every bounded test function. These are
different questions. Total-variation convergence alone guarantees neither
exact coverage nor agreement of collision-energy exponents.

These three thresholds therefore answer different questions about the
same shell. Figure~\ref{fig:shell-phases} brings them together with the
other results of Theorem~\ref{thm:main-overview}, with one row per item.

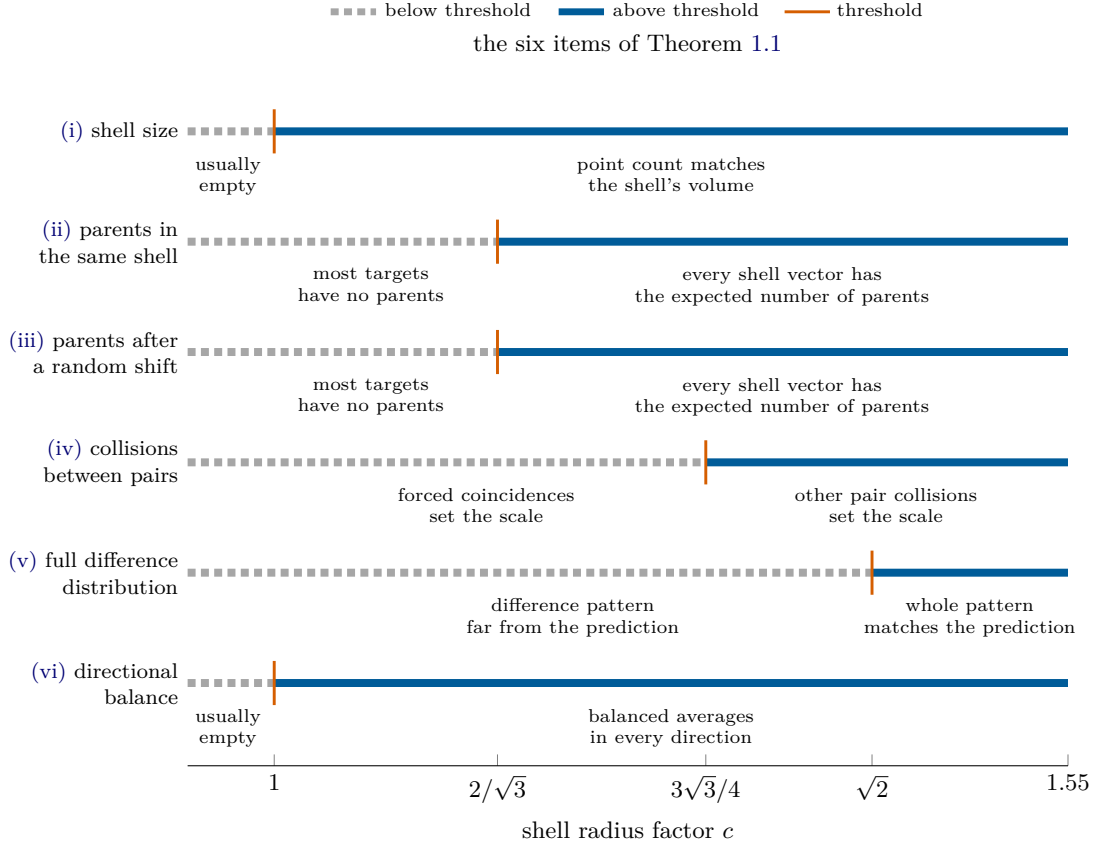
\begin{figure}[t]
\centering
\begin{tikzpicture}
\begin{axis}[
  width=.87\textwidth,
  height=10.7cm,
  xmin=.94,
  xmax=1.55,
  ymin=.25,
  ymax=6.50,
  axis x line*=bottom,
  axis y line*=left,
  y axis line style={draw=none},
  ytick style={draw=none},
  xlabel={shell radius factor $c$},
  ytick={1,2,3,4,5,6},
  yticklabels={
    \hyperref[thm:main-overview]{(vi)} directional\\balance,
    \hyperref[thm:main-overview]{(v)} full difference\\distribution,
    \hyperref[thm:main-overview]{(iv)} collisions\\between pairs,
    \hyperref[thm:main-overview]{(iii)} parents after\\a random shift,
    \hyperref[thm:main-overview]{(ii)} parents in\\the same shell,
    \hyperref[thm:main-overview]{(i)} shell size},
  xtick={1,1.1547005,1.2990381,1.4142136,1.55},
  xticklabels={$1$,$2/\sqrt3$,$3\sqrt3/4$,$\sqrt2$,$1.55$},
  tick label style={font=\small},
  yticklabel style={font=\footnotesize,align=right},
  label style={font=\small},
  title={the six items of Theorem~\ref{thm:main-overview}},
  title style={font=\small},
  legend style={
    at={(.5,1.065)},anchor=south,legend columns=3,
    draw=none,fill=none,font=\scriptsize,
    /tikz/every even column/.append style={column sep=7pt}
  },
  legend cell align=left,
  clip=false
]
\addlegendimage{gray!70,densely dashed,line width=2.6pt}
\addlegendentry{below threshold}
\addlegendimage{figblue,line width=2.6pt}
\addlegendentry{above threshold}
\addlegendimage{figorange,line width=1.1pt}
\addlegendentry{threshold}
\addplot[gray!70,densely dashed,line width=3pt]
  coordinates {(.94,6) (1,6)};
\addplot[figblue,line width=3pt]
  coordinates {(1,6) (1.55,6)};
\addplot[gray!70,densely dashed,line width=3pt]
  coordinates {(.94,5) (1.1547005,5)};
\addplot[figblue,line width=3pt]
  coordinates {(1.1547005,5) (1.55,5)};
\addplot[gray!70,densely dashed,line width=3pt]
  coordinates {(.94,4) (1.1547005,4)};
\addplot[figblue,line width=3pt]
  coordinates {(1.1547005,4) (1.55,4)};
\addplot[gray!70,densely dashed,line width=3pt]
  coordinates {(.94,3) (1.2990381,3)};
\addplot[figblue,line width=3pt]
  coordinates {(1.2990381,3) (1.55,3)};
\addplot[gray!70,densely dashed,line width=3pt]
  coordinates {(.94,2) (1.4142136,2)};
\addplot[figblue,line width=3pt]
  coordinates {(1.4142136,2) (1.55,2)};
\addplot[gray!70,densely dashed,line width=3pt]
  coordinates {(.94,1) (1,1)};
\addplot[figblue,line width=3pt]
  coordinates {(1,1) (1.55,1)};
\draw[figorange,line width=1.1pt]
  (axis cs:1,5.80)--(axis cs:1,6.20);
\draw[figorange,line width=1.1pt]
  (axis cs:1.1547005,4.80)--(axis cs:1.1547005,5.20);
\draw[figorange,line width=1.1pt]
  (axis cs:1.1547005,3.80)--(axis cs:1.1547005,4.20);
\draw[figorange,line width=1.1pt]
  (axis cs:1.2990381,2.80)--(axis cs:1.2990381,3.20);
\draw[figorange,line width=1.1pt]
  (axis cs:1.4142136,1.80)--(axis cs:1.4142136,2.20);
\draw[figorange,line width=1.1pt]
  (axis cs:1,.80)--(axis cs:1,1.20);
\node[font=\scriptsize,align=center,anchor=north]
  at (axis cs:.968,5.85) {usually\\empty};
\node[font=\scriptsize,align=center,anchor=north]
  at (axis cs:1.275,5.85) {point count matches\\the shell's volume};
\node[font=\scriptsize,align=center,anchor=north]
  at (axis cs:1.067,4.85) {most targets\\have no parents};
\node[font=\scriptsize,align=center,anchor=north]
  at (axis cs:1.352,4.85) {every shell vector has\\the expected number of parents};
\node[font=\scriptsize,align=center,anchor=north]
  at (axis cs:1.067,3.85) {most targets\\have no parents};
\node[font=\scriptsize,align=center,anchor=north]
  at (axis cs:1.352,3.85) {every shell vector has\\the expected number of parents};
\node[font=\scriptsize,align=center,anchor=north]
  at (axis cs:1.147,2.85) {forced coincidences\\set the scale};
\node[font=\scriptsize,align=center,anchor=north]
  at (axis cs:1.424,2.85) {other pair collisions\\set the scale};
\node[font=\scriptsize,align=center,anchor=north]
  at (axis cs:1.207,1.85) {difference pattern\\far from the prediction};
\node[font=\scriptsize,align=center,anchor=north]
  at (axis cs:1.482,1.85) {whole pattern\\matches the prediction};
\node[font=\scriptsize,align=center,anchor=north]
  at (axis cs:.968,.85) {usually\\empty};
\node[font=\scriptsize,align=center,anchor=north]
  at (axis cs:1.275,.85) {balanced averages\\in every direction};
\end{axis}
\end{tikzpicture}
\caption[Thresholds for the six main results.]{Thresholds for the six
items of Theorem~\ref{thm:main-overview}, shown in the same order.
Orange marks a threshold; gray and blue show the ranges below and above
it. Items~\textup{(ii)} and~\textup{(iii)} share one boundary.
Item~\textup{(ii)} uses the same radius for parents and targets;
\textup{(iii)} extends its above-threshold conclusion to a random shift.
In~\textup{(iv)}, forced coincidences include repeating a pair, zero
differences, and pairs related by symmetry. Here ``scale'' means
exponential growth. In~\textup{(v)}, the predicted and observed patterns
are compared on the same lattice.}
\label{fig:shell-phases}
\end{figure}

The three larger thresholds in Figure~\ref{fig:shell-phases} come from two basic
geometric calculations. For a difference of length $\|\mathbf d\|$, define
its relative length, or \emph{relative lag}, by $\alpha:=\|\mathbf d\|/R$.
For fixed $0<\alpha<2$, the lens volume has exponential scale
\[
 K_R(\mathbf d)=\exp\bigl(n g_c(\alpha)+o(n)\bigr),
 \qquad
 g_c(\alpha):=\log c+\frac12\log\left(1-\frac{\alpha^2}{4}\right).
\]
Figure~\ref{fig:lens-exponent-map} shows where this exponent is zero. It
also compares the two terms that determine the collision energy, usually
called \emph{additive energy}.

\begin{figure}[t]
\centering
\begin{minipage}[t]{.49\textwidth}
\centering\vspace{0pt}
\begin{tikzpicture}
\begin{axis}[
  width=\linewidth,height=6.5cm,
  xmin=1,xmax=1.56,ymin=0,ymax=1.50,
  xlabel={shell factor $c$},ylabel={relative lag $\alpha=\|\mathbf d\|/R$},
  xtick={1,1.2,1.4,1.55},
  ytick={0,.5,1,1.4142136},yticklabels={$0$,$0.5$,$1$,$\sqrt2$},
  tick label style={font=\scriptsize},
  label style={font=\footnotesize},
  legend style={at={(0,-.27)},anchor=north west,draw=none,
                fill=none,font=\scriptsize,row sep=1pt},
  legend cell align=left,axis on top,clip=true
]
\path[name path=lenslower] (axis cs:1,0)--(axis cs:1.56,0);
\path[name path=lensupper] (axis cs:1,1.50)--(axis cs:1.56,1.50);
\addplot[name path=lensboundary,draw=none,forget plot,
         domain=1:1.56,samples=300] {2*sqrt(1-1/x^2)};
\addplot[figblue!13,forget plot] fill between[of=lenslower and lensboundary];
\addplot[gray!18,forget plot] fill between[of=lensboundary and lensupper];
\addplot[black,line width=1.05pt,domain=1:1.56,samples=300]
  {2*sqrt(1-1/x^2)};
\addlegendentry{lens boundary $g_c(\alpha)=0$}
\addplot[figorange,only marks,mark=*,mark size=2.2pt]
  coordinates {(1.1547005,1)};
\addlegendentry{coverage: $(2/\sqrt3,1)$}
\addplot[figgreen,only marks,mark=diamond*,mark size=2.8pt]
  coordinates {(1.299038106,1.154700539)};
\addlegendentry{energy: $(3\sqrt3/4,2/\sqrt3)$}
\addplot[figpurple,only marks,mark=square*,mark size=2.2pt]
  coordinates {(1.4142136,1.4142136)};
\addlegendentry{TV: $(\sqrt2,\sqrt2)$}
\node[font=\small,text=figblue!80!black]
  at (axis cs:1.37,.48) {$g_c(\alpha)>0$};
\node[font=\small,text=gray!75!black]
  at (axis cs:1.10,1.17) {$g_c(\alpha)<0$};
\end{axis}
\end{tikzpicture}
\end{minipage}\hfill
\begin{minipage}[t]{.49\textwidth}
\centering\vspace{0pt}
\begin{tikzpicture}
\begin{axis}[
  width=\linewidth,height=6.5cm,
  xmin=1,xmax=1.55,ymin=-.30,ymax=1.10,
  xlabel={shell factor $c$},ylabel={energy exponent},
  xtick={1,1.2990381,1.55},
  xticklabels={$1$,$3\sqrt3/4$,$1.55$},
  ytick={0,.5,1},
  tick label style={font=\scriptsize},
  label style={font=\footnotesize},
  legend style={at={(0,-.27)},anchor=north west,draw=none,
                fill=none,font=\scriptsize,row sep=1pt},
  legend cell align=left,samples=160,clip=true
]
\addplot[figorange,densely dashed,line width=1.4pt,domain=1:1.55]
  {2*ln(x)};
\addlegendentry{diagonal: $2\log c$}
\addplot[figblue,line width=1.4pt,domain=1:1.55]
  {3*ln(x)+ln(4/(3*sqrt(3)))};
\addlegendentry{squared-lens exponent}
\addplot[figgreen,only marks,mark=diamond*,mark size=2.8pt]
  coordinates {(1.299038106,.523248144)};
\addlegendentry{energy threshold: $c=3\sqrt3/4$}
\end{axis}
\end{tikzpicture}
\end{minipage}
\caption[Geometric calculations behind the thresholds.]{Geometric
calculations behind the thresholds. Left: with
$\alpha=\|\mathbf d\|/R$, the black curve is $g_c(\alpha)=0$. The lens
volume is exponentially large in the blue region and exponentially small
in the gray region. The circle and square mark the coverage and TV
thresholds on this curve. The green diamond marks the energy threshold at
$(c,\alpha)=(3\sqrt3/4,2/\sqrt3)$, below the black curve. Its vertical
coordinate is the relative length contributing most to the squared-lens
integral. Right: the same diamond marks where the squared-lens and forced
diagonal contributions have equal exponents. The energy growth exponent is
the larger of the two curves, which meet at $c=3\sqrt3/4$. The estimates in
the text justify these volume predictions for a sampled lattice.}
\label{fig:lens-exponent-map}
\end{figure}
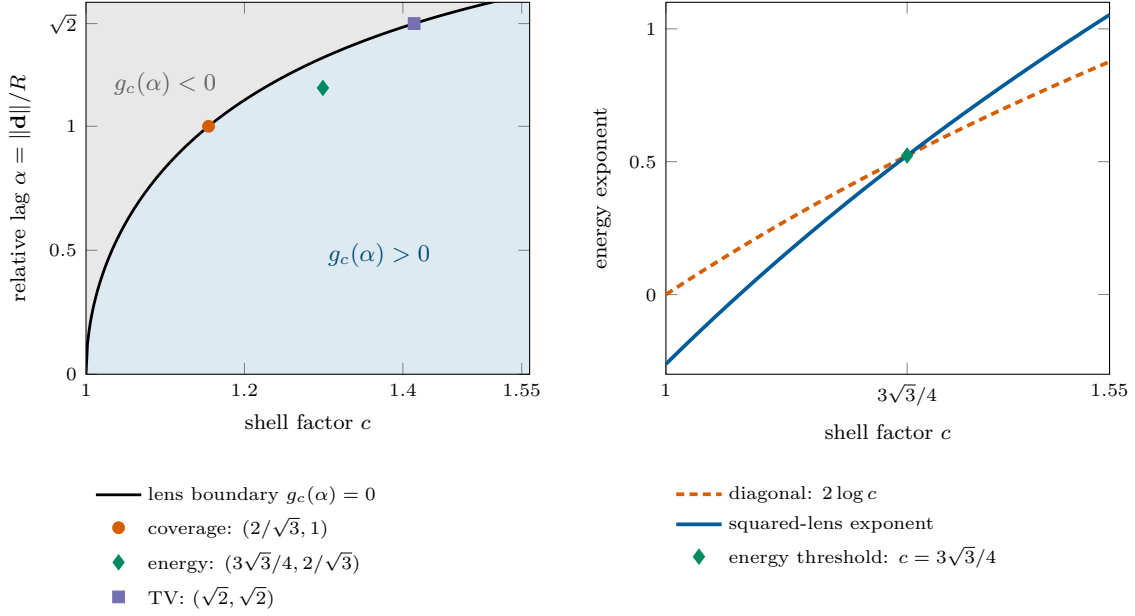

\subsection{One shell and its difference scales}

The simplest result uses the same radius for the endpoints and the target
differences.

All these results start from the same pair counts. Knowing which counts
are positive gives coverage. Summing their squares gives additive energy
and the related collision entropy. Normalizing them gives the difference
distribution used in the total-variation result. If we fix one endpoint
instead, we count its partners in the original or translated lattice.
These statistics describe different aspects of the same pair counts; a
result for one does not by itself prove a result for another.

The parent-count result says more than coverage alone: each target has
approximately the number of representations predicted by its spherical
lens. Because $A_R=-A_R$, the same count also measures the number of
partners available to each centered shell point. This relation between a
shell and its translates is the shell \emph{autocorrelation} studied here.

We can also study differences of other lengths. For $\alpha\in(0,2)$, let
$S_{\alpha R}$ be the Euclidean shell of the same relative width at radius
$\alpha R$. We count differences $\mathbf d\in\cL\cap S_{\alpha R}$ while
both endpoints $\mathbf x,\mathbf y$ stay in the single shell $A_R$.
Thus $S_{\alpha R}$ selects a range of difference lengths; it does not
change the endpoint shell or the lattice. The general estimates control
the representation count of every such difference. They also control every
endpoint's number of partners whose difference lies in that range, under
the stated restrictions on the radius ratios.

Knowing the case $\alpha=1$ does not settle the other cases. For each fixed
$\alpha\ne1$, the two thin shells are disjoint for all sufficiently large
$n$. More generally, a shell of smaller radius is not contained in one of
larger radius. Same-shell coverage therefore does not tell us how
representations are distributed within a given shorter band. This matters
close to the coverage threshold. Choose the endpoint radius just above
$(2/\sqrt3)\Rstar$ and look at a slightly shorter difference band, with
both relative changes on the scale $1/\log n$. Suitable choices bring the
band down to $(2/\sqrt3+o(1))\Rstar$ while leaving enough parents for a
bound that holds for all endpoints and targets at once. This follows from
the same theory for a fixed shell. The argument also covers all nonzero
lattice points inside the resulting critical ball; see
Corollary~\ref{cor:critical-ball-coverage}.

There is also an extension to a randomly shifted lattice, or affine coset.
Choose a shift $\boldsymbol\tau$ independently and uniformly modulo
$\cL$. One endpoint then lies in the translated lattice
$\boldsymbol\tau+\cL$ and the other in $\cL$. Within the stated ranges,
the same lens prediction controls the counts for all endpoints and target
differences in the translated lattice. Its point nearest the origin lies
near $\Rstar$. A further result gives many representations of that point
as a difference, with one parent in each lattice. These results give
geometric estimates for closest vector problems with a random target.
Probability is over both the random lattice and the independent random
shift. We do not claim uniformity over shifts chosen after seeing the
lattice. See Proposition~\ref{thm:affine-pointwise},
Proposition~\ref{prop:affine-minimum}, and
Lemma~\ref{lem:affine-terminal-lens}.

The main difficulty is to control exponentially many targets at once.
The marked-moment estimates in Sections~\ref{sec:rogers-engine}
and~\ref{sec:shell-moments} sum over these targets before applying Markov's
inequality. A further union bound is needed only when combining a given
family of different queries. The bounds on geometric and arithmetic
correction terms hold uniformly over the required ranges. This allows the
moment order to grow when we study shells close to the threshold.

\subsection{Directions and angular tests}

The main contribution concerns the pair counts of a complete shell.
The results on angular isometry and tensor moments describe the shell
directions. Some already follow from lower moments.

For each shell point $\mathbf x$, write
$\widehat{\mathbf x}:=\mathbf x/\|\mathbf x\|$ for its unit direction.
For a test vector $\mathbf u\in\R^n$, the scalar product
$\langle\mathbf u,\widehat{\mathbf x}\rangle$ is its projection onto that
direction. The prediction from a uniform sphere for the average squared
projection is $\|\mathbf u\|^2/n$.

The directional balance in the informal main result means that
\[
 \frac1{N_R}\sum_{\mathbf x\in A_R}
   \langle\mathbf u,\widehat{\mathbf x}\rangle^2
 =(1+o(1))\frac{\|\mathbf u\|^2}{n}
 \qquad\text{for every }\mathbf u\in\R^n.
\]
The relative error tends to zero exponentially fast, with the same bound
for all test vectors. At $c=2/\sqrt3$, higher even directional moments also
match those of a uniform spherical direction. This holds for all even
degrees up to any fixed fraction of the dimension smaller than
$0.080964856\ldots$, with error measured in relative Hilbert--Schmidt norm.

To explain these higher moments, fix a degree $k$ and average every product
of $k$ coordinates of $\widehat{\mathbf x}$ over the shell. Together, these
averages form the empirical $k$-th moment tensor. We compare it with the
tensor obtained from one uniform random direction on the sphere. The
Hilbert--Schmidt error is the square root of the sum of squared differences
between corresponding entries, divided by the norm of the comparison
tensor.

The shell therefore has no preferred direction at the scale measured by
these averages.

We also study bounded nonnegative functions of the angle between two shell
points. We compute their pair sums averaged over lattices, including the
contribution from collinear pairs. We then bound the probability that such
a weighted pair count is unusually large for a sampled lattice. This is
an upper-tail bound; it does not give a matching lower-tail bound or
two-sided concentration. See
Proposition~\ref{thm:weighted-angular-transfer} and
Corollary~\ref{cor:quenched-weighted-kernel-transfer}.
Corollary~\ref{cor:common-cap-statistics} applies this bound to pairs lying
in spherical caps with a common random direction.

\subsection{Contribution and relation to earlier work}
\label{sec:novelty-comparison}

The main contribution is to control the pair counts of one complete
lattice shell with exponentially many points, for all relevant targets at
once. The angular operator and tensor estimates give additional results
for averages of the shell directions. The proofs use the classical
mean-value formulas of Siegel and Rogers \cite{Siegel1945,Rogers1955};
they do not assume that shell points are independent.

Earlier work on high-dimensional random lattices includes Poisson limits
for short vectors, limits for their angles, and moment results for larger
sets \cite{Sodergren2011,SodergrenAngles2011,StrombergssonSodergren2016}.
Our results differ in the statistics studied, the size of the sets, and
the requirement to control all targets at once. We compare each result
with earlier work in its technical section. We also explain which
classical results the moment method uses and where the Euclidean energy
constant comes from.

\subsection{Scope and roadmap}

All statements concern real Haar-random lattices of determinant one,
either centered or with an independent uniform random shift. Whenever we
compare a lattice statistic with a geometric prediction, we specify both
the statistic and the prediction. We make no claim of joint independence
for complete shells or of a full point-process limit. We also do not claim
uniform bounds for shifts chosen after seeing the lattice. Structured
modular or module lattices and finite-precision algorithms are outside the
scope of this paper.

Sections~\ref{sec:setup} and~\ref{sec:comparison-framework} introduce the
model and the geometric predictions. Section~\ref{sec:rogers-engine}
bounds the correction terms in Rogers' formula.
Section~\ref{sec:shell-moments} combines these bounds into the uniform
marked-moment estimates. Section~\ref{sec:growing-rogers} uses these
moments to control all relevant counts for a centered or translated lattice
at once. Sections~\ref{sec:additive-phases}--\ref{sec:convolution} then
treat coverage, additive energy, and total variation, in that order. The
second-moment estimates shared by the energy and total-variation results
are proved before those results. Section~\ref{sec:angular-transfer}
treats angular isometry and tensor moments using moments of order two.
Appendix~\ref{app:angular-kernels} gives the identities and upper-tail
bounds for bounded angular functions. Table~\ref{tab:notation-index} at the
end lists the main notation. The central results appear at the start of
their sections, before the supporting estimates and proofs.

\subsection{Lean formalization}\label{sec:lean-formalization}

All the main theorems of this paper have been formally proved in Lean~4,
under the classical results and probability assumptions stated explicitly
in the code. These assumptions include the Siegel and Rogers mean-value
identities and the spherical Gram-density formula. The full development
compiles and passes its axiom checks. It uses neither proof placeholders
(\texttt{sorry}) nor project-specific axioms.

The Lean source code is available for independent checking,
together with the required dependency versions, build instructions, and a
list of the formalized statements and their assumptions.

\subsection*{Use of AI}

Most of the content of this paper was generated by ChatGPT~5.6 Sol and
ChatGPT~6 Astra. The Lean formalization was produced entirely by AI.
The author is responsible for the final manuscript and for interpreting
the formal verification results.

\section{Setup and continuum-comparison observables}\label{sec:setup}

Let
\[
 X_n:=\mathrm{SL}_n(\R)/\mathrm{SL}_n(\Z)
\]
be the space of unimodular lattices with its Haar probability measure.  We
write $\cL\sim X_n$.  All logarithms are natural, and every asymptotic
statement is as $n\to\infty$.  Unless a theorem states otherwise, named
constants and compact parameter sets are chosen before this limit.
Subscripts on $O_J(\cdot)$ indicate the permitted parameter dependence;
unsubscripted $O(\cdot)$ and $o(\cdot)$ are uniform over the ranges quantified
in the same statement.  The notation $e^{-\Omega_c(n)}$ means a bound
$e^{-a_cn}$ for some $a_c>0$ and all sufficiently large $n$; different
occurrences may use different constants.  Here $c$ is fixed away from
any boundary excluded by the relevant theorem; shrinking windows require
the separate moment--buffer estimates of
Corollary~\ref{cor:moment-buffer}.  Likewise,
$f_n=e^{-\Theta_c(n)}$ means that
$e^{-b_cn}\le f_n\le e^{-a_cn}$ for some $0<a_c\le b_c<\infty$.
For positive sequences, $f_n=\Theta(g_n)$ means that $f_n/g_n$ is bounded
above and below by positive constants.
The random affine extension of $X_n$ is the probability space
\[
 \widetilde X_n:=
 \{(\cL,\boldsymbol\tau):\cL\in X_n,
       \ \boldsymbol\tau\in\R^n/\cL\},
\]
where first $\cL$ is sampled from $X_n$ and then, conditionally on $\cL$,
$\boldsymbol\tau$ is sampled from Haar probability measure on the torus
$\R^n/\cL$.  We identify the second coordinate with the affine lattice
$\boldsymbol\tau+\cL$; all statements are independent of the representative
chosen for $\boldsymbol\tau$.  The closed Euclidean ball of radius $r$ is
$B_r$, and
\[
 \kappa_n:=\vol(B_1),\qquad \Rstar:=\kappa_n^{-1/n},
 \qquad \vol(B_{t^{1/n}\Rstar})=t.
\]
Stirling's formula gives
\begin{equation}\label{eq:Rstar-stirling}
 (\Rstar)^2=\frac{n}{2\pi e}\exp(O(\log n/n)).
\end{equation}
For a Borel function $f$ vanishing at the origin, its Siegel transform is
$\widehat f(\cL)=\sum_{\mathbf x\in\cL\setminus\{0\}}f(\mathbf x)$.
We use thin shells of relative half-width $w_n=n^{-2}$,
\[
 S_R:=\{\mathbf x:R(1-w_n)\le\|\mathbf x\|\le R(1+w_n)\},
\]
and the representation function
\begin{equation}\label{eq:representation-function}
 r_{\cL,R}(\mathbf d):=
 \#\{\mathbf x\in\cL\cap S_R:\mathbf x-\mathbf d\in\cL\cap S_R\}.
\end{equation}

\begin{proposition}[Siegel mean-value formula]\label{thm:siegel}
For $n\ge2$ and every nonnegative integrable $f:\R^n\to\R$,
\[
 \E_{\cL}\widehat f(\cL)=\int_{\R^n}f(\mathbf x)\,d\mathbf x.
\]
\end{proposition}

\section{The continuum-comparison framework}
\label{sec:comparison-framework}

The central question becomes precise only after choosing a common state space
and a mode of comparison.  Let
\[
 A_R(\cL):=\cL\cap S_R,\qquad
 N_R:=|A_R(\cL)|,\qquad P_R:=\vol(S_R).
\]
Whenever $N_R>0$, define the empirical shell and directional probability
measures
\begin{equation}\label{eq:empirical-shell-measures}
 U_{\cL,R}:=\frac1{N_R}\sum_{\mathbf x\in A_R(\cL)}
              \delta_{\mathbf x},
 \qquad
 \nu_{\cL,R}:=\frac1{N_R}\sum_{\mathbf x\in A_R(\cL)}
              \delta_{\mathbf x/\|\mathbf x\|}.
\end{equation}
On the empty-shell event $N_R=0$, set $U_{\cL,R}=\delta_0$
and $\nu_{\cL,R}=\delta_{\mathbf e_1}$, where $\mathbf e_1$ is the
first coordinate unit vector.  All displayed quotients involving $N_R$
are interpreted on $N_R>0$.  Thus the difference law is always a defined
probability measure.  Their continuum counterparts are normalized Lebesgue measure
$\sigma_R=P_R^{-1}\mathbf1_{S_R}(\mathbf x)d\mathbf x$ and normalized
spherical measure $\sigma_{n-1}$.  At normalized radius $c=R/\Rstar$,
we also write $U_{\cL,c}:=U_{\cL,R}$.

For any probability measure $\mu$ on $\R^n$, let $\widetilde\mu$ denote its
reflection, defined by $\widetilde\mu(E):=\mu(-E)$.  We use the normalization
\[
 \TV(\mu,\nu):=\sup_{E\ \mathrm{Borel}}|\mu(E)-\nu(E)|;
\]
on a common countable support this equals
$\tfrac12\sum_x|\mu(\{x\})-\nu(\{x\})|$.

The difference of two independent shell points has empirical law
\begin{equation}\label{eq:empirical-difference-law}
 \Gamma_{\cL,R}:=U_{\cL,R}*\widetilde U_{\cL,R},
 \qquad
 \Gamma_{\cL,R}(\{\mathbf d\})
 =\frac{r_{\cL,R}(\mathbf d)}{N_R^2}.
\end{equation}
The continuous difference law $\sigma_R*\widetilde\sigma_R$ has density
\[
 \frac{K_R(\mathbf d)}{P_R^2},\qquad
 K_R(\mathbf d):=\vol\bigl(S_R\cap(S_R+\mathbf d)\bigr).
\]
Since $\Gamma_{\cL,R}$ is atomic and the continuous law is
absolutely continuous, their total-variation distance is identically one.
The meaningful lattice-supported continuum comparator is therefore
\begin{equation}\label{eq:framework-sampled-comparator}
 \Pi_{\cL,R}(\mathbf d)
 :=\frac{K_R(\mathbf d)}{Z_{\cL,R}},
 \qquad
 Z_{\cL,R}:=\sum_{\mathbf z\in\cL}K_R(\mathbf z),
 \qquad \mathbf d\in\cL.
\end{equation}
Siegel's formula predicts $Z_{\cL,R}\approx P_R^2$; the quenched
normalization and the comparison
$\TV(\Gamma_{\cL,R},\Pi_{\cL,R})$ are proved later.

This leads to three distinct questions.
\begin{enumerate}[leftmargin=*]
\item \emph{Pointwise incidence comparison:} for how many, or for all, target
differences $\mathbf d$ is $r_{\cL,R}(\mathbf d)$ close to the lens volume
$K_R(\mathbf d)$?
\item \emph{Normed autocorrelation comparison:} when do support,
$\ell^2$ collision statistics, and total variation agree with their
continuum predictions?
\item \emph{Marginal angular comparison:} for which scalar, matrix, tensor,
or kernel tests does $\nu_{\cL,R}$ act like $\sigma_{n-1}$?
\end{enumerate}
The theorems answer substantial parts of each question, but they do not
collapse the three notions into one.  In particular, angular moment agreement
does not imply pointwise lens regularity, and support coverage does not imply
total-variation convergence.

For the affine space, replace the first measure in
\eqref{eq:empirical-shell-measures} by the empirical measure on
$(\boldsymbol\tau+\cL)\cap S_R$.  Differences of two points in the same
affine coset are centered lattice vectors, whereas a mixed affine--centered
difference remains in the affine coset.  This is why the marked theory has
both an affine vertex and centered increments.  Averaging over the
independent Haar mark removes one affine column exactly, but it does not
yield a theorem uniform over every shift of a fixed lattice.

\section{The growing-order Rogers moment method}\label{sec:rogers-engine}
This section isolates the growing-order Rogers machinery used by the marked
moment calculations.  It records the exact finite-dimensional formula, the
spherical and Poisson tools, and exhaustive bounds for partition, geometric
nonpartition, and arithmetic/large-coefficient diagrams.  Throughout
Sections~\ref{sec:rogers-engine}--\ref{sec:shell-moments},
write $\ell:=\lceil\log n\rceil$ and let $h$ be the positive integer
half-order of a centered moment.  Uniform assertions allow
$1\le h\le H\log n$ for any fixed $H>0$; the near-critical application
in Section~\ref{sec:growing-rogers} specializes to $h=\ell$.

The mean-value formulas of Siegel and Rogers
\cite{Siegel1945,Rogers1955,RogersMoments1955,Rogers1956}, the Wishart/Gram
density, and the Poisson cumulant identities are classical.
Han~\cite[Propositions~2.2--2.4]{Han2024} records the arithmetic/height split,
a geometric remainder, and the partition terms for product indicators.
These estimates already have explicit order dependence; merely taking
$k=O(\log n)$ is not the new step.
Here the kernels couple a mark to thin lens fibers.
Proposition~\ref{prop:uniform-diagram-budget} retains their exact weighted
fiber powers while summing all nonpartition matrices uniformly for
$k\le2H\log n+1$.  The shell-specific additions are the four scalar support
relations, the first nonpartition column under a product angular measure,
and the separate denominator and height sums.  The $16/27$ and $e/4$
estimates are proved below for these kernels; the factor $q^{-n}$ is inherited
from Rogers' coefficient.

The proof starts with the exact formula in
Proposition~\ref{thm:rogers-general}, the four kernels and their scalar
certificates in Lemmas~\ref{lem:four-raw-kernels}--%
\ref{lem:four-kernel-relation-certificate}, and the first nonpartition
column in Lemma~\ref{fs-lem:first-offending}.  The geometric estimate and
the separate denominator/height sums then enter the complete classification
of Lemma~\ref{lem:complete-diagram-audit} and
Proposition~\ref{prop:uniform-diagram-budget}.  Section~\ref{sec:shell-moments}
evaluates the partition terms and applies this common correction bound.

\subsection{The exact Rogers interface}
\begin{proposition}[Rogers' multiple mean-value formula in the Han normalization]\label{thm:rogers-general}
Let $1\le k\le n-1$ and let $F:(\R^n)^k\to\R_{\ge0}$ be bounded, Borel measurable, and compactly supported, with $F=0$ whenever any coordinate is zero.  For $1\le r\le k$ and $q\ge1$, let $\mathcal D^k_{r,q}$ be the set of integral $r\times k$ matrices $D$ for which there are pivot columns
\[
 1=j_1<j_2<\cdots<j_r\le k
\]
with $j_1=1$ as part of Han's admissibility convention, such that,
writing $[D]^j$ for the $j$-th column of $D$,
\begin{enumerate}[label=(\roman*),leftmargin=*]
\item every column of $D$ is nonzero;
\item $([D]^{j_1},\ldots,[D]^{j_r})=qI_r$;
\item $D_{ij}=0$ whenever $j<j_i$;
\item the gcd of all entries of $D$ is $1$.
\end{enumerate}
Then
\begin{align}
 &\E_{\cL}\sum_{\mathbf{v}_1,\ldots,\mathbf{v}_k\in\cL\setminus\{0\}}
 F(\mathbf{v}_1,\ldots,\mathbf{v}_k)\notag\\
 &\qquad=\sum_{r=1}^k\sum_{q\ge1}\sum_{D\in\mathcal D^k_{r,q}}
 c_D\int_{(\R^n)^r}
 F\!\left(\frac1q(\mathbf{v}_1,\ldots,\mathbf{v}_r)D\right)
 \,d\mathbf{v}_1\cdots d\mathbf{v}_r,\label{eq:rogers-general}
\end{align}
where
\begin{equation}\label{eq:rogers-coeff}
 c_D=
 \frac{\#\{\mathbf{a}\in\{0,1,\ldots,q-1\}^r:\mathbf{a}D/q\in\Z^k\}^{\,n}}
 {q^{nr}}.
\end{equation}
Moreover,
\begin{equation}\label{eq:rogers-denom}
 c_D=1\quad(q=1),\qquad c_D\le q^{-n}\quad(q\ge2).
\end{equation}
\end{proposition}
\begin{proof}
Rogers' original general formula is Theorem~4 of~\cite[Theorem~4, eqs.~(9)--(11), pp.~251--252]{Rogers1955}.  We use exactly the row-echelon/admissible-matrix normalization stated by Han~\cite[Section~2, Theorem~2.1, Eq.~(2.2)]{Han2024}: Han's four defining conditions are the four conditions above, including $1=j_1$.  The coefficient in Han's Eq.~(2.2) is exactly~\eqref{eq:rogers-coeff}, and the sentence immediately following that equation records $c_D\le q^{-n}$.  Since all kernels to which we apply~\eqref{eq:rogers-general} are nonnegative raw-moment kernels, no signed interchange is required; centering is performed only after the raw terms are evaluated.  For general convergence issues see Schmidt~\cite{Schmidt1958}.

\end{proof}

\begin{lemma}[Measurability of the concrete lattice kernels]
\label{lem:haar-kernel-measurable}
Let $F:(\R^n)^k\to[0,\infty]$ be Borel measurable.  Then its Siegel
transform
\[
 \widehat F(\cL):=
 \sum_{\mathbf v_1,\ldots,\mathbf v_k\in\cL\setminus\{0\}}
 F(\mathbf v_1,\ldots,\mathbf v_k)
\]
is measurable on $X_n$.  The same is true on $\widetilde X_n$ when finitely
many coordinates are required to lie in the affine coset, and for every
finite product, finite sum, and nonnegative power of the shell indicators
used below.  Consequently all raw marked moments in
Sections~\ref{sec:shell-moments}--\ref{sec:growing-rogers} are legitimate
nonnegative measurable integrands and Tonelli may be applied before any
centering operation.
\end{lemma}
\begin{proof}
Pull the transform back to $\mathrm{SL}_n(\R)$.  There it is the countable sum
\[
 \sum_{\mathbf m_1,\ldots,\mathbf m_k\in\Z^n\setminus\{0\}}
 F(g\mathbf m_1,\ldots,g\mathbf m_k)
\]
of nonnegative Borel functions of $g$.  It is therefore Borel and is invariant
under the right action of $\mathrm{SL}_n(\Z)$, so it descends to the quotient
Borel structure on $X_n$.  On the affine space, use a Borel fundamental domain
and write every affine point as $\boldsymbol\tau+g\mathbf m$; the same
countable-sum argument applies.  Products and finite sums of the particular
indicator kernels remain Borel, and an integer power is a finite product.
The zero-coordinate convention in
Proposition~\ref{thm:rogers-general} is enforced by inserting the corresponding
indicators.  All kernels are nonnegative, so Tonelli requires no prior
integrability assertion.
\end{proof}

\begin{lemma}[Growing order is within the exact Rogers range]\label{lem:growing-order-legal}
For every fixed $H>0$, all sufficiently large $n$, and integers
$1\le h\le H\log n$, every tuple order $k\le2h+1$ satisfies
\[
 k\le n-1,\qquad
 n\ge\left\lfloor\frac{k^2}{4}\right\rfloor+3.
\]
The first inequality is the validity condition for the exact formula in
Proposition~\ref{thm:rogers-general}.  The second is the sufficient range
for the product-indicator remainder estimate recorded in
\cite[Proposition~2.2]{Han2024}; it is not an additional hypothesis of the
identity.  In particular the specialization $h=\ell$ is admissible.
\end{lemma}
\begin{proof}
Uniformly in the stated range, $k=O_H(\log n)=o(n)$ and
$k^2=O_H(\log^2n)=o(n)$, proving both inequalities.
All Rogers orders in the marked calculations are at most $2h+1$.
Thus the exact finite-dimensional identity applies separately at each $n$;
the shell-specific remainder bounds are proved below.
\end{proof}

\begin{lemma}[Full rank and the first marked pivot]\label{lem:rogers-pivot-facts}
In the normalization of Proposition~\ref{thm:rogers-general}, column $1$ is always the first pivot.  If $r=k$, then $D=qI_k$ and admissibility forces $q=1$.  Thus every nonidentity Rogers diagram has rank at most $k-1$.
\end{lemma}
\begin{proof}
The equality $j_1=1$ is part of Han's admissibility definition~\cite[Theorem~2.1]{Han2024}.  If $r=k$, every column is a pivot, so condition (ii) gives $D=qI_k$.  Condition (iv) says the gcd of all entries is one; hence $q=1$.
\end{proof}

For later reference, for $D\in\mathcal D^k_{r,q}$ write
\begin{equation}\label{eq:diagram-integral-definition}
 \mathcal I_D(F):=c_D\int_{(\R^n)^r}
 F\!\left(q^{-1}(\mathbf v_1,\ldots,\mathbf v_r)D\right)
 \,d\mathbf v_1\cdots d\mathbf v_r.
\end{equation}
Thus the right-hand side of~\eqref{eq:rogers-general} is literally the sum of
the quantities $\mathcal I_D(F)$; it is not an abstract or independently
specified remainder.

\begin{lemma}[Supported Han columns are the actual shell columns]
\label{lem:supported-han-kernel}
Put
\[
 \mathbf y_j=q^{-1}\sum_{i=1}^rD_{ij}\mathbf v_i
 \qquad(1\le j\le k).
\]
Suppose that the integrand in~\eqref{eq:diagram-integral-definition} is one
of the target or endpoint raw-moment kernels below and is nonzero at
$(\mathbf v_1,\ldots,\mathbf v_r)$.  Then every $\mathbf y_j$ satisfies
exactly the shell indicators assigned to column $j$.  If column $1$ is the
mark, Han admissibility gives $\mathbf y_1=\mathbf v_1$ and every nonpivot
column has the literal decomposition
\begin{equation}\label{eq:actual-han-column}
 \mathbf y_j=\frac{D_{1j}}q\mathbf y_1+
              \sum_{i=2}^r\frac{D_{ij}}q\mathbf v_i.
\end{equation}
For $q=1$ and a centered marked column, write $A=D_{1j}$ and
$B_i=D_{ij}$ for the unmarked coefficients.  Uniformly for
$\alpha\in J\Subset(0,2)$, the supported shell constraints imply
\begin{align*}
 \left|2A+\sum_{i=2}^rB_i-1\right|
 &\le C_J\left(1+\sum_{i=2}^r|B_i|\right)w_n
 &&\text{(target)},\\
 \left|2A+\alpha^2\left(\sum_{i=2}^rB_i-1\right)\right|
 &\le C_J\left(1+\sum_{i=2}^r|B_i|\right)w_n
 &&\text{(endpoint)}.
\end{align*}
In the height-one class $|B_i|\le1$, these errors are $O_J(kw_n)$.
The integer rounding in Lemmas~\ref{fs-lem:target-relations},
\ref{fs-lem:endpoint-nonresonance}, and
\ref{lem:generic-endpoint-low-parent} is used only in that class and under
the endpoint nonresonance conditions stated there.  For arbitrary heights
we instead retain the exact transverse coefficients in
Lemma~\ref{lem:uniform-scalar-certificate} and use
Lemma~\ref{fs-lem:arithmetic-direct}.
Finally, any remaining shell indicator is a factor in $[0,1]$, so deleting
it can only increase the integral.
\end{lemma}
\begin{proof}
The first assertion is the definition of the product indicator $F$.  Since
the first pivot column of $D$ equals $q\mathbf e_1$, one has
$\mathbf y_1=\mathbf v_1$, and matrix multiplication gives
\eqref{eq:actual-han-column}.  The identity
\[
 \|\mathbf z\|^2-\|\mathbf z-\mathbf m\|^2
 =2\langle\mathbf z,\mathbf m\rangle-\|\mathbf m\|^2
\]
shows that the normalized longitudinal coordinate of each parent and the
output is $\lambda+\epsilon_i$, where $|\epsilon_i|\le C_Jw_n$ and
$\lambda=1$ in the target case, $\lambda=\alpha^2$ in the endpoint case.
Here the coordinate is twice the inner product with the mark divided by
its squared norm.  Linearity therefore gives
\[
 2A+\lambda\left(\sum_{i=2}^rB_i-1\right)
 =\epsilon_{\rm out}-\sum_{i=2}^rB_i\epsilon_i.
\]
This proves both inequalities with their coefficient dependence intact.
The height-one specialization uses $\sum_i|B_i|\le k$.
Pointwise monotonicity of the product of nonnegative indicators proves the
last assertion.
\end{proof}

\subsection{The spherical Gram density used below}
The only continuous high-dimensional distribution formula imported into the angular estimates is the following standard consequence of the Wishart density.

\begin{lemma}[Gram matrix of independent spherical directions]\label{lem:gram-density}
Let $2\le s\le d-2$ and let $\mathbf{u}_1,\ldots,\mathbf{u}_s$ be independent uniform vectors on $S^{d-1}$.  Their Gram matrix
\[
 G=(\langle\mathbf{u}_i,\mathbf{u}_j\rangle)_{i,j\le s}
\]
has, on the positive-definite correlation-matrix cone $\mathcal E_s^\circ=\{G\succ0:G_{ii}=1\}$, density with respect to Lebesgue measure in the $s(s-1)/2$ off-diagonal coordinates proportional to
\begin{equation}\label{eq:gram-density}
 (\det G)^{(d-s-1)/2}.
\end{equation}
For $s=O(\log n)$ and $n-O(\log n)\le d\le n$, the ratio between the total Euclidean volume of the correlation-matrix domain and the normalizing integral in~\eqref{eq:gram-density} is at most
\begin{equation}\label{eq:gram-normalizer}
 \exp(O(s^2\log n)).
\end{equation}
\end{lemma}
\begin{proof}[Reference and short derivation]
Let $Z$ be a $d\times s$ matrix with independent $N(0,1)$ entries.  Then
$W=Z^{\mathsf T}Z$ is Wishart and has density proportional to
\[
 (\det W)^{(d-s-1)/2}e^{-\operatorname{tr}(W)/2};
\]
see Muirhead~\cite[Section~3.2.1, p.~85]{Muirhead1982}.  Writing $W=DGD$ with $D=\operatorname{diag}(\|Z_1\|,\ldots,\|Z_s\|)$, where $Z_i$ is the $i$-th column of $Z$, separates radial variables from the normalized columns $\mathbf{u}_i=Z_i/\|Z_i\|$ and yields~\eqref{eq:gram-density}.  For~\eqref{eq:gram-normalizer}, put $m=s(s-1)/2$.  The correlation cone lies in $[-1,1]^m$.  The box
\[
 \mathcal B=\{G:|G_{ij}|\le(10sd)^{-1}\ (i<j),\ G_{ii}=1\}
\]
has volume $(5sd)^{-m}$ and lies in the positive-definite cone by
Gershgorin.  Every eigenvalue on $\mathcal B$ is at least
$1-(s-1)/(10sd)$, so
\[
 \inf_{G\in\mathcal B}(\det G)^{(d-s-1)/2}
 \ge \exp(-C s)
\]
in the asserted asymptotic range, where $s\le d-2$ for all sufficiently
large $n$.  In the stated range the exponent is nonnegative and
$0<\det G\le1$, so the normalizing integral is finite on the bounded
domain; the box also shows that it is positive.
Thus the normalizing integral is at least
$(5sd)^{-m}e^{-Cs}$, whereas the Euclidean volume of the whole domain is at
most $2^m$.  Their ratio is at most
$\exp(O(s^2\log(sd)))=\exp(O(s^2\log n))$, with a constant uniform for the
range used below.
\end{proof}

\subsection{Poisson cumulants and centered partition identities}
\begin{lemma}[Poisson cumulants and the no-singleton centered-moment formula]\label{lem:poisson-cumulants}
Let $\Pi$ be a Poisson random measure of intensity $\nu$ on a measurable
space and let $f$ be bounded and measurable with support of finite $\nu$-measure.  Put
$Z=\int f\,d\Pi$, let $\kappa_b(Z)$ denote its $b$-th cumulant, write
$[m]:=\{1,\ldots,m\}$, and let $\mathcal P([m])$ be the set of partitions
of $[m]$.  Then, for every $b\ge1$,
\begin{equation}\label{eq:poisson-cumulants}
 \kappa_b(Z)=\int f^b\,d\nu.
\end{equation}
For any random variable with finite moments through order $m$,
\begin{equation}\label{eq:centered-cumulant-partitions}
 \E(Z-\E Z)^m
 =\sum_{\substack{\pi\in\mathcal P([m])\\ |B|\ge2\ \forall B\in\pi}}
 \prod_{B\in\pi}\kappa_{|B|}(Z).
\end{equation}
In particular every partition contributing to a centered $2h$-th moment has at most $h$ blocks, and the number of such partitions is at most $(2h)^{2h}$.
\end{lemma}
\begin{proof}
The Poisson exponential formula gives
\[
 \log\E e^{tZ}=\int(e^{tf}-1)\,d\nu;
\]
differentiating at $t=0$ yields~\eqref{eq:poisson-cumulants}.  The ordinary moment--cumulant formula is
\[
 \E Z^m=\sum_{\pi\in\mathcal P([m])}\prod_{B\in\pi}\kappa_{|B|}(Z).
\]
Replacing $Z$ by $Z-\E Z$ sets the first cumulant to zero and leaves every cumulant of order at least two unchanged, proving~\eqref{eq:centered-cumulant-partitions}.  With no singleton blocks a partition of $2h$ points has at most $h$ blocks, and the crude bound $|\mathcal P([2h])|\le(2h)^{2h}$ is immediate by assigning to each element a block label in $[2h]$.
\end{proof}

\subsection{Thin shells, fiber disintegration, and a two-vector shell count}
For the near-unit calculations in this section and
Section~\ref{sec:shell-moments}, use the following specialization of the
fixed-shell autocorrelation problem.  Set
\[
 w_n:=n^{-2},\qquad
 \gamma:=1-\frac1\ell=\frac{\ell-1}{\ell},\qquad
 \mathcal N_\star:=\left(\frac43\right)^{n/2},
\]
and define the endpoint shell and the observed difference band by
\[
 S_R:=\{\mathbf{x}:R(1-w_n)\le\|\mathbf{x}\|\le R(1+w_n)\},
 \qquad T_R:=S_{\gamma R}.
\]
Both endpoints in every representation remain in the same shell $S_R$.
The set $T_R$ merely selects which radial part of the difference set
$S_R-S_R$ is being tested.  The parameter-free formulation with
$T_{\alpha,R}=S_{\alpha R}$, including the intrinsic same-shell case
$\alpha=1$, is recorded in Proposition~\ref{thm:intrinsic-marked-moments}.
For every $\mathbf d\in\R^n$, write
\[
 \mu_R(\mathbf d):=K_R(\mathbf d)
 =\vol(S_R\cap(S_R+\mathbf d)).
\]
On $T_R$, define
\[
 \mu_*(R):=\inf_{\mathbf{d}\in T_R}\mu_R(\mathbf{d}),\quad
 \wt_R(\mathbf{d}):=\frac{\mu_*(R)}{\mu_R(\mathbf{d})}.
\]
On $T_R$ one has $0<\wt_R\le1$ and
$\wt_R(\mathbf{d})\mu_R(\mathbf{d})=\mu_*(R)$ exactly.
Every band weight in this paper is extended by zero outside its defining
band; a quotient defining it is evaluated only on that band, where the
lens is positive.

\begin{lemma}[Polar disintegration of every free fiber]\label{lem:fiber-polar}
Fix a nonzero marked vector $\mathbf{m}$ and a shell fiber occurring below, either
\[
 \{\mathbf{x}\in S_R:\mathbf{x}-\mathbf{m}\in S_R\}
 \quad\text{or}\quad
 \{\mathbf{d}\in T_R:\mathbf{m}-\mathbf{d}\in S_R\}.
\]
Write $\mathbf{e}=\mathbf{m}/\|\mathbf{m}\|$ and $\mathbf{v}=t\mathbf{e}+\rho\mathbf{u}$ with $\mathbf{u}\in S^{n-2}\subset\mathbf{e}^{\perp}$.  Membership in either fiber depends only on $(t,\rho)$, and conditional on $(t,\rho)$ the direction $\mathbf{u}$ is exactly uniform.  Distinct free Rogers pivots therefore have independent uniform transverse directions conditional on all their scalar radial/longitudinal coordinates.
\end{lemma}
\begin{proof}
Both defining shell inequalities are functions only of $\|\mathbf{v}\|^2=t^2+\rho^2$ and $\|\mathbf{v}-\mathbf{m}\|^2=(t-\|\mathbf{m}\|)^2+\rho^2$ (or the analogous endpoint formula).  Euclidean measure in these coordinates is a scalar Jacobian times $dt\,d\rho\,d\sigma(\mathbf{u})$.  Product Lebesgue measure over distinct Rogers pivot rows therefore disintegrates into a product of uniform spherical direction measures after the scalar coordinates are fixed.
\end{proof}

\begin{lemma}[Thin-shell coarea formulas]\label{lem:thin-shell-coarea}
Let $w_n=n^{-2}$ and let $\alpha$ range in a fixed compact subinterval of $(0,2)$.  Uniformly for $R>0$ and every $\mathbf{d}$ with $\|\mathbf{d}\|=\alpha R(1+O(w_n))$,
\begin{equation}\label{eq:lens-coarea}
 \log\vol\{\mathbf{x}\in S_R:\mathbf{x}-\mathbf{d}\in S_R\}
 =n\log(R/\Rstar)+\frac n2\log\!\left(1-\frac{\alpha^2}{4}\right)+O(\log n).
\end{equation}
The logarithmic derivative of this volume with respect to $\alpha$ is $O(n)$ uniformly on such a compact interval.  Likewise, for $\mathbf{x}\in S_R$, the unweighted endpoint cap
\[
 C_{\mathbf{x}}:=\{\mathbf{d}\in T_R:\mathbf{x}-\mathbf{d}\in S_R\}
\]
has logarithmic derivative $O(n)$ when $\|\mathbf{x}\|/R$ varies in $[1-w_n,1+w_n]$.
\end{lemma}
\begin{proof}
Put $a=R(1-w_n)$, $b=R(1+w_n)$, and $d=\|\mathbf d\|$.
Here $|S^{n-2}|$ denotes the unnormalized surface area of the unit sphere
in $\R^{n-1}$, equal to $(n-1)\kappa_{n-1}$.
Denote the lens volume on the left of~\eqref{eq:lens-coarea} by $K(d)$.
Use the two radial coordinates
\[
 r=\|\mathbf x\|,\qquad s=\|\mathbf x-\mathbf d\|.
\]
If $\mathbf e=\mathbf d/d$ and
$\mathbf x=t\mathbf e+\rho\mathbf u$, then
\begin{equation}\label{eq:bipolar-rho}
 t=\frac{r^2+d^2-s^2}{2d},\qquad
 \rho(r,s,d)^2
 =r^2-\left(\frac{r^2+d^2-s^2}{2d}\right)^2.
\end{equation}
The Jacobian identity
\[
 dt\,d\rho=\frac{rs}{d\rho}\,dr\,ds
\]
therefore gives the exact bipolar-coordinate formula
\begin{equation}\label{eq:exact-lens-bipolar}
 K(d)=|S^{n-2}|\int_a^b\int_a^b
       \frac{rs}{d}\rho(r,s,d)^{n-3}\,ds\,dr.
\end{equation}
For $\alpha$ in a fixed compact $J\Subset(0,2)$ and all sufficiently large
$n$, every point of the integration rectangle satisfies the strict triangle
inequalities and
\[
 \rho(r,s,d)=R\sqrt{1-\alpha^2/4}\,(1+O_J(w_n)).
\]
Together with $b-a=2w_nR$, the standard formula for $|S^{n-2}|$ and
$\kappa_n(\Rstar)^n=1$, equation~\eqref{eq:exact-lens-bipolar} gives
\eqref{eq:lens-coarea}; all powers of $n$, both shell widths, and replacing
$n-3$ by $n$ contribute only $O_J(\log n)$.

The derivative assertion is not obtained by differentiating that remainder.
Instead, differentiate the exact positive integrand in
\eqref{eq:exact-lens-bipolar}.  Compactness of $J$ and
\eqref{eq:bipolar-rho} give, uniformly on the integration rectangle,
\[
 \left|\partial_d\log\left(\frac{rs}{d}
                 \rho(r,s,d)^{n-3}\right)\right|
 \le \frac{C_Jn}{R}.
\]
Differentiation under the integral is valid on this compact nondegenerate
domain.  Dividing by $K(d)$ shows that $|\partial_d\log K(d)|\le C_Jn/R$,
and hence that the logarithmic derivative with respect to $\alpha$ is
$O_J(n)$.

For the endpoint cap, regard $\|\mathbf x\|$ as the separation variable in
  the same formula, integrate $r=\|\mathbf d\|$ over the difference band and
$s=\|\mathbf x-\mathbf d\|$ over the endpoint shell, and differentiate the
exact integrand.  All three normalized lengths remain in compact subsets of
the strict triangle range, so the identical weighted-average argument gives
an $O_J(n)$ logarithmic derivative as $\|\mathbf x\|/R$ varies across its
shell.
\end{proof}

\begin{lemma}[Thin-shell second moment]\label{lem:thin-shell-count}
Uniformly for $R/\Rstar$ between a positive constant and $\exp(O(\log^2n))$,
\[
 \E_{\cL}|\cL\cap S_R|=\vol(S_R)
\]
and, for all sufficiently large $n$,
\begin{equation}\label{eq:shell-var}
 \Var_{\cL}(|\cL\cap S_R|)\le(2+o(1))\vol(S_R).
\end{equation}
The same holds for $T_R$.
\end{lemma}
\begin{proof}
The mean is Siegel.  Apply Proposition~\ref{thm:rogers-general} with $k=2$ to $F(\mathbf{x},\mathbf{y})=\1_{S_R}(\mathbf{x})\1_{S_R}(\mathbf{y})$.  The full-rank identity diagram gives $\vol(S_R)^2$.  A rank-one diagram has $\mathbf{y}=(a/q)\mathbf{x}$ with $q\ge1$, $a\ne0$, $\gcd(a,q)=1$, and coefficient $q^{-n}$.  The cases $q=1,a=\pm1$ contribute exactly $2\vol(S_R)$.

For $q\ge2$, support of both shell indicators implies
\[
 \frac{|a|}{q}\in\left[\frac{1-w_n}{1+w_n},\frac{1+w_n}{1-w_n}\right],
\]
so for each $q$ there are at most $O(1+qw_n)$ admissible integers $|a|$ (and two signs).  For every such pair, dropping the second shell indicator bounds the integral by $\vol(S_R)$.  Hence the total $q\ge2$ rank-one contribution is at most
\[
 C\vol(S_R)\sum_{q\ge2}q^{-n}(1+qw_n)
 =O(2^{-n})\vol(S_R).
\]
Thus
\[
 \Var_{\cL}(|\cL\cap S_R|)
 \le \bigl(2+O(2^{-n})\bigr)\vol(S_R),
\]
which implies~\eqref{eq:shell-var}.  This is the explicit two-vector specialization of Rogers' Theorem~4~\cite[Theorem~4, eqs.~(9)--(11)]{Rogers1955}.
\end{proof}

\subsection{The four raw marked kernels and their scalar geometry}

The correction argument below is used for four, and only four, raw kernels.
Writing them down before introducing the diagram classes removes a possible
ambiguity about which Rogers column is marked, which variables are free
pivots, and which shell fiber is being integrated.  Fix
$\alpha\in J\Subset(0,2)$, put $S=S_R$, $T=S_{\alpha R}$, and let
$0\le\omega\le1$ be a Borel radial flattening weight on $T$, extended
by zero to $\R^n$.  For $s\ge1$ define
\begin{align}
 F_s^{\rm ct}(\mathbf d,\mathbf x_1,\ldots,\mathbf x_s)
 &:=\mathbf1_T(\mathbf d)\omega(\mathbf d)^s
   \prod_{j=1}^s\mathbf1_S(\mathbf x_j)
                       \mathbf1_S(\mathbf x_j-\mathbf d),
 \label{eq:raw-kernel-ct}\\
 F_s^{\rm ce}(\mathbf x,\mathbf d_1,\ldots,\mathbf d_s)
 &:=\mathbf1_S(\mathbf x)
   \prod_{j=1}^s\mathbf1_T(\mathbf d_j)
                       \mathbf1_S(\mathbf x-\mathbf d_j)
                       \omega(\mathbf d_j),
 \label{eq:raw-kernel-ce}\\
 F_s^{\rm at}(\mathbf d,\mathbf c_1,\ldots,\mathbf c_s)
 &:=\mathbf1_T(\mathbf d)\omega(\mathbf d)^s
   \prod_{j=1}^s\mathbf1_S(\mathbf c_j)
                       \mathbf1_S(\mathbf d+\mathbf c_j),
 \label{eq:raw-kernel-at}\\
 F_s^{\rm ae}(\mathbf x,\mathbf c_1,\ldots,\mathbf c_s)
 &:=\mathbf1_S(\mathbf x)
   \prod_{j=1}^s\mathbf1_S(\mathbf c_j)
                       \mathbf1_T(\mathbf x-\mathbf c_j)
                       \omega(\mathbf x-\mathbf c_j).
 \label{eq:raw-kernel-ae}
\end{align}
Here the superscripts mean centered target, centered endpoint, affine target,
and affine endpoint.  In the first two kernels every displayed variable is a
lattice column and the first column is the mark.  In the last two kernels the
first variable is integrated over Euclidean space after affine unfolding;
only the $s$ centered columns $\mathbf c_j$ are presented to Rogers' formula.

\begin{lemma}[Raw-kernel identity and free fibers]
\label{lem:four-raw-kernels}
For each $1\le s\le2h$, the $s$-occurrence raw terms in the centered target,
centered endpoint, affine target, and affine endpoint moment expansions are,
respectively, the lattice sums of
\eqref{eq:raw-kernel-ct}--\eqref{eq:raw-kernel-ae}, with the first variable
summed over the indicated centered or affine lattice.  After conditioning on
the mark, every free pivot belongs to precisely one of the following fibers:
\begin{center}
\begin{tabular}{c|c|c|c}
kernel & mark & free pivot & Rogers order\\ \hline
$F_s^{\rm ct}$ & $\mathbf d\in T$ &
 $\mathbf x\in S$, $\mathbf x-\mathbf d\in S$ & $s+1$\\
$F_s^{\rm ce}$ & $\mathbf x\in S$ &
 $\mathbf d\in T$, $\mathbf x-\mathbf d\in S$ & $s+1$\\
$F_s^{\rm at}$ & $\mathbf d\in T$ (Euclidean) &
 $\mathbf c\in S$, $\mathbf d+\mathbf c\in S$ & $s$\\
$F_s^{\rm ae}$ & $\mathbf x\in S$ (Euclidean) &
 $\mathbf c\in S$, $\mathbf x-\mathbf c\in T$ & $s$.
\end{tabular}
\end{center}
\end{lemma}

\begin{proof}
For the centered target, expand the $s$-th power of
$\omega(\mathbf d)r_{\cL,R}(\mathbf d)$ and name its $s$ parent points
$\mathbf x_1,\ldots,\mathbf x_s$; this gives
\eqref{eq:raw-kernel-ct}.  Expanding the endpoint degree gives
\eqref{eq:raw-kernel-ce}.  In the affine cases, first write every other
affine point as the affine mark plus a centered lattice vector.  For a
fundamental domain $\mathcal F$ of $\cL$, the nonnegative identity
\[
 \int_{\mathcal F}\sum_{\mathbf z\in\boldsymbol\tau+\cL}
 H(\mathbf z)d\boldsymbol\tau=\int_{\R^n}H(\mathbf z)d\mathbf z
\]
then unfolds the mark; its full measurable form is recorded later in
Lemma~\ref{lem:affine-fiber-unfold}.  The resulting kernels are
\eqref{eq:raw-kernel-at} and~\eqref{eq:raw-kernel-ae}.  The four fiber
descriptions are simply the nonzero factors involving one free pivot.  Since
$S$ is separated from the origin, the apparently unrestricted sums over
$\mathbf c\in\cL$ contain no zero Rogers column.  The asserted tuple orders
follow by counting the lattice, rather than Euclidean, variables.
\end{proof}

\begin{lemma}[Uniform scalar-coordinate constraints]
\label{lem:uniform-scalar-certificate}
Let $\mathbf e$ be the unit vector in the direction of the fixed mark and
write a free pivot as
$\mathbf v=t\mathbf e+\rho\mathbf u$, with
$\mathbf u\in S^{n-2}\subset\mathbf e^\perp$.  Uniformly for
$\alpha\in J$ and all four kernels, its supported scalar coordinates satisfy
\begin{center}
\begin{tabular}{c|c|c}
kernel & $t/R$ & $\rho/R$\\ \hline
$F_s^{\rm ct}$ & $\alpha/2+O_J(w_n)$ &
 $\sqrt{1-\alpha^2/4}+O_J(w_n)$\\
$F_s^{\rm ce}$ & $\alpha^2/2+O_J(w_n)$ &
 $\alpha\sqrt{1-\alpha^2/4}+O_J(w_n)$\\
$F_s^{\rm at}$ & $-\alpha/2+O_J(w_n)$ &
 $\sqrt{1-\alpha^2/4}+O_J(w_n)$\\
$F_s^{\rm ae}$ & $1-\alpha^2/2+O_J(w_n)$ &
 $\alpha\sqrt{1-\alpha^2/4}+O_J(w_n)$.
\end{tabular}
\end{center}
The constants are independent of $R$.  All displayed transverse radii are
therefore bounded above and below by fixed positive multiples of $R$.
Conditional on the complete scalar tuple, distinct free pivots have product
uniform measure in their transverse directions.  For a dependent column with transverse coefficients $b_i/q$, let
$\rho_0$ be the common transverse radius in the table.  Retain the exact
coefficients $c_i=(b_i/q)(\rho_i/\rho_0)$ after fixing the pivot scalars.
Then
\[
 \|c\|^2=(1+O_J(w_n))\|b/q\|^2,
 \qquad
 \Big\|\sum_i c_i\mathbf u_i\Big\|^2=1+O_J(w_n)
 \quad\hbox{on output support}.
\]
The relative error is uniform even for unbounded $q$ and coefficient height;
no additive coefficient-error estimate is asserted in that range.  The
exact scalar Jacobians remain inside the fiber measures.
\end{lemma}

\begin{proof}
The first row follows from
\[
 2\langle\mathbf x,\mathbf d\rangle
 =\|\mathbf x\|^2+\|\mathbf d\|^2
   -\|\mathbf x-\mathbf d\|^2.
\]
The second is the same identity with $\mathbf x$ as the marked axis and
$\mathbf d$ as the pivot.  For the third row use
$\|\mathbf d+\mathbf c\|^2-\|\mathbf c\|^2
 =2\langle\mathbf c,\mathbf d\rangle+\|\mathbf d\|^2$;
for the fourth use
$\|\mathbf x-\mathbf c\|^2
 =\|\mathbf x\|^2+\|\mathbf c\|^2
   -2\langle\mathbf x,\mathbf c\rangle$.
Every relevant norm is $R$ or $\alpha R$ times $1+O(w_n)$, which gives the
listed longitudinal coordinates.  The transverse coordinates follow from
$\rho^2=\|\mathbf v\|^2-t^2$.  Compactness of $J\Subset(0,2)$ makes the
estimates uniform and nondegenerate.  Product angular measure is
Lemma~\ref{lem:fiber-polar}.  Finally, $\rho_i/\rho_0=1+O_J(w_n)$ uniformly, so summing
$c_i^2=(b_i/q)^2(1+O_J(w_n))$ proves the relative squared-norm bound.
The output has the same supported transverse-radius interval as a pivot.
Fubini retains the exact polar Jacobian.
\end{proof}

For the general centered endpoint kernel only finitely many low-parent
integer relations must be separated from zero.  We fix
$C_{\rm lp}:=10$ and define
\begin{equation}\label{eq:low-parent-gap}
 \Delta_{\rm lp}(\alpha):=
 \min_{\substack{(A,z)\in\mathbb Z^2\setminus\{(0,0)\}\\
                  |A|\le C_{\rm lp},\ |z|\le3}}
 |2A+\alpha^2z|.
\end{equation}
For $0<\alpha<2$, this gap vanishes exactly when
\[
 \alpha^2\in\left\{\frac23,\,1,\,\frac43,\,2,\,
                         \frac83,\,3,\,\frac{10}3\right\}.
\]
Indeed, a zero has $z\ne0$ and $\alpha^2=-2A/z$ with $1\le|z|\le3$.

\begin{lemma}[Height-one relations for the four kernels]
\label{lem:four-kernel-relation-certificate}
Let a supported $q=1$ nonpivot column have $t\le k$ distinct free
unmarked parents and let every nonzero unmarked coefficient
$B_j$ lie in $\{\pm1\}$.  Then the following statements hold uniformly for
$\alpha\in J$.
\begin{enumerate}[label=\textup{(\roman*)},leftmargin=*]
\item In $F_s^{\rm ct}$, if the column is
$\mathbf x'=A\mathbf d+\sum_jB_j\mathbf x_j$, then
\begin{equation}\label{eq:four-kernel-ct-relation}
 2A+\sum_jB_j=1.
\end{equation}
Thus $t=0,2$ are impossible; for $t=1$ the column is exactly equality
$(A,B_1)=(0,1)$ or reflection $(A,B_1)=(1,-1)$.
\item In $F_s^{\rm ce}$, if
$\mathbf d'=A\mathbf x+\sum_jB_j\mathbf d_j$, then
\begin{equation}\label{eq:four-kernel-ce-relation}
 \left|2A+\alpha^2\left(\sum_jB_j-1\right)\right|
 \le C_J(t+1)w_n.
\end{equation}
For $\alpha=(\ell-1)/\ell$ and $k\le2H\ell+1$ this forces
$A=0$ and $\sum_jB_j=1$.  For arbitrary $\alpha\in J$, the same conclusion
holds when $t\le2$ and
$\Delta_{\rm lp}(\alpha)>C_J(t+1)w_n$.  In either case equality is the only
one-parent column and $t=0,2$ are impossible.
\item In $F_s^{\rm at}$, a column
$\mathbf c'=\sum_jB_j\mathbf c_j$ satisfies
$\sum_jB_j=1$.  Hence equality is the only one-parent column and $t=0,2$
are impossible.
\item In $F_s^{\rm ae}$, a column
$\mathbf c'=\sum_jB_j\mathbf c_j$ satisfies
\begin{equation}\label{eq:four-kernel-ae-relation}
 \left|1-\frac{\alpha^2}{2}\right|
 \left|\sum_jB_j-1\right|\le C_J(t+1)w_n.
\end{equation}
If $|1-\alpha^2/2|\ge\delta>0$, this again forces
$\sum_jB_j=1$ for $k=O(\log n)$, with the same small-$t$ conclusions.
\end{enumerate}
In the centered cases the marked coefficient $A$ is fixed once the
unmarked coefficients are fixed, and $|A|=O_J(k)$.
\end{lemma}

\begin{proof}
For the centered target fiber, every supported parent and output obeys
\[
 \frac{\langle\mathbf x,\mathbf d\rangle}{\|\mathbf d\|^2}
 =\frac12+O_J(w_n).
\]
Insert the relation for $\mathbf x'$ and use linearity.  The distance of
$2A+\sum_jB_j-1$ from zero is $O_J((t+1)w_n)=o(1)$; it is an integer, so
it is zero.  Direct enumeration gives the small-$t$ alternatives in~(i).

For a centered endpoint parent,
\[
 \frac{2\langle\mathbf x,\mathbf d\rangle}{\|\mathbf x\|^2}
 =\alpha^2+O_J(w_n).
\]
Substitution gives~\eqref{eq:four-kernel-ce-relation}.  If
$\alpha=(\ell-1)/\ell$, multiply by $\ell^2$.  The left side becomes the
absolute value of
\[
 2A\ell^2+(\ell-1)^2z,\qquad z:=\sum_jB_j-1,
\]
up to an error $O_J(\ell^2(k+1)n^{-2})=o(1)$.  The integer therefore
vanishes.  Coprimality of $\ell^2$ and $(\ell-1)^2$, together with
$|z|\le k+1<\ell^2$ for large $n$, gives $z=A=0$.  If instead $t\le2$,
then $|z|\le3$ and~\eqref{eq:four-kernel-ce-relation} gives
$|A|\le C_{\rm lp}$ for large $n$; the stated gap again forces
$(A,z)=(0,0)$.  Parity gives the small-$t$ conclusions.

For the affine target use
$\langle\mathbf c,\mathbf d\rangle/\|\mathbf d\|^2
=-1/2+O_J(w_n)$ for the parents and output.  For the affine endpoint use
$\langle\mathbf x,\mathbf c\rangle/\|\mathbf x\|^2
=1-\alpha^2/2+O_J(w_n)$.  Substitution gives (iii)
and~\eqref{eq:four-kernel-ae-relation}; the left side in (iii) is an
integer at distance $o(1)$ from zero.  Finally,
\eqref{eq:four-kernel-ct-relation} determines $A$ in the target case, while
\eqref{eq:four-kernel-ce-relation} and $\alpha\in J$ give $|A|=O_J(k)$ in
the endpoint case.
\end{proof}

\subsection{Marked-diagram correction engine}
The next lemmas are the reusable output of this section.  Section
\ref{sec:shell-moments} verifies the shell-specific one-parent relations and
combines the partition diagrams; the present bounds then control every
remaining admissible matrix.

\subsubsection{An exhaustive marked-diagram trichotomy}\label{fs-sec:rogers}
\begin{definition}[Marked Rogers diagram classes]\label{fs-def:diagram-classes}
For a centered kernel, apply Proposition~\ref{thm:rogers-general} with the marked
target or endpoint in column $1$.  For an affine kernel, first unfold the
Euclidean mark and apply Rogers only to the remaining centered columns, as in
Lemma~\ref{lem:four-raw-kernels}.  Under the kernel-specific parameter conditions in
Proposition~\ref{prop:kernel-trichotomy-audit}, every supported admissible
matrix lies in exactly one of the following classes.
\begin{enumerate}[label=(\Alph*),leftmargin=*]
\item \emph{Partition diagrams}: $q=1$ and every nonmarked dependent column
is one of the one-parent equality/reflection identifications in
Lemma~\ref{lem:four-kernel-relation-certificate}.  In an affine kernel the
word ``nonmarked'' refers to every Rogers column, since the mark has already
been removed.
\item \emph{Geometric nonpartition diagrams}: $q=1$, every nonmarked free-row coefficient is in $\{0,\pm1\}$, and some dependent column genuinely combines at least three distinct free transverse pivots.
\item \emph{Arithmetic/large-coefficient diagrams}: $q\ge2$, or $q=1$ and some nonmarked free-row coefficient in a nonpivot column has absolute value at least $2$.
\end{enumerate}
A large coefficient in the centered marked row alone does not create a
fourth class: once the nonmarked coefficients are fixed, the longitudinal
equation fixes it and the transverse relation is unchanged.  This refines
the standard split in Han~\cite[Propositions~2.2--2.4]{Han2024} by using the
special shell geometry to identify exactly which $q=1$,
$\{0,\pm1\}$ matrices are partition diagrams.
\end{definition}

\begin{lemma}[Exhaustiveness and disjointness of the trichotomy]
\label{lem:diagram-trichotomy-exhaustive}
Suppose that the shell-specific height-one relation analysis has the following
two properties: every supported column involving at most one distinct free
unmarked pivot is either a designated equality/reflection partition column or
is unsupported, and no supported nonpartition column involves exactly two
such pivots.  Then classes \textup{(A)}--\textup{(C)} of
Definition~\ref{fs-def:diagram-classes} are pairwise disjoint and contain
every supported admissible matrix.
\end{lemma}
\begin{proof}
Let $D$ be supported.  If $q\ge2$, then $D$ is in class~\textup{(C)}.  Hence
assume $q=1$.  If any nonmarked free-row entry of a nonpivot column has
absolute value at least two, then again $D$ is in class~\textup{(C)}.
We may therefore assume that all such entries belong to $\{0,\pm1\}$.

Scan the nonpivot columns in their Rogers order.  If every column is a
one-parent equality/reflection identification, the matrix is, by definition,
a partition diagram and lies in class~\textup{(A)}.  Otherwise take the first
nonpartition column.  The stated low-parent classification rules out zero,
one, and two distinct free unmarked pivots for this column.  It consequently
combines at least three distinct free transverse pivots and $D$ lies in
class~\textup{(B)}.  These alternatives are mutually exclusive because they
are separated first by $q$, then by coefficient height, and finally by the
one-parent partition property.

For the centered target, near-unit centered endpoint, generic centered
endpoint, affine target, and nonresonant affine endpoint, the two assumed
properties are precisely the four cases of
Lemma~\ref{lem:four-kernel-relation-certificate}.  Thus the trichotomy is
proved before it is used; the later shell-moment lemmas merely specialize
that support relation to their notation.
\end{proof}

\begin{proposition}[Trichotomy for the four kernels]
\label{prop:kernel-trichotomy-audit}
Every raw marked kernel used in the centered target, centered endpoint,
affine target, and affine endpoint moment propositions satisfies the hypotheses
of Lemma~\ref{lem:diagram-trichotomy-exhaustive}, under exactly the parameter
conditions stated in those propositions.  More precisely:
\begin{enumerate}[label=\textup{(\alph*)},leftmargin=*]
\item for a centered target, a supported height-one column with $t$ distinct
unmarked parents obeys $2A+\sum_{j=1}^tB_j=1$; the only one-parent columns are
equality and reflection, and $t=0,2$ are impossible;
\item for the near-unit centered endpoint band, the Diophantine identity
$2A+\gamma^2(\sum_jB_j-1)=O((t+1)w_n)$ forces
$A=0$ and $\sum_jB_j=1$; hence equality is the only one-parent column and
$t=0,2$ are impossible;
\item for a general fixed-ratio centered endpoint, the same conclusion for
$t\le2$ follows whenever
$\Delta_{\rm lp}(\alpha)\ge C_{\rm nr}(J)n^{-2}$ with
$C_{\rm nr}(J)$ larger than the uniform support-error constant;
\item after affine unfolding, the target equation forces
$\sum_jB_j=1$, while the endpoint equation forces the same identity whenever
$\inf_{\alpha\in J}|1-\alpha^2/2|>0$.  Equality is the only affine
one-parent column; there is no affine reflection partition.
\end{enumerate}
In every case, therefore, a supported height-one column is either a stated
partition column or combines at least three distinct free transverse pivots.
All remaining supported matrices have $q\ge2$ or a large unmarked
coefficient and belong to the arithmetic class.
\end{proposition}

\begin{proof}
Parts~\textup{(a)}--\textup{(d)}, including the parity checks, are the four
cases of Lemma~\ref{lem:four-kernel-relation-certificate}.  In part~(b),
multiplication by $\ell^2$ is the only use of the rational choice
$\gamma=(\ell-1)/\ell$.  Part~(c) invokes the gap only for $t\le2$, because
every $t\ge3$ column already receives the geometric penalty.  The final
sentence of that lemma also proves that no fifth ``marked-row only'' class is
missing.  Lemma~\ref{lem:diagram-trichotomy-exhaustive} now completes the
classification.  The same-shell centered endpoint conclusion in
Theorem~\ref{thm:coverage} uses central symmetry and the target count, so it
introduces no additional raw endpoint kernel.
\end{proof}

\begin{lemma}[Missing-fiber pivot principle]\label{fs-lem:missing-fiber}
In an $s$-occurrence raw marked moment, every nonidentity diagram has at
most $s-1$ free fiber pivots, both in the centered and unfolded-affine cases.
\end{lemma}
\begin{proof}
For a centered kernel the Rogers order is $k=s+1$; column $1$ is the
marked pivot and a nonidentity diagram has rank $r\le k-1=s$, leaving
exactly $r-1\le s-1$ free fiber pivots.  After affine unfolding the mark is
external, the Rogers order is $k=s$, and all $r\le k-1=s-1$ pivots are
free fibers.  These are applications of Lemma~\ref{lem:rogers-pivot-facts}.
\end{proof}

\begin{lemma}[First offending column principle]\label{fs-lem:first-offending}
Scan the nonpivot columns of a supported $q=1$ diagram from left to right and
let $j_*$ be the first nonpartition column.  Let
$\mathcal I_*$ be the set of free fiber rows having a nonzero coefficient
in column $j_*$; the centered marked row is excluded.  Retain every pivot
factor, even when its column occurs after $j_*$, and every partition factor
before $j_*$.  This retained integrand factors, conditional on the mark, as
\begin{equation}\label{eq:pre-offending-factorization}
 G_0(\mathbf m)\prod_{i\in\mathcal I_{\rm free}}
                         G_i(\mathbf m,\mathbf v_i),
 \qquad 0\le G_i\le1,
\end{equation}
after repeated occurrences belonging to the same partition block are
absorbed into the corresponding factor $G_i$.  Here
$\mathcal I_{\rm free}=\{2,\ldots,r\}$ in a centered marked diagram and
$\mathcal I_{\rm free}=\{1,\ldots,r\}$ after affine unfolding.
Consequently, after
disintegrating every pivot into scalar and transverse coordinates, the
directions $(\mathbf u_i)_{i\in\mathcal I_*}$ have product uniform
spherical measure.  Define the retained fiber masses at the fixed mark by
\[
 W_i(\mathbf m):=\int_{\R^n}G_i(\mathbf m,\mathbf v_i)\,d\mathbf v_i.
\]
Let $\mathcal M$ be the external or lattice marked domain, and let $V_i$
be constants satisfying $W_i(\mathbf m)\le V_i$ throughout $\mathcal M$.
If the shell indicator $H_{j_*}$ belonging to column $j_*$ has conditional
angular integral at most $\eta$ for every supported scalar tuple, then the
complete diagram integral, including all columns after $j_*$, satisfies
\begin{equation}\label{eq:first-offending-complete-bound}
 \begin{aligned}
 \mathcal I_D(F)
 &\le\eta\int_{\mathcal M}
       \prod_{i\in\mathcal I_{\rm free}}W_i(\mathbf m)\,d\mathbf m\\
 &\le\eta\,\vol(\mathcal M)
       \prod_{i\in\mathcal I_{\rm free}}V_i.
 \end{aligned}
\end{equation}
\end{lemma}
\begin{proof}
Write the pivot columns as $j_1<\cdots<j_r$.  Han's echelon condition
$D_{ij}=0$ for $j<j_i$ implies that a column at position $j$ can involve
only rows whose pivot has already appeared.  A partition column before
$j_*$ is, by Lemma~\ref{lem:four-kernel-relation-certificate}, either a
literal repetition of one active pivot or, only for the centered target
kernel, its reflection through the mark.  Its shell indicator and weight
therefore depend on the mark and that one pivot, but on no other pivot.
Multiplying all earlier columns in the same block merely changes the
one-pivot factor $G_i$.  Induction over the columns before $j_*$ proves
\eqref{eq:pre-offending-factorization}.  Pivot-column indicators themselves
already have this product form.

Delete the dependent-column indicators and weights after $j_*$, retaining
all pivot-column factors.  The deleted factors lie in $[0,1]$, so this gives
an upper bound.  At a fixed mark, disintegrate each exact one-pivot measure
$G_i(\mathbf m,\mathbf v_i)d\mathbf v_i$ into its scalar coordinates and
its transverse direction.  Radiality about the marked axis and
Lemma~\ref{lem:fiber-polar} give normalized uniform angular measures;
product Lebesgue measure makes them independent.  Integrating the directions
in $\mathcal I_*$ contributes at most $\eta$.  Tonelli restores the exact
scalar measures, of masses $W_i(\mathbf m)$.  Integrating the mark and using
$G_0\le1$ gives the first inequality in
\eqref{eq:first-offending-complete-bound}; the uniform envelopes $V_i$ give
the second.
\end{proof}

\subsubsection{Arithmetic and large-coefficient diagrams}\label{fs-sec:arithmetic-direct}
\begin{lemma}[Rayleigh--determinant bound]\label{fs-lem:rayleigh-det}
Let $s\ge2$, let $G\succeq0$ be an $s\times s$ correlation matrix, and let
$\mathbf c\in\R^s\setminus\{0\}$.  If
$\mathbf c^{\mathsf T}G\mathbf c=\tau$, then
\begin{equation}\label{fs-eq:rayleigh-det}
 \det G\le\frac{\tau}{\|\mathbf c\|^2}
 \left(\frac{s-\tau/\|\mathbf c\|^2}{s-1}\right)^{s-1}
 \le\frac{e\,\tau}{\|\mathbf c\|^2}.
\end{equation}
Thus if $\tau=1+o(1)$ and $\|\mathbf c\|^2\ge4(1-o(1))$, then
$\det G\le e/4+o(1)$.
\end{lemma}
\begin{proof}
Rotate so that $\mathbf c/\|\mathbf c\|$ is the first basis vector.  The
$(1,1)$ entry becomes $g_{11}=\tau/\|\mathbf c\|^2$.  The Schur complement is
positive semidefinite with trace at most $s-g_{11}$, so AM--GM on its $s-1$
eigenvalues gives the first inequality.  The second is
$(1+(1-g_{11})/(s-1))^{s-1}\le e$.
\end{proof}

\begin{lemma}[General coefficient spherical penalty]\label{fs-lem:general-spherical-penalty}
Fix $H>0$.  Let $n-O(\log n)\le d_\perp\le n$ and
$2\le s\le2H\log n+1$.  For independent
uniform $\mathbf{u}_1,\ldots,\mathbf{u}_s\in S^{d_\perp-1}$ and
$\mathbf c=(c_1,\ldots,c_s)\in\R^s\setminus\{0\}$, uniformly for
$\delta_n=o(1)$,
\begin{align}
 &\Prb\!\left[\left\|\sum_{i=1}^s c_i\mathbf{u}_i\right\|^2\in[1-\delta_n,1+\delta_n]\right]\notag\\
 &\quad\le\exp(O(s^2\log n))
 \min\left\{1,\frac{e(1+\delta_n)}{\|\mathbf c\|^2}\right\}^{(d_\perp-s-1)/2}.
 \label{fs-eq:general-spherical}
\end{align}
For radial factors $a_i=1+O(\eta_n)$, $\eta_n=o(1)$, apply the
same bound to the exact vector $(c_i a_i)_i$; its squared norm is
$(1+O(\eta_n))\|c\|^2$, uniformly in the coefficient height.
\end{lemma}
\begin{proof}
By Lemma~\ref{lem:gram-density}, the Gram density is proportional to
$(\det G)^{(d_\perp-s-1)/2}$ and its domain/normalizer loss is at most
$\exp(O(s^2\log n))$.  On the event,
$\mathbf c^{\mathsf T}G\mathbf c\le1+\delta_n$, so
Lemma~\ref{fs-lem:rayleigh-det} gives the displayed determinant bound.  The radial assertion follows from
$\sum_i c_i^2a_i^2=(1+O(\eta_n))\sum_i c_i^2$; it uses a relative
error and therefore remains valid for arbitrarily large coefficients.
\end{proof}

\begin{lemma}[Fiber-product integral majorization]
\label{fs-lem:fiber-product-majorization}
Fix any external marked variables (possibly none) and let the remaining free
Rogers pivots $(\mathbf v_i)_{i\in\mathcal I}$ be restricted to shell fibers
$F_i$.  Disintegrate each
Lebesgue measure as
\[
 d\mathbf v_i=d\lambda_i(s_i)\,d\sigma_i(\mathbf u_i),
\]
where $s_i$ contains all radial/longitudinal coordinates (and their exact
Jacobian), while $\sigma_i$ is uniform probability measure on the relevant
transverse sphere.  Suppose that, for every supported scalar tuple
$s=(s_i)_{i\in\mathcal I}$, an indicator $H$ satisfies
\[
 \int H(s,(\mathbf u_i)_{i\in\mathcal I})
       \prod_{i\in\mathcal I}d\sigma_i(\mathbf u_i)\le\eta.
\]
Then, for every additional indicator $0\le G\le H$ and fiber weights
$0\le f_i\le1$,
\begin{equation}
 \int G\prod_{i\in\mathcal I} f_i(\mathbf v_i)\,d\mathbf v_i
 \le \eta\prod_{i\in\mathcal I}\vol(F_i).
 \label{fs-eq:fiber-product-majorization}
\end{equation}
The same assertion holds with weighted fiber masses on the right whenever
$f_i$ is incorporated into $d\lambda_i$.
\end{lemma}
\begin{proof}
Lemma~\ref{lem:fiber-polar} gives the displayed disintegration with no
change of variables beyond the exact polar Jacobian already contained in
$d\lambda_i$.  Integrate the transverse directions first, use $G\le H$ and
the uniform bound $\eta$, and then apply Tonelli to the nonnegative scalar
measures.  Their total masses are precisely the (weighted) fiber masses.
All remaining dependent-column restrictions only decrease the integral.
\end{proof}

\begin{lemma}[Uniform denominator--height summation]
\label{lem:denominator-height-summation}
Fix $H_0,C_0>0$.  Uniformly for
$1\le k\le2H_0\log n+1$ and every $\delta_n=O(k n^{-2})$, one has
\begin{align}
 &(C_0k)^{k^2}\sum_{q\ge2}q^{-n+k^2}
 \le e^{Ck^2\log n}2^{-n},
 \label{eq:denominator-sum-explicit}\\
 &\sum_{t\ge1}(C_0k2^t)^{k^2}
 \left(\frac{e+\delta_n}{2^{2t}}\right)^{(n-C_0k)/2}
 \le e^{Ck^2\log n}
       \left(\frac e4+\varepsilon_n\right)^{n/2},
 \label{eq:height-sum-explicit}
\end{align}
where $C=C(C_0,H_0)$ and one may take
$\varepsilon_n=C'_{C_0,H_0}\log n/n$.  The same two bounds hold after
multiplication by $2^k$ and after replacing $k^2$ in any polynomial
enumeration exponent by $C_0k^2$.
\end{lemma}

\begin{proof}
Put $m=n-k^2$.  Uniformly in the stated range, $m=n-o(n)$ and
\[
 \sum_{q\ge2}q^{-m}
 \le2^{-m}\left(1+\sum_{q\ge3}(2/q)^m\right)
 \le2^{-m}(1+e^{-\Omega(n)}).
\]
The factor $(C_0k)^{k^2}$ and the difference between $2^{-m}$ and
$2^{-n}$ are both $e^{O(k^2\log n)}$, proving
\eqref{eq:denominator-sum-explicit}.

For the height sum, the quotient of the term with index $t+1$ by the term
with index $t$ is
\[
 2^{k^2-(n-C_0k)}=e^{-\Omega(n)}
\]
uniformly for large $n$.  The series is therefore its $t=1$ term times
$1+e^{-\Omega(n)}$.  That term is
\[
 (2C_0k)^{k^2}
 \left(\frac{e+\delta_n}{4}\right)^{(n-C_0k)/2}.
\]
The polynomial factor and the shift of the exponent by $C_0k/2$ cost
$e^{O(k^2\log n)}$.  Since
$\delta_n=O(\log n/n^2)$, enlarging a deterministic base by
$C'\log n/n$ absorbs this error uniformly and proves
\eqref{eq:height-sum-explicit}.  The final variants only enlarge the
constant $C$.
\end{proof}

\begin{lemma}[Direct marked arithmetic remainder]\label{fs-lem:arithmetic-direct}
Fix $J\Subset(0,2)$ and $H>0$, and use one of the four raw kernels of
Lemma~\ref{lem:four-raw-kernels} at Rogers order $k\le2H\log n+1$.
Put $\chi=1$ for a centered kernel and $\chi=0$ after affine unfolding.
Suppose that the marked domain is $\mathcal M$ and each free pivot carries
a scalar weight in $[0,1]$ with weighted fiber mass at most $V_{\rm fib}$,
uniformly in the mark.  Then the total contribution of matrices satisfying
$q\ge2$, or $q=1$ with some free-row entry of magnitude at least two, is
at most
\begin{equation}\label{fs-eq:direct-arithmetic}
 \vol(\mathcal M)\max(1,V_{\rm fib})^{k-1-\chi}
 e^{O_{J,H}(k^2\log n)}
 \left[2^{-n}+\left(\frac e4+o(1)\right)^{n/2}\right].
\end{equation}
The exponent $k-1-\chi$ is $k-2$ for centered kernels and $k-1$
for unfolded affine kernels, as in Proposition~\ref{prop:uniform-diagram-budget}.
The respective free rows are $2,\ldots,r$ and $1,\ldots,r$.  Extra dependent-column factors in $[0,1]$ only
decrease this bound.
\end{lemma}
\begin{proof}
Every matrix under consideration is nonidentity, so $r\le k-1$ and the
number of free fibers is exactly $f=r-\chi\le k-1-\chi$.
All estimates below are conditional on the mark.  For an affine kernel the
mark is Euclidean and is not a Rogers row.

\emph{Exact coefficients and a height-attaining column.}
For a nonpivot column $j$, let $b_j$ be its vector of free-row integer
entries.  In a centered kernel its output is
$q^{-1}(A_j\mathbf m+\sum_i b_{ij}\mathbf v_i)$; in an affine kernel
omit the marked term.  Comparable input and output norms give
\begin{equation}\label{fs-eq:A-bound}
 |A_j|\le C_J\bigl(q+\sqrt{k}\|b_j\|\bigr)
 \le C_Jk(q+\|b_j\|).
\end{equation}
Let $\rho_0$ be the common transverse radius in
Lemma~\ref{lem:uniform-scalar-certificate}.  For each exact scalar tuple
the output constraint implies
\[
 \Big\|\sum_i c_i\mathbf u_i\Big\|^2\in[1-C_Jw_n,1+C_Jw_n],
 \quad c_i=\frac{b_{ij}\rho_i}{q\rho_0},
 \quad \|c\|^2\ge(1-C_Jw_n)\|b_j/q\|^2.
\]
If $b_j$ has at least two nonzero entries, Lemmas
\ref{fs-lem:general-spherical-penalty} and
\ref{fs-lem:fiber-product-majorization} bound the complete pivot integral
by the product of its weighted fiber masses times
\begin{equation}\label{fs-eq:large-coeff-angular}
 e^{C_Jk^2\log n}
 \min\left\{1,\frac{e+C_Jw_n}{\|b_j/q\|^2}\right\}^{p_n},
 \qquad p_n:=\frac{n-k-2}{2}>0.
\end{equation}
Here replacing the actual Gram-density exponent by $p_n$ enlarges the bound.
We keep the selected column and all pivot factors, and drop every other
dependent-column factor before integrating directions.

A column with $b_j=0$ is unsupported because its output transverse radius
would vanish.  A column with exactly one nonzero entry satisfies
$|b_{ij}|/q=1+O_J(w_n)$ on support.  Thus, for all sufficiently large $n$,
such a column has $\|b_j/q\|<2$.  This argument applies to every $q$;
it does not round $|b_{ij}|/q$ to an integer.  In particular a supported
height-attaining column of relative norm at least two has at least two
parents.  Centered rank-one diagrams have no free transverse row and are
unsupported; affine rank-one diagrams can occur only in the bounded-height
denominator family below if they are arithmetic.

\emph{Disjoint bins and their counts.}
Set
\[
 L(D):=\max_{j\ \mathrm{nonpivot}}\|b_j/q\|,
 \qquad
 B(D):=\max_{\substack{i\ \mathrm{free}\\j\ \mathrm{nonpivot}}}|b_{ij}|.
\]
The maximum over an empty set is zero.  The arithmetic family is the
disjoint union
\begin{align*}
 \mathcal C_{\rm b}&=\{q\ge2,\ L(D)<2\},\\
 \mathcal C_{\rm d,a}&=\{q\ge2,\ 2^a\le L(D)<2^{a+1}\},\quad a\ge1,\\
 \mathcal C_{\rm h,a}&=\{q=1,\ B(D)\ge2,\
                         2^a\le L(D)<2^{a+1}\},\quad a\ge1.
\end{align*}
The last family is exhaustive at $q=1$ because $B(D)\ge2$ implies
$L(D)\ge2$.  This coefficient-height test is made before norm binning:
a $q=1$ column with many $\pm1$ entries belongs to the geometric class,
even if its Euclidean coefficient norm exceeds two.

There are at most $2^k$ pivot patterns.  In a bin $L(D)<2H_*$ every free
entry has magnitude at most $2qH_*$, and \eqref{fs-eq:A-bound} bounds every
marked entry by $C_JkqH_*$.  Ignoring all echelon-zero and gcd restrictions
therefore gives the uniform overcount
\begin{equation}\label{fs-eq:matrix-count}
 2^k(C_JkqH_*)^{k^2},
 \qquad H_*=1\text{ in }\mathcal C_{\rm b},\quad
 H_*=2^a\text{ in a height bin}.
\end{equation}
For a height bin choose the first column attaining the maximum $L(D)$.
It has at least two parents and supplies \eqref{fs-eq:large-coeff-angular}
with $\|b_j/q\|\ge2^a$.  Its index is determined by $D$ and needs no
extra enumeration.

\emph{Three separate summations.}
Write
\[
 E_n=2^k(C_Jk)^{k^2}e^{C_Jk^2\log n},\qquad
 Q_n=\sum_{q\ge2}q^{-n+k^2},\qquad
 T_n=\sum_{a\ge1}2^{ak^2}
       \left(\frac{e+C_Jw_n}{2^{2a}}\right)^{p_n}.
\]
After division by
$\vol(\mathcal M)\max(1,V_{\rm fib})^{k-1-\chi}$, the three disjoint
contributions are bounded, respectively, by
\begin{align}
 \mathcal C_{\rm b}:&\quad E_n Q_n,
 &\bigcup_{a\ge1}\mathcal C_{\rm d,a}:&\quad E_n Q_nT_n,
 &\bigcup_{a\ge1}\mathcal C_{\rm h,a}:&\quad E_n T_n.
 \label{eq:arithmetic-master-sum}
\end{align}
The first two retain the exact Rogers bound $c_D\le q^{-n}$; the last
has $c_D=1$.  No suppression is discarded inside either infinite sum.
Lemma~\ref{lem:denominator-height-summation} gives, uniformly in $k$,
\[
 Q_n\le e^{O(k^2)}2^{-n},\qquad
 T_n\le e^{O(k^2\log n)}(e/4+o(1))^{n/2}.
\]
Indeed the ratio of consecutive height terms is
$2^{k^2-2p_n}=e^{-\Omega(n)}$, so the bin $[2,4)$ dominates.  Since
$Q_n<1$ for large $n$, the mixed contribution is at most $E_nT_n$;
the total is at most $E_n(Q_n+2T_n)$.  This proves the claimed bound.
In centered kernels all marked entries have already been counted by
\eqref{fs-eq:A-bound}.  If $q=1$ and only a marked entry is large, then
$B(D)\le1$ and the support relation places the matrix in the partition
or geometric class, so this does not leave an omitted arithmetic family.
\end{proof}

\begin{table}[t]
\centering\small
\begin{tabularx}{\textwidth}{@{}lXXX@{}}
\toprule
Disjoint family & Coefficient test & Available suppression & Normalized sum\\
\midrule
$\mathcal C_{\rm b}$ & $q\ge2$, $L<2$ & $q^{-n}$ & $E_nQ_n$\\
$\mathcal C_{\rm d,a}$ & $q\ge2$, $2^a\le L<2^{a+1}$
& $q^{-n}((e+C_Jw_n)/2^{2a})^{p_n}$ & $E_nQ_nT_n$ (sum over $a$)\\
$\mathcal C_{\rm h,a}$ & $q=1$, $B\ge2$, $2^a\le L<2^{a+1}$
& $((e+C_Jw_n)/2^{2a})^{p_n}$ & $E_nT_n$ (sum over $a$)\\
\bottomrule
\end{tabularx}
\caption[Denominator and height contributions.]{The three arithmetic families of
Lemma~\ref{fs-lem:arithmetic-direct}, with $a\ge1$.  A one-parent column
cannot attain $L\ge2$ on support; all bounded-height denominator matrices
are included in the first row.  Each normalized sum must be multiplied by
$\vol(\mathcal M)\max(1,V_{\rm fib})^{k-1-\chi}$.}
\label{tab:arithmetic-budget-ledger}
\end{table}

\begin{lemma}[Arithmetic shell-wide margin]\label{fs-rem:arithmetic-margin}
\[
 \mathcal N_\star\left(\frac e4\right)^{n/2}
 =\left(\frac e3\right)^{n/2}=e^{-0.049306\ldots n},
 \qquad
 \mathcal N_\star2^{-n}
 =\left(\frac1{\sqrt3}\right)^n=e^{-0.549306\ldots n}.
\]
\end{lemma}
\begin{proof}
Immediate from $\mathcal N_\star=(4/3)^{n/2}$.
\end{proof}

\subsubsection{The \texorpdfstring{$16/27$}{16/27} geometric penalty}\label{fs-sec:angular}
\begin{lemma}[Gram determinant optimization]\label{fs-lem:gram-opt}
Let $s\ge3$ and let $G\succeq0$ be an $s\times s$ correlation matrix.  Put
$\mathbf e=(1,\ldots,1)^{\mathsf T}\in\R^s$.  If
$\mathbf e^{\mathsf T}G\mathbf e=1$, then
\begin{equation}\label{fs-eq:gram-opt}
 \det G\le\frac1s\left(1+\frac1s\right)^{s-1}\le\frac{16}{27},
\end{equation}
with strict inequality in the last comparison for $s>3$.
\end{lemma}
\begin{proof}
The feasible set is permutation invariant and $\log\det$ is concave on
positive-definite matrices.  Symmetrizing gives
$G_\rho=(1-\rho)I+\rho\mathbf e\mathbf e^{\mathsf T}$.  The constraint gives
$\rho=-1/s$.  The eigenvalue along $\mathbf e$ is $1/s$ and the other $s-1$
eigenvalues are $1+1/s$, proving the first bound.  The expression decreases
for integer $s\ge3$ and equals $16/27$ at $s=3$.  Singular matrices follow
by continuity.
\end{proof}

\begin{lemma}[Spherical multi-parent penalty]\label{fs-lem:spherical-penalty}
Fix $H>0$.  Let $n-O(\log n)\le d_\perp\le n$ and
$3\le t\le2H\log n+1$.  For independent uniform
$\mathbf u_1,\ldots,\mathbf u_t\in S^{d_\perp-1}$, signs
$\varepsilon_i\in\{\pm1\}$, and $\delta_n,\eta_n=o(1)$, let
$c_i=\varepsilon_i a_i$ with $|a_i-1|\le\eta_n$.  Then
\begin{equation}\label{fs-eq:spherical-penalty}
 \Prb\left[\left\|\sum_{i=1}^t c_i\mathbf u_i\right\|^2
                  \in[1-\delta_n,1+\delta_n]\right]
 \le e^{O_H(t^2\log n)}
 \left(\frac{16}{27}+O(\delta_n+\eta_n)\right)^{d_\perp/2}.
\end{equation}
The constants are uniform over the signs and the exact radial factors.
\end{lemma}
\begin{proof}
Let $G$ be the Gram matrix and set
$\lambda=c^{\mathsf T}Gc/\|c\|^2$.  On the event,
\[
 0\le\lambda\le\frac{1+\delta_n}{t(1-\eta_n)^2}
                   =\frac{1+O(\delta_n+\eta_n)}t<1.
\]
The first bound in Lemma~\ref{fs-lem:rayleigh-det} gives
$\det G\le f_t(\lambda)$, where
$f_t(u)=u((t-u)/(t-1))^{t-1}$.  This function is increasing on $[0,1]$
and $0\le f_t'(u)\le e$ there.  Consequently
\[
 \det G\le f_t(1/t)+O(\delta_n+\eta_n)
 \le16/27+O(\delta_n+\eta_n)
\]
by Lemma~\ref{fs-lem:gram-opt}.  Apply the exact Gram density and
normalizer of Lemma~\ref{lem:gram-density}.  Its exponent is
$(d_\perp-t-1)/2$; increasing it to $d_\perp/2$ costs $e^{O(t)}$, which
is included in the prefactor.  This proof keeps the exact radial
coefficients and does not replace a weighted vector sum by an unweighted
sum with an uncontrolled additive error.
\end{proof}

\begin{lemma}[Classification and fiber powers for the four kernels]
\label{lem:complete-diagram-audit}
Use the four kernels \eqref{eq:raw-kernel-ct}--\eqref{eq:raw-kernel-ae}.
For $F_s^{\rm ct}$ and $F_s^{\rm at}$ take any $\alpha\in J\Subset(0,2)$.
For $F_s^{\rm ce}$ assume either $\alpha=\gamma$ in the near-unit range
or $\Delta_{\rm lp}(\alpha)\ge C_{\rm nr}(J)n^{-2}$, with $C_{\rm nr}$
larger than the low-parent support-error constant.  For $F_s^{\rm ae}$
assume $|1-\alpha^2/2|\ge\delta>0$.
Table~\ref{tab:four-kernel-audit} and the following rules classify every
admissible matrix with nonzero integral.

\emph{Tuple and rank.}  Table~\ref{tab:four-kernel-audit} specifies the
actual ordered Rogers tuple.  If its matrix has rank $r$, put $\chi=1$
in the centered cases and $\chi=0$ in the affine cases.  The exact free-fiber
count is $f=r-\chi$.  The identity has $f=s$; every other matrix has
$f\le s-1$.  For affine kernels the notation $\mathcal I_D(F)$ includes
the external Euclidean mark integral after applying Rogers to the displayed
tuple only.

\emph{Height-one columns.}  At $q=1$, after ruling out entries of magnitude
at least two in the free rows, let $B_j\in\{0,\pm1\}^f$ be the free part
of a nonpivot column and let $t_j=|\operatorname{supp}B_j|$.  A partition
column is exactly the one-parent reuse listed in the table.  Otherwise
the first nonpartition column $j_*$ has $t_{j_*}\ge3$, and its exact
transverse coefficient vector is
\[
 c_*=(B_{ij_*}\rho_i/\rho_0)_{i\in\operatorname{supp}B_{j_*}},
 \qquad \|c_*\|^2=t_{j_*}(1+O_J(w_n)).
\]
Its parents are distinct free integration variables with independent
uniform transverse directions conditional on their scalars.

\emph{Arithmetic columns.}  If $q\ge2$ or some free-row entry has magnitude
at least two, assign the matrix to exactly one of
$\mathcal C_{\rm b},\mathcal C_{\rm d,a},\mathcal C_{\rm h,a}$ in
Lemma~\ref{fs-lem:arithmetic-direct}.  The norm $L(D)$ and integer height
$B(D)$ have distinct roles.  A height-attaining one-parent column is
unsupported in every bin with $L(D)\ge2$.

\emph{Fiber masses.}  At a fixed mark define
\[
 W(\mathbf m)=
 \begin{cases}
 \omega(\mathbf d)\int\mathbf1_S(\mathbf v)
              \mathbf1_S(\mathbf v-\mathbf d)d\mathbf v,&\mathrm{ct},\\
 \int\mathbf1_T(\mathbf v)\mathbf1_S(\mathbf x-\mathbf v)
                \omega(\mathbf v)d\mathbf v,&\mathrm{ce},\\
 \omega(\mathbf d)\int\mathbf1_S(\mathbf v)
              \mathbf1_S(\mathbf d+\mathbf v)d\mathbf v,&\mathrm{at},\\
 \int\mathbf1_S(\mathbf v)\mathbf1_T(\mathbf x-\mathbf v)
                \omega(\mathbf x-\mathbf v)d\mathbf v,&\mathrm{ae}.
 \end{cases}
\]
For every correction, its complete fiber factor is bounded by
$W(\mathbf m)^f\le\max(1,V_{\rm fib})^{s-1}$ whenever
$W(\mathbf m)\le V_{\rm fib}$.  Thus the final exponent is $s-1=k-2$
for centered kernels and $s-1=k-1$ for unfolded-affine kernels.
\end{lemma}

\begin{table}[t]
\centering\footnotesize
\setlength{\tabcolsep}{3pt}
\renewcommand{\arraystretch}{1.14}
\begin{tabularx}{\textwidth}{@{}l>{\raggedright\arraybackslash}p{.22\textwidth}
>{\raggedright\arraybackslash}p{.23\textwidth}Xl@{}}
\toprule
Kernel & Ordered tuple; mark & One free fiber & Partition column; scalar test & $f$\\
\midrule
ct & $(\mathbf d,\mathbf x_1,\ldots,\mathbf x_s)$;
$\mathbf d$ is column 1
& $\mathbf v\in S$, $\mathbf v-\mathbf d\in S$
& $\mathbf v$ or $\mathbf d-\mathbf v$;
$2A+\sum B_i=1$ & $r-1$\\
ce & $(\mathbf x,\mathbf d_1,\ldots,\mathbf d_s)$;
$\mathbf x$ is column 1
& $\mathbf v\in T$, $\mathbf x-\mathbf v\in S$
& $\mathbf v$; $|2A+\alpha^2(\sum B_i-1)|\le C_J(t+1)w_n$ & $r-1$\\
at & $(\mathbf c_1,\ldots,\mathbf c_s)$;
$\mathbf d$ external
& $\mathbf v\in S$, $\mathbf d+\mathbf v\in S$
& $\mathbf v$; $\sum B_i=1$ & $r$\\
ae & $(\mathbf c_1,\ldots,\mathbf c_s)$;
$\mathbf x$ external
& $\mathbf v\in S$, $\mathbf x-\mathbf v\in T$
& $\mathbf v$; $\sum B_i=1$ under the gap assumption & $r$\\
\bottomrule
\end{tabularx}
\caption[Four-kernel tuples and fibers.]{Exact tuples and partition
columns under the hypotheses of Lemma~\ref{lem:complete-diagram-audit}.
The centered marked row is excluded from $B$ and from the free-fiber count.
The affine mark is not a column of the Rogers matrix.}
\label{tab:four-kernel-audit}
\end{table}

\begin{proof}
The tuple and scalar assertions are Lemmas~\ref{lem:four-raw-kernels}
and~\ref{lem:uniform-scalar-certificate}; the exact count $r-\chi$ and
the rank bound are Lemma~\ref{fs-lem:missing-fiber}.
The zero-, one-, and two-parent cases are exactly the enumeration in
Lemma~\ref{lem:four-kernel-relation-certificate}: ct allows equality and
reflection, whereas ce, at, and ae allow only equality under the stated
conditions.  A supported centered marked coefficient is uniquely determined
by its scalar equation, since its allowed interval has length $O_J(kw_n)<1$,
and has magnitude at most $C_Jk$.

First test $q$, then the largest free-row integer entry, and then scan the
height-one nonpivot columns for partition reuse.  These successive tests
are mutually exclusive and exhaustive.  The preceding enumeration forces
$t_{j_*}\ge3$ in the only remaining nonpartition case.
Han's echelon condition says that all those parent pivots precede $j_*$.
Keep all pivot factors, absorb earlier partition reuses into their own
pivot factors, and delete all later dependent-column factors.  The
remaining angular measure is a product by
Lemma~\ref{fs-lem:first-offending}; no conditional independence is asserted
after imposing a coupling constraint.  The exact coefficient vector above
now gives the geometric penalty by Lemma~\ref{fs-lem:spherical-penalty}.
The arithmetic alternatives and their one-parent edge case have been
summed separately in Lemma~\ref{fs-lem:arithmetic-direct}.

Finally, in a target kernel the raw weight is $\omega(\mathbf d)^s$.
Allocate one factor $\omega(\mathbf d)$ to each of the $f$ free pivots;
the leftover factor is $\omega(\mathbf d)^{s-f}\le1$.  In an endpoint
kernel keep the weight of each pivot occurrence and drop all other
occurrence weights.  Repeated weights can only reduce its mass since
$0\le\omega\le1$.  All kept weights are scalar about the marked axis.
The exact product integration therefore gives $W(\mathbf m)^f$ in every
row, including when $W<1$.  Restoring the mark gives the asserted volume
and fiber power.
\end{proof}

\begin{proposition}[Uniform nonpartition integral bound]
\label{prop:uniform-diagram-budget}
Fix $J\Subset(0,2)$ and constants $c_0,K,H>0$.  There are a constant
$C=C(J,c_0,K,H)$, an integer $n_0$, and one deterministic sequence
$\varepsilon_n=\varepsilon_n(J,c_0,K,H)\downarrow0$ such that the following
holds for $n\ge n_0$, all $1\le k\le2H\log n+1$, all radii
\[
 c_0\Rstar\le R\le e^{K\log^2n}\Rstar,
\]
and all target/source ratios in $J$ satisfying the kernel-specific
conditions of Lemma~\ref{lem:complete-diagram-audit}.  For the affine
endpoint take a fixed gap $\delta>0$; $C,n_0$ may also depend on $\delta$.
If the common marked domain is
$\mathcal M$ and all free weighted fiber masses are at most $V_{\rm fib}$,
then all supported nonpartition matrices are covered by the following
two bounds.  Every geometric nonpartition diagram $D$ satisfies
\begin{equation}\label{eq:uniform-geometric-diagram}
 \mathcal I_D(F)\le
 \vol(\mathcal M)\max(1,V_{\rm fib})^{k-2}
 e^{Ck^2\log n}
 \left(\frac{16}{27}+\varepsilon_n\right)^{n/2}.
\end{equation}
Moreover, after enlarging $C$ once, the finite sum over all such geometric
diagrams obeys the same displayed right-hand side.  For fixed $k$ this family
is finite: $q=1$, every unmarked entry lies in $\{0,\pm1\}$, and the marked
entries are then bounded by the longitudinal equation.
The sum over all supported arithmetic/large-coefficient diagrams satisfies
\begin{equation}\label{eq:uniform-arithmetic-diagrams}
 \sum_{D\ {\rm arithmetic}}\mathcal I_D(F)\le
 \vol(\mathcal M)\max(1,V_{\rm fib})^{k-2}
 e^{Ck^2\log n}
 \left[
   \left(\frac e4+\varepsilon_n\right)^{n/2}+2^{-n}
 \right].
\end{equation}
The same $\varepsilon_n$ works simultaneously for target kernels, endpoint
kernels, and all raw orders $s\le2h$.  For an unfolded-affine kernel of
Rogers tuple order $k$, both right-hand sides remain valid with the factor
$\max(1,V_{\rm fib})^{k-2}$ replaced explicitly by
$\max(1,V_{\rm fib})^{k-1}$; no other term changes.
\end{proposition}
\begin{proof}
Lemma~\ref{lem:complete-diagram-audit} and
Table~\ref{tab:four-kernel-audit} make the
classification and the exact fiber powers explicit.  We now collect the
uniform estimates; partition contributions are evaluated in
Section~\ref{sec:shell-moments}.

\emph{Step 1: the actual column and its product measure.}
Lemma~\ref{lem:four-raw-kernels} lists every kernel to be considered.
Lemma~\ref{lem:supported-han-kernel} identifies a Han matrix column with the
literal output vector in that kernel.  In the geometric class choose its
first nonpartition column $j_*$.  By
Lemma~\ref{fs-lem:first-offending}, all factors before $j_*$ split over
individual free pivots.  Conditional on their exact scalar coordinates,
the transverse directions entering $j_*$ consequently have product uniform
measure.  Every dependent-column factor after $j_*$ lies in $[0,1]$ and can
be dropped; every pivot factor is retained.

\emph{Step 2: one geometric diagram.}
Lemma~\ref{lem:complete-diagram-audit} says that $j_*$ combines
$t\ge3$ distinct free pivots.  By
Lemma~\ref{lem:uniform-scalar-certificate}, after division by the common transverse radius $\rho_0$, its exact
nonzero coefficients are signs times $1+O_J(w_n)$ and output support
forces the squared norm into $[1-C_Jw_n,1+C_Jw_n]$.
Lemma~\ref{fs-lem:spherical-penalty}, with transverse dimension $n-1$, now
gives
\[
 e^{C_1k^2\log n}
 \left(\frac{16}{27}+C_2kw_n\right)^{n/2}.
\]
The exponent shifts from $n-1$ to the Gram-density exponent are
$O(k)$ and are included in the displayed prefactor.
Equation~\eqref{eq:first-offending-complete-bound} integrates the scalar
measures and the mark.  Apply it with $V_i=V_{\rm fib}$ for every free
pivot.  A centered nonidentity diagram has at most $k-2$ unmarked free
pivots by Lemma~\ref{fs-lem:missing-fiber}, which gives
\eqref{eq:uniform-geometric-diagram} for one $D$.

\emph{Step 3: summing the geometric diagrams.}
There are at most $2^k$ pivot patterns.  Once a pattern is fixed, the
unmarked part of a height-one matrix has at most $k^2$ entries, each in
$\{0,\pm1\}$.  In a centered kernel the longitudinal support relation fixes each
marked-row entry and bounds it by $C_Jk$; in an affine kernel there is no
marked row.  Hence the number of supported geometric matrices is at most
\[
 2^k3^{k^2}(C_Jk)^k\le e^{C_3k^2\log(2k)}.
\]
The choice of $j_*$ is already determined by a matrix, and allowing another
factor $k$ only enlarges this bound.  This proves the summed geometric
estimate after increasing $C$.

\emph{Step 4: arithmetic diagrams.}
Lemma~\ref{fs-lem:arithmetic-direct} treats centered and affine
kernels in one statement, with exponent $k-1-\chi$.  Its three disjoint
contributions in \eqref{eq:arithmetic-master-sum} are $E_nQ_n$,
$E_nQ_nT_n$, and $E_nT_n$.  The denominator and height estimates of
Lemma~\ref{lem:denominator-height-summation} give
\eqref{eq:uniform-arithmetic-diagrams}, with exponent $k-2$ when
$\chi=1$ and $k-1$ when $\chi=0$.  The latter free-fiber count also
applies to Step 2 and proves the affine geometric estimate.

\emph{Step 5: one uniform error sequence.}
All uses of compactness occur in the four explicit scalar formulas of
Lemma~\ref{lem:uniform-scalar-certificate}.  Their constants depend only on
$J$; scaling by $R$ removes the radius, and the exact scalar Jacobians are
never compared pointwise.  In the height-one geometric class, uniformly for
$k\le2H\log n+1$, the accumulated coefficient and shell error is at most
$C_4kw_n$.  Arithmetic diagrams use the exact coefficients and uniform
relative squared-norm error $O_J(w_n)$ of
Lemma~\ref{lem:uniform-scalar-certificate}.  Gram normalization,
dimension shifts, pivot patterns, marked coefficients, and matrix
enumeration together cost at most $e^{C_5k^2\log n}$.  We may therefore take
the nonincreasing envelope of
\[
 \varepsilon_n=C_6\frac{\log n}{n};
\]
it dominates $C_4kw_n$, tends to zero, and is independent of
$D,k,R$, the raw order, and the chosen kernel.  The upper bound on $R$ is
used only to keep the later continuum fiber envelopes in their stated
range; the normalized diagram constants themselves do not depend on it.
This completes both estimates.  Every later $o(1)$ inside the bases
$16/27$ and $e/4$ refers to this common deterministic choice.
\end{proof}

\begin{corollary}[Shell-scale geometric penalty]\label{fs-cor:angular-penalty}
Fix $H>0$ and $1\le h\le H\log n$.  Consider a geometric nonpartition diagram under the hypotheses of
Proposition~\ref{prop:uniform-diagram-budget}, with at most $2h+1$
Rogers columns.  The conditional angular integral of the shell indicator
for its first nonpartition column, with respect to the product uniform
measure of Lemma~\ref{fs-lem:first-offending}, is at most
\begin{equation}\label{fs-eq:angular-penalty}
 \exp(O_H(h^2\log n))
 \left(\frac{16}{27}+o(1)\right)^{n/2},
\end{equation}
uniformly over the mark and every supported scalar tuple.  The full diagram
integral is bounded by restoring the mark integral and fiber masses in
\eqref{eq:first-offending-complete-bound}, as in
Proposition~\ref{prop:uniform-diagram-budget}.
\end{corollary}
\begin{proof}
The first nonpartition column combines at least three free pivots.  In this
height-one geometric class, shell scalar errors accumulate to
$O(h w_n)=O_H(\log n/n^2)=o(1)$, while transverse dimension is $n-1$.
Apply Lemma~\ref{fs-lem:spherical-penalty} to its conditional angular
integral, and Lemma~\ref{fs-lem:first-offending} to the full integral.
\end{proof}

\section{High moments for marked shell-incidence counts}\label{sec:shell-moments}
\label{sec:rogers-shell}
This section specializes the diagram bounds of Section~\ref{sec:rogers-engine}
to target multiplicities and endpoint degrees.  We state its two principal
moment estimates first; their conversion into simultaneous pointwise
regularity is carried out in Section~\ref{sec:growing-rogers}.

\subsection{Uniform centered and affine moment bounds}

We first state the uniform bounds in their intrinsic fixed-shell form.
The endpoint population is always $\cL\cap S_R$; the parameter $\alpha$ only
selects the radial band in which its differences are observed.  For
$\alpha\in(0,2)$ let $T_{\alpha,R}:=S_{\alpha R}$ and put
\[
 K_R(\mathbf d):=\vol(S_R\cap(S_R+\mathbf d)),\qquad
 m_*:=\inf_{\mathbf d\in T_{\alpha,R}}K_R(\mathbf d),\qquad
 \omega(\mathbf d):=
 \begin{cases}m_*/K_R(\mathbf d),&\mathbf d\in T_{\alpha,R},\\
 0,&\mathbf d\notin T_{\alpha,R}.
 \end{cases}
\]
Write $P:=\vol(S_R)$ and $Q:=\vol(T_{\alpha,R})$.  Given fixed
$J\Subset(0,2)$ and constants $c_0,K,H>0$, fix the common sequence
$\varepsilon_n=\varepsilon_n(J,c_0,K,H)$ from
Proposition~\ref{prop:uniform-diagram-budget} and put
\begin{equation}\label{eq:diagram-correction-budget}
 \Xi_n:=\left(\frac{16}{27}+\varepsilon_n\right)^{n/2}
       +\left(\frac e4+\varepsilon_n\right)^{n/2}+2^{-n}.
\end{equation}
For $\mathbf x\in S_R$ define
\[
 a_R^\omega(\mathbf x):=\int_{T_{\alpha,R}}
 \mathbf 1_{S_R}(\mathbf x-\mathbf d)\omega(\mathbf d)\,d\mathbf d,
 \qquad a^*:=\sup_{\mathbf x\in S_R}a_R^\omega(\mathbf x).
\]
Only relations involving at most two free unmarked pivots must be excluded
before the three-parent Gram penalty applies.  We use the finite low-parent
gap $\Delta_{\rm lp}$ defined in~\eqref{eq:low-parent-gap}; its harmless
cutoff is $C_{\rm lp}=10$.

\begin{proposition}[Uniform marked shell moments]\label{thm:intrinsic-marked-moments}
Fix a compact interval $J\Subset(0,2)$ and constants $c_0,K,H>0$.
Uniformly for $\alpha\in J$, for
$c_0\le R/\Rstar\le\exp(K\log^2n)$, and for integers
$1\le h\le H\log n$, one has
\begin{align}
 &\E_{\cL}\sum_{\mathbf d\in\cL\cap T_{\alpha,R}}
 \left|\omega(\mathbf d)r_{\cL,R}(\mathbf d)-m_*\right|^{2h}\notag\\
 &\quad\le Qe^{O(h\log(2h))}m_*\max(1,m_*)^{h-1}
 +Qe^{O(h^2\log n)}\max(1,m_*)^{2h-1}\Xi_n.
 \label{eq:intrinsic-target-moment}
\end{align}
Let $C_{\rm nr}(J)$ be as in
Lemma~\ref{lem:generic-endpoint-low-parent}.  If in addition
$\Delta_{\rm lp}(\alpha)\ge C_{\rm nr}(J)n^{-2}$, then, with
\[
 D^\omega_{\cL,R,\alpha}(\mathbf x)
 :=\sum_{\mathbf d\in\cL\cap T_{\alpha,R}}
 \mathbf 1_{S_R}(\mathbf x-\mathbf d)\omega(\mathbf d),
\]
\begin{align}
 &\E_{\cL}\sum_{\mathbf x\in\cL\cap S_R}
 \left|D^\omega_{\cL,R,\alpha}(\mathbf x)
             -a_R^\omega(\mathbf x)\right|^{2h}\notag\\
 &\quad\le Pe^{O(h\log(2h))}a^*\max(1,a^*)^{h-1}
 +Pe^{O(h^2\log n)}\max(1,a^*)^{2h-1}\Xi_n.
 \label{eq:intrinsic-endpoint-moment}
\end{align}
All implicit constants are uniform on $J$.
\end{proposition}

For a random affine lattice, keep the same shell, difference band, and
flattening weight.  For
$\mathbf d\in(\boldsymbol\tau+\cL)\cap T_{\alpha,R}$ and
$\mathbf x\in(\boldsymbol\tau+\cL)\cap S_R$, define
\begin{align}
 r^{\rm aff}_{\cL,\boldsymbol\tau,R}(\mathbf d)
 &:=\sum_{\mathbf c\in\cL}
   \mathbf1_{S_R}(\mathbf c)\mathbf1_{S_R}(\mathbf d+\mathbf c),
 \label{eq:affine-representation}\\
 D^{{\rm aff},\omega}_{\cL,\boldsymbol\tau,R,\alpha}(\mathbf x)
 &:=\sum_{\mathbf c\in\cL}
   \mathbf1_{S_R}(\mathbf c)\mathbf1_{T_{\alpha,R}}(\mathbf x-\mathbf c)
   \omega(\mathbf x-\mathbf c).
 \label{eq:affine-endpoint-degree}
\end{align}
The first variable counts representations
$\mathbf d=\mathbf x-\mathbf c$ with an affine endpoint
$\mathbf x\in(\boldsymbol\tau+\cL)\cap S_R$ and a centered endpoint
$\mathbf c\in\cL\cap S_R$.  Its flattened continuum mean is $m_*$.  The
continuum mean of the weighted endpoint degree is the same function
$a_R^\omega(\mathbf x)$ defined before
Proposition~\ref{thm:intrinsic-marked-moments}.

\begin{proposition}[Uniform random-affine shell moments]
\label{thm:affine-marked-moments}
Fix a compact interval $J\Subset(0,2)$ and constants $c_0,K,H>0$.  Uniformly
for $\alpha\in J$, for $c_0\le R/\Rstar\le\exp(K\log^2n)$, and for integers
$1\le h\le H\log n$, one has
\begin{align}
 &\E_{\cL,\boldsymbol\tau}
 \sum_{\mathbf d\in(\boldsymbol\tau+\cL)\cap T_{\alpha,R}}
 \left|\omega(\mathbf d)
 r^{\rm aff}_{\cL,\boldsymbol\tau,R}(\mathbf d)-m_*\right|^{2h}
 \notag\\
 &\quad\le
 Qe^{O(h\log(2h))}m_*\max(1,m_*)^{h-1}
 +Qe^{O(h^2\log n)}\max(1,m_*)^{2h-1}\Xi_n.
 \label{eq:affine-target-moment}
\end{align}
If, in addition,
\begin{equation}\label{eq:affine-endpoint-gap}
 \delta_J:=\inf_{\alpha\in J}\left|1-\frac{\alpha^2}{2}\right|>0,
\end{equation}
then
\begin{align}
 &\E_{\cL,\boldsymbol\tau}
 \sum_{\mathbf x\in(\boldsymbol\tau+\cL)\cap S_R}
 \left|D^{{\rm aff},\omega}_{\cL,\boldsymbol\tau,R,\alpha}(\mathbf x)
       -a_R^\omega(\mathbf x)\right|^{2h}\notag\\
 &\quad\le
 Pe^{O(h\log(2h))}a^*\max(1,a^*)^{h-1}
 +Pe^{O(h^2\log n)}\max(1,a^*)^{2h-1}\Xi_n.
 \label{eq:affine-endpoint-moment}
\end{align}
All implicit constants are uniform on $J$.  After unfolding, the largest
tuple presented to Rogers' formula has order $2h$, one less than in the
corresponding centered marked-target calculation.
\end{proposition}

Str\"ombergsson--S\"odergren prove Gaussian and functional limits at
subexponential volume, as well as fixed-order moment convergence at
exponential volume below an order-dependent growth threshold
\cite[Theorems~1.3, 1.5, and~1.7]{StrombergssonSodergren2016}.
The bounds above concern individual marked lens and endpoint fibers at
moment order proportional to $\log n$, with one correction bound that
survives summation over every target in a complete shell.

We prove these two propositions below.  The near-unit geometry and partition
calculations come first, followed by the general centered classification and
the affine unfolding argument.  For the near-unit specialization put
\[
 R_\ell:=\sqrt{\frac43}\,e^{1/\ell}\Rstar.
\]

\subsection{Incidence geometry and continuum scales}\label{fs-sec:geometry}
For fixed relative difference length $\alpha\in(0,2)$ set
\[
 A(\alpha):=\alpha\sqrt{1-\frac{\alpha^2}{4}}.
\]
Note that $A(1)^n=(3/4)^{n/2}=\mathcal N_\star^{-1}$.
For $\mathbf x\in S_R$, write
\[
 a_R^{\wt}(\mathbf x):=
 \int_{T_R}\1_{S_R}(\mathbf x-\mathbf d)\wt_R(\mathbf d)\,d\mathbf d.
\]

\begin{lemma}[Near-unit difference factor and continuum scales]\label{fs-lem:continuum-scales}
Fix a constant $K>0$.  Uniformly for
$R_\ell\le R\le\exp(K\log^2n)\Rstar$, define
\[
 p_\gamma(R):=\frac1{\vol(S_R)^2}
 \int_{S_R}\int_{S_R}\1_{T_R}(\mathbf{x}-\mathbf{y})\,d\mathbf{x}\,d\mathbf{y}.
\]
Then
\begin{equation}\label{fs-eq:pairprob}
 p_\gamma(R)=\mathcal N_\star^{-1}\exp\!\left(
 -\frac{2n}{3\ell}+O\!\left(\frac n{\ell^2}+\log n\right)\right).
\end{equation}
Moreover, uniformly for $\mathbf{d}\in T_R$ and $\mathbf{x}\in S_R$,
\begin{align}
 \mu_R(\mathbf{d})&=\mu_*(R)(1+O(1/n)),\label{fs-eq:mu-uniform}\\
 a_R^{\wt}(\mathbf{x})&=\bar a_R(1+O(1/n)),\label{fs-eq:a-uniform}
\end{align}
where
\begin{align}
 \mu_*(R)&=\frac{\vol(S_R)^2}{\vol(T_R)}p_\gamma(R)(1+O(1/n)),\label{fs-eq:mu-scale}\\
 \bar a_R&:=\frac{\mu_*(R)\vol(T_R)}{\vol(S_R)}
 =(1+O(1/n))\vol(S_R)p_\gamma(R).\label{fs-eq:a-scale}
\end{align}
Finally,
\begin{align}
 \log\bar a_R&\ge\frac n{3\ell}-O\!\left(\frac n{\ell^2}+\log n\right),\label{fs-eq:a-min}\\
 \log\mu_*(R)&\ge\frac{4n}{3\ell}-O\!\left(\frac n{\ell^2}+\log n\right).\label{fs-eq:mu-min}
\end{align}
\end{lemma}
\begin{proof}
Condition on source radii $r_1,r_2=R(1+O(w_n))$.  If $u=\langle\mathbf{x},\mathbf{y}\rangle/(r_1r_2)$, then the spherical density of $u$ is a polynomial factor times $(1-u^2)^{(n-3)/2}$.  Requiring $\|\mathbf{x}-\mathbf{y}\|=\gamma R(1+O(w_n))$ forces
\[
 u=1-\frac{\gamma^2}{2}+O(w_n)
\]
and an $O(w_n)$ interval in $u$.  Hence, including the target radial Jacobian, the exponential factor is
\[
 \left(\gamma\sqrt{1-\frac{\gamma^2}{4}}\right)^n=A(\gamma)^n,
\]
up to $e^{O(\log n)}$.  Since
\[
 \log A(1-1/\ell)=\log A(1)-\frac{2}{3\ell}+O(\ell^{-2}),
\]
we obtain~\eqref{fs-eq:pairprob}.

Double counting gives
\[
 \int_{T_R}\mu_R(\mathbf{d})\,d\mathbf{d}=\vol(S_R)^2p_\gamma(R).
\]
By Lemma~\ref{lem:thin-shell-coarea}, the logarithm of the equal-radius lens volume has derivative $O(n)$ on the compact ratio range in question.  Across the relative target width $w_n=n^{-2}$, $\mu_R(\mathbf{d})$ therefore varies by $1+O(nw_n)=1+O(1/n)$, proving~\eqref{fs-eq:mu-uniform} and~\eqref{fs-eq:mu-scale}.  Similarly
\[
 \int_{S_R}a_R^{\wt}(\mathbf{x})\,d\mathbf{x}
 =\int_{T_R}\wt_R(\mathbf{d})\mu_R(\mathbf{d})\,d\mathbf{d}
 =\mu_*(R)\vol(T_R).
\]
Rotational symmetry makes $a_R^{\wt}(\mathbf{x})$ a function only of $\|\mathbf{x}\|$.  Lemma~\ref{lem:thin-shell-coarea}, together with $\wt_R=1+O(1/n)$ from~\eqref{fs-eq:mu-uniform}, gives an $O(n)$ logarithmic derivative for the weighted cap integral on the present compact ratio range, while the endpoint shell has relative width $w_n=n^{-2}$.  Hence $a_R^{\wt}(\mathbf{x})$ varies by $1+O(nw_n)=1+O(1/n)$ across $S_R$, giving~\eqref{fs-eq:a-uniform}; the displayed integral identity then gives~\eqref{fs-eq:a-scale}.
Indeed, in the exact bipolar formula the weight depends on the target radius
but not on the separation variable $\|\mathbf x\|$; multiplying the positive
integrand by this fixed nonnegative weight leaves the pointwise logarithmic
derivative bound unchanged.

Finally
\[
 \log\vol(S_R)=n\log(R/\Rstar)+O(\log n),
\qquad
 \log\frac{\vol(T_R)}{\vol(S_R)}=n\log\gamma+O(1/n),
\]
and at $R=R_\ell$,
\[
 \log\vol(S_R)=\frac12\log(4/3)n+\frac n\ell+O(\log n),
\qquad
 n\log\gamma=-\frac n\ell+O(n/\ell^2).
\]
Substitution into~\eqref{fs-eq:pairprob}--\eqref{fs-eq:a-scale} gives~\eqref{fs-eq:a-min}--\eqref{fs-eq:mu-min}.  Larger $R$ multiplies the relevant shell/fiber volumes by the same positive radial scale factor and only improves these lower bounds.
\end{proof}

\subsection{Target multiplicities: exact partition term and corrections}\label{fs-sec:target-moment}
Fix a difference-band query and abbreviate $S=S_R$, $T=T_R$, $\mu_*=\mu_*(R)$ and $\wt=\wt_R$.  For $\mathbf{d}\in\cL\cap T$ define
\[
 X_\cL^{\wt}(\mathbf{d})
 =\wt(\mathbf{d})\sum_{\mathbf{x}\in\cL}\1_S(\mathbf{x})\1_S(\mathbf{x}-\mathbf{d}),
\]
and
\[
 \mathfrak M_T:=\E_\cL\sum_{\mathbf{d}\in\cL\cap T}
 |X_\cL^{\wt}(\mathbf{d})-\mu_*|^{2h}.
\]

\begin{lemma}[Growing marked target moment]\label{fs-thm:target-high-moment}
For an absolute $C>0$,
\begin{align}
 \mathfrak M_T\le{}&\vol(T)e^{Ch\log(2h)}\mu_*\max(1,\mu_*)^{h-1}\notag\\
 &+\vol(T)e^{Ch^2\log n}\max(1,\mu_*)^{2h-1}
 \left(\frac{16}{27}+o(1)\right)^{n/2}\notag\\
 &+\vol(T)e^{Ch^2\log n}\max(1,\mu_*)^{2h-1}
 \left[\left(\frac e4+o(1)\right)^{n/2}+2^{-n}\right].\label{fs-eq:target-high-moment}
\end{align}
\end{lemma}
We prove this bound after Lemmas~\ref{fs-lem:target-relations}--%
\ref{fs-lem:target-corrections}, which identify and bound its raw-moment terms.

\begin{lemma}[Target $q=1$ relation classification]\label{fs-lem:target-relations}
Let a supported $q=1$ dependent parent column with nonmarked coefficients in $\{0,\pm1\}$ be
\[
 \mathbf{x}'=A\mathbf{d}+\sum_{j=1}^tB_j\mathbf{x}_j,
 \qquad A\in\Z,\quad B_j\in\{\pm1\}.
\]
Then
\begin{equation}\label{fs-eq:target-longitudinal}
 2A+\sum_{j=1}^tB_j=1.
\end{equation}
If $t=1$, the only possibilities are equality $\mathbf{x}'=\mathbf{x}_1$ and reflection $\mathbf{x}'=\mathbf{d}-\mathbf{x}_1$; $t=2$ is impossible; $t=0$ is impossible.  Hence every supported nonpartition column of this coefficient type has $t\ge3$.
\end{lemma}
\begin{proof}
Apply part~\textup{(i)} of
Lemma~\ref{lem:four-kernel-relation-certificate} to the centered target
kernel, with $\alpha=\gamma$ and $k=2h+1=O(\log n)$.
Its integer longitudinal identity is~\eqref{fs-eq:target-longitudinal},
and its zero-, one-, and two-parent classification gives the remaining
assertions.
\end{proof}

\begin{lemma}[Target partition identification]\label{fs-lem:target-affine-partitions}
Fix the marked target $\mathbf{d}$.  The supported $q=1$ diagrams in which every dependent parent column contains at most one nonzero nonmarked coefficient are in bijection with set partitions of the occurrence slots, together with equality/reflection choices within each block.  Their Rogers coefficient is one.  Their complete raw-moment contribution is exactly the moment contribution of
\begin{equation}\label{fs-eq:target-poisson-variable}
 Y_{\mathbf{d}}=2\wt(\mathbf{d})Z_{\mathbf{d}},
 \qquad Z_{\mathbf{d}}\sim\operatorname{Pois}(\mu_R(\mathbf{d})/2).
\end{equation}
More explicitly, write $\mu=\mu_R(\mathbf d)$, $w=\omega(\mathbf d)$,
and let $\mathfrak P_m(\mathbf d)$ be the partition contribution to the
algebraic $m$-th moment centered at $w\mu$.  For every $1\le m\le2h$,
\begin{equation}\label{eq:target-centered-partition-bijection}
 \mathfrak P_m(\mathbf d)
 =w^m\!\sum_{\substack{\pi\in\mathcal P([m])\\
                          |B|\ge2\ \forall B\in\pi}}
       2^{m-|\pi|}\mu^{|\pi|}.
\end{equation}
In particular, the exponent of $\mu$ is at most $h$ when $m=2h$.
\end{lemma}
\begin{proof}
For a one-parent column $\mathbf{x}'=A\mathbf{d}+B\mathbf{x}$,
orthogonal projection to $\mathbf{d}^{\perp}$ gives
$\mathbf{x}'_{\perp}=B\mathbf{x}_{\perp}$.  Both transverse lengths
are $\sqrt{R^2-\|\mathbf{d}\|^2/4}\,(1+O(w_n))$, uniformly in the
fixed difference band.  Hence the integer $B$ has $|B|=1$ for all
sufficiently large $n$; a column with no parent is unsupported.
Lemma~\ref{fs-lem:target-relations} now shows that the column is exactly
$\mathbf{x}$ or $\mathbf{d}-\mathbf{x}$.
The involution $\mathbf{x}\mapsto\mathbf{d}-\mathbf{x}$ preserves the lens and has no fixed point because $\|\mathbf{d}\|/2<R(1-w_n)$.  Thus every free parent pivot specifies one unordered pair $\{\mathbf{x},\mathbf{d}-\mathbf{x}\}$ and every later occurrence in the same block chooses one of its two orientations.  Conversely every such signed block pattern is admissible.  This is the shell-specific marked version of the signed partition matrices in Han~\cite[Proposition~2.4]{Han2024}; $q=1$ gives $c_D=1$ by~\eqref{eq:rogers-coeff}.  A block of size $b$ contributes $2^{b-1}\wt(\mathbf{d})^b\mu_R(\mathbf{d})$, exactly the $b$-th cumulant of~\eqref{fs-eq:target-poisson-variable}.

For completeness, here is the marked-row reconstruction which is essential
for the reflected orientation.  Normalize the first occurrence in each block
to have sign $+1$.  If occurrence $j$ belongs to parent row $p(j)$ and has
orientation $\varepsilon_j\in\{\pm1\}$, set
\begin{equation}\label{eq:target-marked-reconstruction}
 B_{ij}=\varepsilon_j\mathbf1_{\{i=p(j)\}},
 \qquad A_j=\frac{1-\varepsilon_j}{2}.
\end{equation}
Then $A_j\in\{0,1\}$ and
$2A_j+\sum_iB_{ij}=1$.  For $\varepsilon_j=+1$ this is the equality column
$(A_j,B_{p(j),j})=(0,1)$; for $\varepsilon_j=-1$ it is the reflection column
$(1,-1)$ and hence \emph{both} its marked and parent entries are nonzero.
Conversely the one-parent classification in
Lemma~\ref{fs-lem:target-relations} recovers exactly
\eqref{eq:target-marked-reconstruction}.  Thus the one-parent condition is a
condition on the \emph{unmarked} submatrix; adding the marked row by
\eqref{eq:target-marked-reconstruction} is a bijection, not an omission of the
reflection coefficient.  The $b-1$ nonpivot orientations in a block give
exactly $2^{b-1}$ choices.  A raw partition $\pi$ therefore contributes
$2^{m-|\pi|}w^m\mu^{|\pi|}$.  In the centered expansion each singleton
can be supplied either by its raw one-point block or by the deterministic
subtraction, both of magnitude $w\mu$.  A partition with $a$ singleton
slots has the resulting coefficient
$\sum_{j=0}^a\binom aj(-1)^{a-j}$, which vanishes for $a>0$ and equals
one for $a=0$.  This proves
\eqref{eq:target-centered-partition-bijection}; a partition without
singletons has at most $h$ blocks when $m=2h$.
\end{proof}

\subsubsection{A fourth-moment example}
Fix $\mathbf d$, and abbreviate $\mu=\mu_R(\mathbf d)$ and
$w=\omega(\mathbf d)$ in this paragraph only.  The ordered tuple is
$(\mathbf d,\mathbf x_1,\mathbf x_2,\mathbf x_3,\mathbf x_4)$.
Choose the least occurrence of each block as its pivot; later slots in that
block are either the pivot or its reflection, using
\eqref{eq:target-marked-reconstruction}.  There are no further choices of
marked-row entries.  For example the partition $\{1,3\}\mid\{2,4\}$ has
the four matrices
\[
 D_{\varepsilon_3,\varepsilon_4}=
 \begin{pmatrix}
 1&0&0&(1-\varepsilon_3)/2&(1-\varepsilon_4)/2\\
 0&1&0&\varepsilon_3&0\\
 0&0&1&0&\varepsilon_4
 \end{pmatrix},\qquad \varepsilon_3,\varepsilon_4\in\{\pm1\}.
\]
Each has coefficient one and integral $w^4\mu^2$ at the fixed mark.
The complete raw fourth-moment partition table is
\[
\begin{array}{c|r|r|l}
 \text{block sizes}&\text{partitions}&\text{orientations per partition}
    &\text{total contribution}\\ \hline
 1+1+1+1&1&1&w^4\mu^4\\
 2+1+1&6&2&12w^4\mu^3\\
 2+2&3&4&12w^4\mu^2\\
 3+1&4&4&16w^4\mu^2\\
 4&1&8&8w^4\mu
\end{array}.
\]
Centering is performed after the raw Rogers expansions.  In the binomial
expansion about $w\mu$, any full partition with $a$ singleton slots has
coefficient $\sum_{j=0}^a\binom aj(-1)^{a-j}$: each such slot can come
either from a raw singleton or from the deterministic centering term,
both of mass $w\mu$.  This is zero unless $a=0$.  Hence the surviving
partitions are exactly
\[
 \{1,2\}\mid\{3,4\},\quad
 \{1,3\}\mid\{2,4\},\quad
 \{1,4\}\mid\{2,3\},\quad \{1,2,3,4\},
\]
and their complete contribution is
\[
 3(2w^2\mu)^2+8w^4\mu
 =12w^4\mu^2+8w^4\mu.
\]
This is an identity for the partition contribution, not for the entire
lattice moment.  Every raw correction with $s\le4$ occurrences has at
most $s-1$ free fibers; multiplying by its centering factor of degree
$4-s$ leaves the common upper power $\max(1,\mu_*)^3$.  In general the
same two arguments give at most $h$ blocks for the partition contribution
and the upper power $2h-1$ for corrections to the centered $2h$-th moment.

\begin{lemma}[Target partition contribution]\label{fs-lem:target-main}
The total target partition contribution satisfies
\begin{equation}\label{fs-eq:target-main-int}
 \mathfrak M_T^{\rm part}
 \le\vol(T)\exp(C h\log(2h))\mu_*\max(1,\mu_*)^{h-1}
\end{equation}
for an absolute $C$.
\end{lemma}
\begin{proof}
For~\eqref{fs-eq:target-poisson-variable},
\[
 \E Y_{\mathbf{d}}=\wt(\mathbf{d})\mu_R(\mathbf{d})=\mu_*,
\]
and for every $b\ge2$,
\[
 \kappa_b(Y_{\mathbf{d}})
 =2^{b-1}\wt(\mathbf{d})^b\mu_R(\mathbf{d})
 =2^{b-1}\wt(\mathbf{d})^{b-1}\mu_*
 \le2^{b-1}\mu_*.
\]
By Lemma~\ref{lem:poisson-cumulants}, the centered $2h$-th moment contains
exactly set partitions of $[2h]$ with no singleton block.  A partition with
$j$ blocks has $1\le j\le h$ and its cumulant product is at most
\[
 2^{2h-j}\mu_*^j
 \le 2^{2h}\mu_*\max(1,\mu_*)^{h-1}.
\]
There are at most $(2h)^{2h}$ such partitions, and
$2^{2h}(2h)^{2h}\le e^{4h\log(2h)}$ for every integer $h\ge1$.
Integration over $T$ proves the bound, also when $\mu_*\le1$.
In particular, at $h=1$ the exact partition variance is
$2\wt(\mathbf d)^2\mu_R(\mathbf d)=2\wt(\mathbf d)\mu_*$;
the factor $e^{Ch\log(2h)}$ includes this order-two constant.
\end{proof}

\begin{lemma}[Target correction diagrams]\label{fs-lem:target-corrections}
For every raw $s$-occurrence target moment, $1\le s\le2h$, the geometric
nonpartition contribution is at most
\begin{equation}\label{fs-eq:target-r2-raw-stmt}
 \vol(T)\exp(O(h^2\log n))\max(1,\mu_*)^{s-1}
 \left(\frac{16}{27}+o(1)\right)^{n/2},
\end{equation}
and the arithmetic/large-coefficient contribution is at most
\begin{equation}\label{fs-eq:target-r1-raw-stmt}
 \vol(T)\exp(O(h^2\log n))\max(1,\mu_*)^{s-1}
 \left[2^{-n}+\left(\frac e4+o(1)\right)^{n/2}\right].
\end{equation}
\end{lemma}
\begin{proof}
Fix $\mathbf{d}$.  By Lemma~\ref{lem:fiber-polar}, conditional on scalar coordinates, each free parent has an independent uniform transverse direction in $\mathbf{d}^{\perp}$.  Lemma~\ref{fs-lem:target-relations} says the first nonpartition $q=1$, $\{0,\pm1\}$ column combines $t\ge3$ such directions; Corollary~\ref{fs-cor:angular-penalty} gives the factor in~\eqref{fs-eq:target-r2-raw-stmt}.  There are at most $3^{O((s+1)^2)}=\exp(O(h^2))$ such signed matrices.

A rank-$r$ correction has at most $r-1\le s-1$ free parent pivots.  By~\eqref{fs-eq:mu-min}, $\mu_R(\mathbf{d})\ge\mu_*>1$ for all sufficiently large $n$, uniformly in the radius range.  The raw moment carries the common factor $\wt(\mathbf{d})^s$.  Using $\wt=\mu_*/\mu_R$,
\[
 \wt(\mathbf{d})^s\mu_R(\mathbf{d})^{r-1}
 \le\wt^s\mu_R^{s-1}
 =\wt\,\mu_*^{s-1}\le\mu_*^{s-1}.
\]
This is the exact missing-fiber weight calculation; no unweighted $\mu_{\max}$ substitution is needed.

For arithmetic diagrams apply Lemma~\ref{fs-lem:arithmetic-direct} with marked domain $T$ and parent-fiber mass $\mu_R(\mathbf{d})$, then use the same weight inequality.  This gives~\eqref{fs-eq:target-r1-raw-stmt}.  Lemma~\ref{lem:diagram-trichotomy-exhaustive}, using the relation classification of Lemma~\ref{fs-lem:target-relations}, shows that these cases exhaust all nonpartition admissible matrices.
\end{proof}

\begin{proof}[Proof of Lemma~\ref{fs-thm:target-high-moment}]
Expand
\[
 (X_\cL^{\wt}(\mathbf{d})-\mu_*)^{2h}
 =\sum_{s=0}^{2h}\binom{2h}{s}(-\mu_*)^{2h-s}
   X_\cL^{\wt}(\mathbf{d})^s.
\]
The partition diagrams, including the full-rank identity terms, combine exactly into Lemma~\ref{fs-lem:target-main}.  For every nonpartition term take absolute values and use Lemma~\ref{fs-lem:target-corrections}.  Its fiber power $\mu_*^{s-1}$ times the centering power $\mu_*^{2h-s}$ is exactly $\mu_*^{2h-1}$.  The $2h+1$ binomial terms cost only $e^{O(h)}$.
\end{proof}

\subsection{Endpoint degrees: Diophantine nonresonance and moments}\label{fs-sec:endpoint-moment}
For $\mathbf{x}\in\cL\cap S$, write
\[
 Y_\cL^{\wt}(\mathbf{x})
 =\sum_{\mathbf{d}\in\cL}\1_T(\mathbf{d})\1_S(\mathbf{x}-\mathbf{d})\wt(\mathbf{d}),
\]
and
\[
 \mathfrak M_S:=\E_\cL\sum_{\mathbf{x}\in\cL\cap S}
 |Y_\cL^{\wt}(\mathbf{x})-a_R^{\wt}(\mathbf{x})|^{2h}.
\]

\begin{lemma}[Growing marked endpoint moment]\label{fs-thm:endpoint-high-moment}
For an absolute $C>0$,
\begin{align}
 \mathfrak M_S\le{}&\vol(S)e^{Ch\log(2h)}\bar a_R\max(1,\bar a_R)^{h-1}\notag\\
 &+\vol(S)e^{Ch^2\log n}\max(1,\bar a_R)^{2h-1}
 \left(\frac{16}{27}+o(1)\right)^{n/2}\notag\\
 &+\vol(S)e^{Ch^2\log n}\max(1,\bar a_R)^{2h-1}
 \left[\left(\frac e4+o(1)\right)^{n/2}+2^{-n}\right].\label{fs-eq:endpoint-high-moment}
\end{align}
\end{lemma}
We prove this bound after Lemmas~\ref{fs-lem:endpoint-nonresonance}--%
\ref{fs-lem:endpoint-corrections}, which supply its partition and correction terms.

\begin{lemma}[Endpoint Diophantine nonresonance]\label{fs-lem:endpoint-nonresonance}
For a supported $q=1$ dependent target column
\[
 \mathbf{d}'=A\mathbf{x}+\sum_{j=1}^tB_j\mathbf{d}_j,
 \qquad A\in\Z,\quad B_j\in\{\pm1\},
\]
with all nonmarked coefficients in $\{0,\pm1\}$, support forces
\begin{equation}\label{fs-eq:endpoint-exact-relation}
 A=0,\qquad\sum_{j=1}^tB_j=1.
\end{equation}
Thus equality is the only one-parent relation; $t=0$ and $t=2$ are impossible; every nonpartition relation has $t\ge3$.
\end{lemma}
\begin{proof}
Part~\textup{(ii)} of
Lemma~\ref{lem:four-kernel-relation-certificate}, at $\alpha=\gamma$, gives
\begin{equation}\label{fs-eq:endpoint-approx-relation}
 \left|2A+\gamma^2\left(\sum_jB_j-1\right)\right|
 =O((t+1)w_n).
\end{equation}
The near-unit case of that lemma applies since $k=2h+1=O(\ell)$ and
forces~\eqref{fs-eq:endpoint-exact-relation}.  Its small-parent
classification also gives the stated alternatives for $t=0,1,2$.
\end{proof}

\begin{lemma}[Endpoint partition identification]\label{fs-lem:endpoint-partitions}
Fix the marked endpoint $\mathbf{x}$.  Among supported $q=1$ diagrams with nonmarked coefficients in $\{0,\pm1\}$, every dependent target column involving at most one free target pivot is equality.  Hence diagrams with no multi-target column are exactly ordinary set-partition diagrams, with coefficient one.
\end{lemma}
\begin{proof}
Lemma~\ref{fs-lem:endpoint-nonresonance} forces $A=0$ and the sum of nonzero nonmarked coefficients to equal one.  With at most one such coefficient, it must be a single $+1$.  Conversely equality is admissible.  The resulting one-nonzero-entry-per-column matrices are the partition matrices of Han~\cite[Proposition~2.4]{Han2024}, specialized to the present support.

More explicitly, order the occurrence slots $1,\ldots,s$.  Given a set
partition $\pi$, make the first slot of every block a pivot and let every
later slot in that block be the corresponding unit column.  Han's echelon
condition holds because the block pivot precedes every reuse; the marked-row
entry is zero in every target column, and $q=1$ gives Rogers coefficient
one.  Conversely, the first pivot occurrence of each distinct free target
defines a block and the preceding classification forces every later column
in that block to be its unit reuse.  These operations are inverse, including
the full-rank identity diagram, which corresponds to the all-singleton
partition.  A block $B$ therefore contributes exactly
\[
 \int_{C_{\mathbf x}}\omega(\mathbf d)^{|B|}\,d\mathbf d,
\]
the $|B|$-th cumulant of the Poisson integral used below.  Thus centering
removes precisely the singleton blocks and no diagram is lost or counted
twice.
\end{proof}

\begin{lemma}[Endpoint partition contribution]\label{fs-lem:endpoint-main}
The complete endpoint partition contribution satisfies
\begin{equation}\label{fs-eq:endpoint-main}
 \mathfrak M_S^{\rm part}
 \le\vol(S)\exp(C h\log(2h))\bar a_R\max(1,\bar a_R)^{h-1}.
\end{equation}
\end{lemma}
\begin{proof}
For fixed continuum $\mathbf{x}$, let $\Pi$ be a Poisson random measure on
\[
 C_{\mathbf{x}}:=\{\mathbf{d}\in T:\mathbf{x}-\mathbf{d}\in S\}
\]
with Lebesgue intensity and define
\[
 Z_{\mathbf{x}}:=\int_{C_{\mathbf{x}}}\wt(\mathbf{d})\,\Pi(d\mathbf{d}).
\]
Then
\[
 \E Z_{\mathbf{x}}=a_R^{\wt}(\mathbf{x}),
 \qquad
 \kappa_b(Z_{\mathbf{x}})=\int_{C_{\mathbf{x}}}\wt(\mathbf{d})^b\,d\mathbf{d}
 \le a_R^{\wt}(\mathbf{x})\quad(b\ge2).
\]
By Lemma~\ref{fs-lem:endpoint-partitions}, these are exactly the partition
diagrams in the raw moments.  Lemma~\ref{lem:poisson-cumulants} deletes
singleton blocks after centering.  For every $h\ge1$, the remaining
partitions have at most $h$ blocks and their number is at most $(2h)^{2h}$.
Writing $a=a_R^{\wt}(\mathbf x)$, each cumulant product is at most
$a\max(1,a)^{h-1}$.  By~\eqref{fs-eq:a-uniform},
\[
 \max(1,a)^{h-1}
 \le e^{O(h/n)}\max(1,\bar a_R)^{h-1},
 \qquad
 \int_S a_R^{\wt}(\mathbf x)\,d\mathbf x=\vol(S)\bar a_R.
\]
Integration and absorption of $(2h)^{2h}e^{O(h/n)}$ into
$e^{Ch\log(2h)}$ prove~\eqref{fs-eq:endpoint-main}, including $h=1$.
\end{proof}

\begin{lemma}[Endpoint correction diagrams]\label{fs-lem:endpoint-corrections}
For every raw $s$-occurrence endpoint moment, $1\le s\le2h$, the geometric
nonpartition contribution is at most
\begin{equation}\label{fs-eq:endpoint-geom-raw}
 \vol(S)e^{O(h^2\log n)}\max(1,\bar a_R)^{s-1}
 \left(\frac{16}{27}+o(1)\right)^{n/2},
\end{equation}
and the arithmetic/large-coefficient contribution is at most
\begin{equation}\label{fs-eq:endpoint-arith-raw}
 \vol(S)e^{O(h^2\log n)}\max(1,\bar a_R)^{s-1}
 \left[2^{-n}+\left(\frac e4+o(1)\right)^{n/2}\right].
\end{equation}
\end{lemma}
\begin{proof}
Fix $\mathbf{x}$.  Every valid target has transverse decomposition relative to $\mathbf{x}$ with a common transverse radius up to $1+O(w_n)$.  Lemma~\ref{fs-lem:endpoint-nonresonance} says the first genuine signed nonpartition relation contains $t\ge3$ free targets, and Corollary~\ref{fs-cor:angular-penalty} gives the $16/27$ factor.

The exact weighted fiber accounting is as follows.  If a free target pivot $\mathbf{d}$ is reused $m\ge1$ times in a raw moment, its weight is $\wt(\mathbf{d})^m\le\wt(\mathbf{d})$.  Hence its complete scalar/fiber integral is at most
\[
 \int_{C_{\mathbf{x}}}\wt(\mathbf{d})\,d\mathbf{d}=a_R^{\wt}(\mathbf{x}).
\]
By Lemma~\ref{fs-lem:missing-fiber}, a correction has at most $s-1$ such pivots.  This gives $a_R^{\wt}(\mathbf{x})^{s-1}$ exactly, rather than relying on an unweighted cap-volume approximation.  Use~\eqref{fs-eq:a-uniform} after integrating $\mathbf{x}$.

For the arithmetic family, apply Lemma~\ref{fs-lem:arithmetic-direct} with these weighted cap masses.  This yields~\eqref{fs-eq:endpoint-arith-raw}.  Exhaustiveness follows from Lemma~\ref{lem:diagram-trichotomy-exhaustive} together with Lemma~\ref{fs-lem:endpoint-nonresonance}.
\end{proof}

\begin{proof}[Proof of Lemma~\ref{fs-thm:endpoint-high-moment}]
Expand the centered $2h$-th power into raw moments.  Lemma~\ref{fs-lem:endpoint-main} combines all partition diagrams, including full rank, into the first term.  For every correction take absolute values and use Lemma~\ref{fs-lem:endpoint-corrections}.  A raw $s$-occurrence correction has power $\bar a_R^{s-1}$, and the remaining centering factor is $\bar a_R^{2h-s}$, leaving exactly $\bar a_R^{2h-1}$.  The binomial sum costs $e^{O(h)}$.
\end{proof}

\subsection{Fixed-ratio moment stability}\label{fs-sec:terminal}
The marked target proof is stable when the relative difference radius varies
in a compact nondegenerate interval.  The following version records one ratio
that will arise naturally in later additive applications.

\begin{lemma}[Fixed-ratio target moment]\label{fs-lem:fixed-ratio-moment}
There exists $\varepsilon_0>0$ such that, uniformly for $R=(\sqrt{4/3}+o(1))\Rstar$ and target micro-shells of relative half-width $w_n$ with
\[
 \left|\alpha-\frac{\sqrt3}{2}\right|\le\varepsilon_0,
\]
the target centered-moment bound of Lemma~\ref{fs-thm:target-high-moment} remains valid after replacing $T_R$ and $\mu_*(R)$ by the corresponding fixed-ratio target micro-shell and lens mean.  The assertion holds uniformly for every integer moment parameter $1\le h\le\ell$.
\end{lemma}
\begin{proof}
For height-one equal-radius parent constraints, the longitudinal relation
in Lemma~\ref{fs-lem:target-relations} is $2A+\sum B_j=1$, independently
of $\alpha$.  Hence the partition and geometric classes are unchanged.
On a compact neighborhood of $\sqrt3/2$, transverse fiber radii are
nondegenerate, so Lemma~\ref{fs-lem:arithmetic-direct} also applies uniformly.
The three correction bases persist, and the partition estimate in
Lemma~\ref{fs-lem:target-main} applies for every $h\ge1$ and positive
flattened mean.
\end{proof}

\begin{corollary}[Centered terminal fixed-ratio pointwise regularity]
\label{cor:centered-terminal-lens}
Let $R=(\sqrt{4/3}+o(1))\Rstar$, and let $A$ be a predetermined union of at
most $e^{o(n)}$ relative-width-$n^{-2}$ micro-shells whose radii are
$(1+o(1))\Rstar$, uniformly over the union.  With probability
$1-e^{-\Omega(n)}$ over $\cL$, simultaneously for every
$\mathbf d\in\cL\cap A$,
\begin{equation}\label{eq:centered-terminal-lens}
 r_{\cL,R}(\mathbf d)\ge\frac{\mu_R(\mathbf d)}{2},
 \qquad
 \mu_R(\mathbf d)\ge
 \left(\sqrt{\frac{13}{12}}+o(1)\right)^n.
\end{equation}
\end{corollary}
\begin{proof}
Keep the original, possibly overlapping, micro-shells $T_j$ and use
$m_{*,j}$ and
$\omega_j$ for their respective flattened means and weights.  Their
target/source ratios satisfy $\alpha_j=\sqrt3/2+o(1)$ uniformly.
Lemma~\ref{fs-lem:fixed-ratio-moment} with $h=1$ and
Lemma~\ref{fs-thm:target-high-moment}, summed over $j$, give
\[
 \E\sum_j\sum_{\mathbf d\in\cL\cap T_j}
 \bigl|\omega_j(\mathbf d)r_{\cL,R}(\mathbf d)-m_{*,j}\bigr|^2
 \le e^{o(n)}\vol(A)\mu_A,
\]
where
\[
 \mu_A=\left(\sqrt{\frac43}
        \sqrt{1-\frac3{16}}+o(1)\right)^n
 =\left(\sqrt{\frac{13}{12}}+o(1)\right)^n.
\]
If there are $M=e^{o(n)}$ constituent shells, then
$\sum_j\vol(T_j)\le M\vol(A)$, so the predetermined overlap
multiplicity is included in the displayed $e^{o(n)}$ factor.  Here
$\vol(A)=e^{o(n)}$, and the correction terms are smaller by one of the
three strict exponential bases in the preceding marked-moment bounds.  Since
on $T_j$
$\omega_j(\mathbf d)\mu_R(\mathbf d)=m_{*,j}$ exactly, any violation
of the first inequality in~\eqref{eq:centered-terminal-lens} contributes at
least $m_{*,j}^2/4$ to the displayed nonnegative sum.  Markov's
inequality therefore bounds the probability of any violation by
$e^{o(n)}/\mu_A=e^{-\Omega(n)}$.  The second inequality is the uniform
coarea estimate~\eqref{eq:lens-coarea} at
$R/\Rstar=\sqrt{4/3}+o(1)$ and $\alpha=\sqrt3/2+o(1)$.
\end{proof}

\subsection{Proof of the uniform centered moment bound}

\begin{lemma}[Uniform generic endpoint low-parent classification]
\label{lem:generic-endpoint-low-parent}
Fix $J\Subset(0,2)$.  There is a constant $C_{\rm nr}=C_{\rm nr}(J)>0$
such that the following holds for all sufficiently large $n$.  Let
$\alpha\in J$ satisfy
\begin{equation}\label{eq:generic-endpoint-gap}
 \Delta_{\rm lp}(\alpha)\ge C_{\rm nr}n^{-2}.
\end{equation}
If a supported $q=1$ height-one target column has the form
\[
 \mathbf d'=A\mathbf x+\sum_{j=1}^tB_j\mathbf d_j,
 \qquad B_j\in\{\pm1\},\qquad 0\le t\le2,
\]
then $t=1$, $A=0$, $B_1=1$, and $\mathbf d'=\mathbf d_1$.
Consequently every supported height-one nonpartition column combines at
least three distinct free target pivots.
\end{lemma}
\begin{proof}
Part~\textup{(ii)} of
Lemma~\ref{lem:four-kernel-relation-certificate} gives
\[
 |2A+\alpha^2z|\le C_J(t+1)n^{-2},
 \qquad z=\sum_jB_j-1,
\]
and shows that $|A|\le C_{\rm lp}$ and $|z|\le3$ when $t\le2$.
Choose $C_{\rm nr}>3C_J$.  The gap hypothesis forces $(A,z)=(0,0)$,
and the small-$t$ enumeration in the same lemma leaves only
$t=1$, $B_1=1$.
\end{proof}

\begin{proof}[Proof of Proposition~\ref{thm:intrinsic-marked-moments}]
The marked target occupies the first Rogers column.  For a height-one
column the equal-radius shell identities give $2A+\sum_jB_j=1$ after
rounding the $O(hn^{-2})$ support error, independently of $\alpha$;
see Lemma~\ref{lem:four-kernel-relation-certificate}.  Hence the
partition diagrams are precisely the equality/reflection blocks from
Lemma~\ref{fs-lem:target-affine-partitions}.  The calculation in
Lemma~\ref{fs-lem:target-main} bounds their centered contribution by
$e^{O(h\log(2h))}m_*\max(1,m_*)^{h-1}$ for every $h\ge1$.

The complete classification and Proposition~\ref{prop:uniform-diagram-budget} bound
all nonpartition matrices by the three terms of $\Xi_n$.  There are at
most $s-1$ free fibers in an $s$-occurrence raw correction.  If there are
$r-1\le s-1$ free target fibers and $\mu=\mu_R(\mathbf d)$, then
\[
 \omega(\mathbf d)^s\mu^{r-1}
 =m_*^{r-1}\omega(\mathbf d)^{s-r+1}
 \le \max(1,m_*)^{s-1}.
\]
For endpoint fibers, the product of at most $s-1$ weighted cap masses is
likewise at most $\max(1,a^*)^{s-1}$.  Multiplication by the centering factor,
using $m_*,a^*\le\max(1,m_*),\max(1,a^*)$, and summation over
$s\le2h$ gives \eqref{eq:intrinsic-target-moment}.

For endpoints, Lemma~\ref{lem:generic-endpoint-low-parent} is applied only to
columns containing at most two free target pivots.  It shows that equality is
the only supported one-parent relation and that two-parent relations are
unsupported.  A relation containing three or more pivots need not satisfy
$A=z=0$; it is already a geometric nonpartition relation and is controlled by
the $16/27$ Gram penalty.  This is precisely the exhaustive split of
Lemma~\ref{lem:diagram-trichotomy-exhaustive}.  The partition calculation
therefore applies with $a^*$ in place of $m_*$, and the unchanged geometric
and arithmetic estimates give \eqref{eq:intrinsic-endpoint-moment}.
Compactness of $J$ keeps all transverse radii away from zero and all coarea
Jacobians within fixed bounds.  Finally $2h+1=O(\log n)$ lies in the exact
Rogers range by Lemma~\ref{lem:growing-order-legal}, and all matrix-counting
and Gram-normalization losses are $e^{O(h^2\log n)}$.
\end{proof}

\begin{remark}[Why only a finite gap is required]
The support error in a height-one column with $t$ free parents is
$O_J((t+1)n^{-2})$.
Using an $O(h)$-height Diophantine gap to force $A=z=0$ for every column would
be unnecessarily strong and is generally false for three-parent resonances.
The proof excludes only $t\le2$ and sends every $t\ge3$ relation to the Gram
estimate.  Condition~\eqref{eq:generic-endpoint-gap} consequently has no
hidden dependence on the moment order and is valid uniformly down to $h=1$.
Its exceptional set is contained in $O(n^{-2})$-neighborhoods of the finite
zero set listed after~\eqref{eq:low-parent-gap}.  Hence every fixed ratio
outside that set satisfies the condition for all sufficiently large $n$.
In particular, $\alpha=1$ and $\alpha=2/\sqrt3$ are excluded by this generic
endpoint condition.  The later same-shell coverage and energy results use
the target estimate~\eqref{eq:intrinsic-target-moment}, which holds at both
ratios without this gap assumption.
\end{remark}

\subsection{Random affine lattices and unfolded marked moments}
\label{sec:affine-transfer}

The centered moment bound marks a point of $\cL$.  In an inhomogeneous problem the
distinguished point instead belongs to a coset $\boldsymbol\tau+\cL$.  For an
independent Haar shift this change does not require a separate Rogers formula:
the affine point can be unfolded to Lebesgue measure before the remaining
centered lattice variables are averaged.  General affine higher-moment formulas are available in
\cite[Theorem~2.12]{AlamGhoshHan2024}.  The unfolding identity below is
classical; it identifies the averaged variable and the removed Rogers
column.  The additional assertion in Proposition~\ref{thm:affine-marked-moments}
is uniform control of these marked shell fibers at growing moment order,
including the endpoint nonresonance condition.

\begin{lemma}[Affine fiber unfolding]\label{lem:affine-fiber-unfold}
Fix a unimodular lattice $\cL$.  Suppose that
$F:\R^n\times(\R^n)^k\to\R$ is Borel measurable in its first variable on
every lattice slice, and define
\[
 H(\mathbf z):=
 \sum_{\mathbf v_1,\ldots,\mathbf v_k\in\cL\setminus\{\mathbf0\}}
 F(\mathbf z,\mathbf v_1,\ldots,\mathbf v_k).
\]
If $F\ge0$, the following equality holds as an equality of extended
nonnegative integrals.  If $F$ is signed, assume instead the discrete-slice
absolute-integrability condition
\begin{equation}\label{eq:affine-slice-integrability}
 \int_{\R^n}
 \sum_{\mathbf v_1,\ldots,\mathbf v_k\in\cL\setminus\{\mathbf0\}}
 |F(\mathbf z,\mathbf v_1,\ldots,\mathbf v_k)|\,d\mathbf z<\infty.
\end{equation}
In the nonnegative case the identity is extended-valued; under
\eqref{eq:affine-slice-integrability} both sides are finite.  In either case,
\begin{align}
 &\E_{\boldsymbol\tau\mid\cL}
 \sum_{\mathbf z\in\boldsymbol\tau+\cL}
 \sum_{\mathbf v_1,\ldots,\mathbf v_k\in\cL\setminus\{\mathbf0\}}
 F(\mathbf z,\mathbf v_1,\ldots,\mathbf v_k)\notag\\
 &\qquad=
 \int_{\R^n}
 \sum_{\mathbf v_1,\ldots,\mathbf v_k\in\cL\setminus\{\mathbf0\}}
 F(\mathbf z,\mathbf v_1,\ldots,\mathbf v_k)\,d\mathbf z.
 \label{eq:affine-fiber-unfold}
\end{align}
For jointly Borel raw kernels satisfying the usual Siegel--Rogers
hypotheses, the remaining lattice variables may consequently be averaged
by those formulas.
\end{lemma}

\begin{proof}
Let $\mathcal F$ be a measurable fundamental domain of $\cL$, chosen with
volume one.  Tonelli's theorem applied to the countable orbit sum gives, in
the nonnegative case,
\[
 \int_{\mathcal F}\sum_{\mathbf u\in\cL}
 H(\boldsymbol\tau+\mathbf u)\,d\boldsymbol\tau
 =\sum_{\mathbf u\in\cL}\int_{\mathcal F+\mathbf u}
 H(\mathbf z)\,d\mathbf z
 =\int_{\R^n}H(\mathbf z)\,d\mathbf z.
\]
For signed $F$, condition~\eqref{eq:affine-slice-integrability} applies the
same nonnegative calculation to $|F|$.  Hence $H$ and its periodization are
absolutely integrable, and Fubini gives the signed identity.  Notice that
ordinary Lebesgue integrability of $F$ on the continuous product
$\R^n\times(\R^n)^k$ would not suffice: it does not control the values of
$F$ on the discrete lattice slices.  The raw-moment applications below use
the nonnegative branch.
\end{proof}

\begin{corollary}[Affine--centered Campbell identity]
\label{cor:affine-centered-campbell}
For every nonnegative Borel function, or absolutely integrable Borel
function, $G:\R^n\times\R^n\to\R$ that vanishes when its second coordinate
is zero,
\begin{equation}\label{eq:affine-centered-campbell}
 \E_{\cL,\boldsymbol\tau}
 \sum_{\mathbf z\in\boldsymbol\tau+\cL}
 \sum_{\mathbf v\in\cL\setminus\{\mathbf0\}}G(\mathbf z,\mathbf v)
 =\int_{\R^n}\int_{\R^n}G(\mathbf z,\mathbf v)
 \,d\mathbf v\,d\mathbf z.
\end{equation}
\end{corollary}

\begin{proof}
For $G\ge0$, apply the nonnegative branch of
Lemma~\ref{lem:affine-fiber-unfold} and then
Proposition~\ref{thm:siegel} to the remaining centered column.  In the signed
case, apply this argument first to $|G|$.  The right-hand side is finite by
absolute integrability, so Fubini implies
\eqref{eq:affine-slice-integrability} for almost every lattice; the signed
identity follows by subtraction of positive and negative parts.
\end{proof}

\begin{corollary}[Affine point-count moments]
\label{cor:affine-count-moments}
For every bounded Borel set $A\subset\R^n$, put
$X_A:=|(\boldsymbol\tau+\cL)\cap A|$.  Then
\begin{equation}\label{eq:affine-count-moments}
 \E X_A=\vol(A),\qquad
 \E[X_A(X_A-1)]=\vol(A)^2,\qquad
 \Var X_A=\vol(A).
\end{equation}
\end{corollary}

\begin{proof}
The mean is Lemma~\ref{lem:affine-fiber-unfold} with no centered column.  For
the factorial second moment, write the second affine point as
$\mathbf z+\mathbf v$, where $\mathbf v\in\cL\setminus\{\mathbf0\}$.
After unfolding $\mathbf z$, Proposition~\ref{thm:siegel} gives
\[
 \int_{\R^n}\int_{\R^n}
 \mathbf1_A(\mathbf z)\mathbf1_A(\mathbf z+\mathbf v)
 \,d\mathbf v\,d\mathbf z=\vol(A)^2.
\]
The variance identity follows from
$X_A^2=X_A(X_A-1)+X_A$.  These identities also appear in
Athreya~\cite[Lemmas~3--4]{Athreya2015}.
\end{proof}

\begin{lemma}[Affine partition diagrams after unfolding]
\label{lem:affine-partition-identification}
Fix $\alpha\in J\Subset(0,2)$ and the unfolded Euclidean mark, and take
$n$ sufficiently large.  For the affine target kernel, all supported
height-one diagrams without a multi-parent column are in bijection with set
partitions of the $s$ centered-source occurrences; a block of size $b$
contributes
\begin{equation}\label{eq:affine-target-block-cumulant}
 \omega(\mathbf d)^b K_R(\mathbf d).
\end{equation}
For the affine endpoint kernel assume in addition
$\inf_{\alpha\in J}|1-\alpha^2/2|>0$, with the lower bound on $n$ also
depending on this gap.  Then the analogous diagrams are in bijection with
set partitions of the $s$ centered-source occurrences, and a block of size
$b$ contributes
\begin{equation}\label{eq:affine-endpoint-block-cumulant}
 \int_{\{\mathbf c\in S_R:\,\mathbf x-\mathbf c\in T_{\alpha,R}\}}
       \omega(\mathbf x-\mathbf c)^b\,d\mathbf c.
\end{equation}
These are respectively the cumulants of
$\omega(\mathbf d)\operatorname{Pois}(K_R(\mathbf d))$ and of the
corresponding weighted Poisson integral.  In particular, the full-rank
identity is the all-singleton partition and centering deletes exactly the
singleton blocks.
\end{lemma}

\begin{proof}
After unfolding, the Euclidean mark is not a Rogers column.  By parts~(iii)
and~(iv) of Lemma~\ref{lem:four-kernel-relation-certificate}, under the
nonresonance hypothesis in the endpoint case, a supported one-parent column
is exactly a $+1$ unit reuse of an earlier centered pivot.  Thus a partition
$\pi$ produces a unique echelon matrix by taking the first occurrence in
each block as a pivot and every later occurrence as its unit column.
Conversely, the pivot and unit-reuse pattern of any such matrix recovers
$\pi$.  The two constructions are inverse and $q=1$ makes every coefficient
$c_D$ equal to one.

For a target block all repeated columns equal the same centered point
$\mathbf c$ in the lens
$\{\mathbf c\in S_R:\mathbf d+\mathbf c\in S_R\}$.  Integrating that pivot
gives $K_R(\mathbf d)$ and the repeated global weight is
$\omega(\mathbf d)^b$, proving
\eqref{eq:affine-target-block-cumulant}.  For an endpoint block the pivot is
integrated over the displayed cap and each reuse carries its identical
factor $\omega(\mathbf x-\mathbf c)$, proving
\eqref{eq:affine-endpoint-block-cumulant}.  The Poisson cumulant identity is
Lemma~\ref{lem:poisson-cumulants}.  Under the stated hypotheses the
height-one support relation permits only coefficient $+1$.  The endpoint gap is
essential here: at $\alpha=\sqrt2$ and $\|\mathbf x\|=\|\mathbf c\|=R$
with $\langle\mathbf x,\mathbf c\rangle=0$, both $\mathbf c$ and
$-\mathbf c$ satisfy the endpoint kernel.  Such a reflected one-parent
diagram is excluded precisely by the gap assumption.
\end{proof}

\begin{proof}[Proof of Proposition~\ref{thm:affine-marked-moments}]
Expand either centered power into raw moments and apply
Lemma~\ref{lem:affine-fiber-unfold} to the affine mark.  Conditional on the
resulting Euclidean mark, an $s$-occurrence raw moment contains only $s$
centered lattice columns and is exactly $F_s^{\rm at}$ or
$F_s^{\rm ae}$ from~\eqref{eq:raw-kernel-at}--\eqref{eq:raw-kernel-ae}.

We first classify height-one relations.  If $\mathbf d$ is the affine target
and $\mathbf c$ is a valid centered source point, the two equal-radius source
constraints give, uniformly for $\alpha\in J$,
\begin{equation}\label{eq:affine-target-longitudinal}
 \frac{2\langle\mathbf c,\mathbf d\rangle}{\|\mathbf d\|^2}
 =-1+O(w_n).
\end{equation}
Suppose that another supported centered column satisfies
$\mathbf c'=\sum_{j=1}^tB_j\mathbf c_j$ with
$B_j\in\{\pm1\}$.  Applying
\eqref{eq:affine-target-longitudinal} to $\mathbf c'$ and using linearity
shows that
$|\sum_jB_j-1|=O(hw_n)$.  The left side is an integer, so it vanishes for
all sufficiently large $n$.

For an affine endpoint $\mathbf x$, a valid centered source point instead
satisfies
\begin{equation}\label{eq:affine-endpoint-longitudinal}
 \frac{\langle\mathbf x,\mathbf c\rangle}{\|\mathbf x\|^2}
 =1-\frac{\alpha^2}{2}+O(w_n).
\end{equation}
The same substitution for $\mathbf c'$ gives
\[
 \left(1-\frac{\alpha^2}{2}\right)
 \left(\sum_jB_j-1\right)=O(hw_n).
\]
Under~\eqref{eq:affine-endpoint-gap}, this again forces
$\sum_jB_j=1$.  In either calculation equality is the only one-parent
relation, two parents are impossible, and every signed nonpartition column
combines at least three independent transverse pivots.

It follows from Lemma~\ref{lem:affine-partition-identification} that the
partition diagrams are ordinary equality partitions, including the
full-rank identity term.  In the endpoint case this invokes the gap in
\eqref{eq:affine-endpoint-gap}, which excludes the reuse
$\mathbf c\mapsto-\mathbf c$.  The block
cumulants~\eqref{eq:affine-target-block-cumulant}--%
\eqref{eq:affine-endpoint-block-cumulant}, together with
Lemma~\ref{lem:poisson-cumulants}, give precisely the first term in each
displayed bound, including the $\max(1,\cdot)$ envelope when the continuum
mean is below one.  Explicitly, for every $b\ge2$,
\[
 \omega(\mathbf d)^bK_R(\mathbf d)
 =\omega(\mathbf d)^{b-1}m_*\le m_*,
 \qquad
 \int\omega(\mathbf x-\mathbf c)^b d\mathbf c
 \le a_R^\omega(\mathbf x).
\]
Every centered $2h$-moment partition has at most $h$ nonsingleton blocks,
and there are at most $(2h)^{2h}$ such partitions.  This gives the stated
Poisson envelopes with $e^{O(h\log(2h))}$ uniformly for $h\ge1$.

At the first signed nonpartition column,
Corollary~\ref{fs-cor:angular-penalty} gives the $16/27$ factor.  For the
remaining arithmetic class, the proof of
Lemma~\ref{fs-lem:arithmetic-direct} applies with the affine marked column
deleted.  A nonidentity $s$-column diagram has rank at most $s-1$: at $q=1$
the only full-rank matrix is the identity partition term, while at $q\ge2$
the putative full-rank matrix $qI_s$ violates the primitivity condition.
Thus the missing-fiber argument leaves at most $s-1$ free fibers.  The same
coefficient count and Gram-determinant estimate yield
\[
 e^{O(h^2\log n)}
 \left[\left(\frac e4+o(1)\right)^{n/2}+2^{-n}\right]
\]
times the corresponding fiber envelope.  Multiplication by the remaining
centering factors leaves
$\max(1,m_*)^{2h-1}$ or $\max(1,a^*)^{2h-1}$, and the binomial sum costs
$e^{O(h)}$.  Integrating the unfolded mark proves both assertions.
\end{proof}

\begin{remark}[The affine endpoint resonance]
The target-multiplicity estimate has no analogue of the centered
low-parent Diophantine condition.  For endpoint degrees, however,
$\alpha=\sqrt2$ is a genuine orthogonality resonance in
\eqref{eq:affine-endpoint-longitudinal}; the proposition requires the
parameter range to stay a fixed distance from that value.  The near-unit ratio
$\alpha=1-1/\lceil\log n\rceil$ and the terminal ratio
$\alpha=\sqrt3/2+o(1)$ used below remain uniformly nonresonant.
\end{remark}

\section{From high moments to uniform shell-incidence regularity}\label{sec:growing-rogers}
This section turns the integrated bounds of Section
\ref{sec:shell-moments} into simultaneous estimates for every lattice point
in a fixed endpoint shell and every lattice difference in the prescribed
radial band.  We first give the near-critical specialization
$T_R=S_{(1-1/\ell)R}$, for which the endpoint degree requires the sharpest
bookkeeping.  The same-shell theorem $T_R=S_R$ is stated intrinsically in
Section~\ref{sec:additive-phases}.  Apart from the shell-count estimate, the
remaining steps are continuum exponent calculations, Markov's inequality,
and union bounds.  Recall
\[
 R_\ell=\sqrt{\frac43}\,e^{1/\ell}\Rstar.
\]

\subsection{Near-critical difference-band theorem and proof map}
For $\mathbf{d}\in T_R$ and $\mathbf{x}\in S_R$, define
\[
 r_\cL^{\wt}(\mathbf{d}):=\wt_R(\mathbf{d})
 \#\{\mathbf{x}\in\cL\cap S_R:\mathbf{x}-\mathbf{d}\in S_R\},
\]
\[
 D_\cL^{\wt}(\mathbf{x}):=\sum_{\mathbf{d}\in\cL\cap T_R}
 \1_{S_R}(\mathbf{x}-\mathbf{d})\wt_R(\mathbf{d}),
\]
and
\[
 a_R^{\wt}(\mathbf{x}):=\int_{T_R}\1_{S_R}(\mathbf{x}-\mathbf{d})\wt_R(\mathbf{d})\,d\mathbf{d}.
\]
The variables $r_\cL^\wt$ and $D_\cL^\wt$ are respectively
$X_\cL^\wt$ and $Y_\cL^\wt$ from Section~\ref{sec:shell-moments};
$D_\cL^\wt$ is also the specialization
$D^\omega_{\cL,R,\gamma}$ of the general weighted endpoint count.

\begin{theorem}[Near-critical fixed-shell regularity]\label{fs-thm:main-pointwise}
Fix a constant $K>0$ and a predetermined family of at most $\exp(o(n))$
fixed-shell difference queries $(S_R,T_R)$ satisfying
\[
 R_\ell\le R\le\exp(K\log^2n)\Rstar.
\]
With probability $1-e^{-\Omega(n)}$ over $\cL\sim X_n$, simultaneously for every query,
\begin{enumerate}[label=(\roman*),leftmargin=*]
\item every $\mathbf{d}\in\cL\cap T_R$ satisfies
\[
 \frac12\mu_*(R)\le r_\cL^{\wt}(\mathbf{d})\le\frac32\mu_*(R);
\]
\item every $\mathbf{x}\in\cL\cap S_R$ satisfies
\[
 \frac12a_R^{\wt}(\mathbf{x})\le D_\cL^{\wt}(\mathbf{x})\le\frac32a_R^{\wt}(\mathbf{x});
\]
\item $|\cL\cap S_R|=(1+o(1))\vol(S_R)$ and $|\cL\cap T_R|=(1+o(1))\vol(T_R)$.
\end{enumerate}
Uniformly over the family,
\begin{align*}
 \log\mu_*(R)
 &\ge\frac{4n}{3\ell}-O\!\left(\frac n{\ell^2}+\log n\right),\\
 \inf_{\mathbf{x}\in S_R}\log a_R^{\wt}(\mathbf{x})
 &\ge\frac n{3\ell}-O\!\left(\frac n{\ell^2}+\log n\right).
\end{align*}
\end{theorem}

S\"odergren studies short-vector lengths and angles
\cite{Sodergren2011,SodergrenAngles2011}; Kim treats observation sets of
volume proportional to $n$ \cite{Kim2016}, and Holm treats increasingly
many shortest vectors, with their number growing more slowly than $n^{1/6}$
\cite{Holm2022}.  The simultaneous object here is instead every marked
target and endpoint in one complete exponential shell.  Relative to
unmarked counting limits, the additional feature is the individual lens
fiber and the explicit tradeoff between moment order and distance from the
coverage edge in Corollary~\ref{cor:moment-buffer}.

The proof combines the continuum scales of
Lemma~\ref{fs-lem:continuum-scales}, the target and endpoint moments of
Lemmas~\ref{fs-thm:target-high-moment} and~\ref{fs-thm:endpoint-high-moment},
and the shell count in Lemma~\ref{lem:thin-shell-count}.
Here the choice $h=\ell$ compensates for the shrinking radius buffer.
For fixed $c>2/\sqrt3$, the proof of Theorem~\ref{thm:coverage}
instead uses a fixed $h>\log c/\log(c\sqrt3/2)$.

\subsection{Shell-wide Markov step}\label{fs-sec:pointwise}
\begin{proof}[Proof of Theorem~\ref{fs-thm:main-pointwise}]
Take $h=\ell=\lceil\log n\rceil$.  Fix one query and write $S=S_R$, $T=T_R$.  For endpoints, if any $\mathbf{x}\in\cL\cap S$ violates the stated factor-$3/2$ bound, then
\[
 \sum_{\mathbf{x}\in\cL\cap S}|Y_\cL^{\wt}(\mathbf{x})-a_R^{\wt}(\mathbf{x})|^{2h}
 \ge2^{-2h}\inf_{\mathbf{x}\in S}a_R^{\wt}(\mathbf{x})^{2h}.
\]
Using Lemma~\ref{fs-thm:endpoint-high-moment} and~\eqref{fs-eq:a-uniform}, Markov gives
\begin{align}
 \Prb[\text{some endpoint is bad}]
 \le{}&\vol(S)e^{Ch\log(2h)}\bar a_R^{-h}\notag\\
 &+\frac{\vol(S)}{\bar a_R}e^{Ch^2\log n}
 \left[\left(\frac{16}{27}+o(1)\right)^{n/2}\right.\notag\\
 &\hspace{36mm}\left. +\left(\frac e4+o(1)\right)^{n/2}+2^{-n}\right].\label{fs-eq:endpoint-failure}
\end{align}
The harmless factor $2^{2h}=e^{O(h)}$ is absorbed in the displayed overheads.

At the worst radius $R=R_\ell$,
\[
 \log\vol(S)=\frac12\log(4/3)n+\frac n\ell+O(\log n),
 \qquad
 h\log\bar a_R=\frac n3+O(n/\ell).
\]
Thus the first term has logarithm
\[
 \left(\frac12\log\frac43-\frac13\right)n+o(n)
 =-0.18949\ldots n+o(n).
\]
Furthermore
\[
 \frac{\vol(S)}{\bar a_R}
 =\mathcal N_\star
   \exp\left(\frac{2n}{3\ell}+O(n/\ell^2+\log n)\right).
\]
Therefore the three correction exponents are respectively
\[
 n\log\frac89+o(n)=-0.11778\ldots n+o(n),
\]
\[
 \frac n2\log\frac e3+o(n)=-0.049306\ldots n+o(n),
\]
and
\[
 -\frac n2\log3+o(n)=-0.549306\ldots n+o(n).
\]
The overhead $e^{O(h^2\log n)}=e^{O(\log^3n)}=e^{o(n)}$ cannot change
any sign.  Under $R\mapsto uR$ with $u\ge1$, both $\vol(S_R)$ and
$\bar a_R$ are multiplied by $u^n$, since the relative band and its
flattening weight are unchanged.  Thus the correction prefactor
$\vol(S_R)/\bar a_R$ is invariant, while the partition prefactor
$\vol(S_R)\bar a_R^{-h}$ is multiplied by $u^{n(1-h)}\le1$.
The endpoint failure bound is therefore uniform in $R$; the same scaling
applies to $\vol(T_R)$ and $\mu_*(R)$ below.

For targets, Lemma~\ref{fs-thm:target-high-moment} gives the same Markov
expression with $\vol(T),\mu_*$ in place of
$\vol(S),\bar a_R$.  At $R_\ell$,
\[
 \log\mu_*=\frac{4n}{3\ell}+O(n/\ell^2+\log n),
 \qquad
 \log\vol(T)=\frac12\log(4/3)n+O(n/\ell^2+\log n).
\]
Consequently its partition term has logarithm
\[
 \log\vol(T)-h\log\mu_*+O(h\log(2h))
 =\left(\frac12\log\frac43-\frac43\right)n+o(n),
\]
which is strictly smaller than the endpoint partition exponent.  Its
correction prefactor is
\[
 \frac{\vol(T)}{\mu_*}
 =\mathcal N_\star\exp\left(-\frac{4n}{3\ell}
                      +O(n/\ell^2+\log n)\right).
\]
After multiplication by the $16/27$, $e/4$, and $2^{-n}$ terms, the three
leading exponential constants are therefore the same strict negative
constants displayed for endpoints, with an additional favorable
$e^{-4n/(3\ell)+o(n/\ell)}$ factor.  Thus target failure is also
$e^{-\Omega(n)}$.

For shell populations, Lemma~\ref{lem:thin-shell-count} and Chebyshev with relative tolerance $n^{-1}$ give failure
\[
 O\!\left(\frac{n^2}{\vol(S)}\right)+O\!\left(\frac{n^2}{\vol(T)}\right)=e^{-\Omega(n)},
\]
because the smallest endpoint and difference bands have exponential volume.
The moment sums have already accounted for every target and endpoint in
each query, so no additional union over those lattice points is needed.
The family contains only $\exp(o(n))$ predetermined queries, and a union
over this family preserves $1-e^{-\Omega(n)}$ success probability.
\end{proof}

\begin{corollary}[The near-critical difference band reaches the critical radius]
\label{cor:critical-child}
For the smallest permitted endpoint radius
\[
 R_\ell=\sqrt{\frac43}\,e^{1/\ell}\Rstar,
\]
the target radius satisfies
\begin{align*}
 \gamma R_\ell
 &=\sqrt{\frac43}\,
   \exp\!\left(\frac1\ell+\log\left(1-\frac1\ell\right)\right)\Rstar\\
 &=\sqrt{\frac43}\,
   \exp\!\left(-\frac1{2\ell^2}+O(\ell^{-3})\right)\Rstar
  =\left(\sqrt{\frac43}+o(1)\right)\Rstar.
\end{align*}
Consequently Theorem~\ref{fs-thm:main-pointwise} certifies the complete
fixed-shell incidence pattern for differences down to the critical
$\sqrt{4/3}$ scale, up to a vanishing relative error.
\end{corollary}
\begin{proof}
Use
$\log(1-1/\ell)=-1/\ell-1/(2\ell^2)+O(\ell^{-3})$.
The regularity is a property of the autocorrelation of the fixed shell at
$R_\ell$; $T_R$ only selects the difference radii.  It need not be treated
as a new endpoint shell.
\end{proof}

\subsection{A quantitative moment--buffer tradeoff}

The choice $h=\ell$ and source buffer $e^{1/\ell}$ is convenient but not
canonical.  The following variant isolates exactly what the partition term
requires.

\begin{corollary}[Moment--buffer tradeoff]\label{cor:moment-buffer}
Fix constants $\kappa>0$ and $\beta$ satisfying
\begin{equation}\label{eq:moment-buffer-condition}
 \kappa\left(\beta-\frac23\right)
 >\log\sqrt{\frac43}.
\end{equation}
Put
\[
 h_\kappa:=\lceil\kappa\ell\rceil,
 \qquad
 R_{\beta}:=\sqrt{\frac43}\,e^{\beta/\ell}\Rstar.
\]
If the centered moments throughout the proof are taken at order
$2h_\kappa$, then conclusions \textup{(i)--(iii)} of
Theorem~\ref{fs-thm:main-pointwise} hold simultaneously for every query
in any predetermined $\exp(o(n))$ family of difference-band queries with
endpoint radii $R_{\beta}\le R\le\exp(K\log^2n)\Rstar$, for fixed $K>0$,
with probability $1-e^{-\Omega(n)}$.  Uniformly over the family, the
continuum lower bounds are
\begin{align*}
 \log\mu_*(R)
 &\ge\frac{(\beta+1/3)n}{\ell}
   -O\!\left(\frac n{\ell^2}+\log n\right),\\
 \inf_{\mathbf{x}\in S_R}\log a_R^{\wt}(\mathbf{x})
 &\ge\frac{(\beta-2/3)n}{\ell}
   -O\!\left(\frac n{\ell^2}+\log n\right).
\end{align*}
The largest Rogers tuple has order
$2h_\kappa+1=O(\log n)$, and the last target radius is still
\[
 \gamma R_{\beta}
 =\sqrt{\frac43}\,
   \exp\!\left(\frac{\beta-1}{\ell}
              -\frac1{2\ell^2}+O(\ell^{-3})\right)\Rstar
 =\left(\sqrt{\frac43}+o(1)\right)\Rstar.
\]
For $\kappa=1$, any
$\beta>2/3+\log\sqrt{4/3}=0.81050\ldots$ works; the main theorem uses
the round choice $\beta=1$.
\end{corollary}
\begin{proof}
At $R=R_{\beta}$, write $P_R=\vol(S_R)$ and $Q_R=\vol(T_R)$.
The continuum calculation gives
\begin{align*}
 \log P_R
 &=n\log\sqrt{\frac43}+\frac{\beta n}{\ell}+O(\log n),\\
 \log p_\gamma(R)
 &=-n\log\sqrt{\frac43}-\frac{2n}{3\ell}
   +O(n/\ell^2+\log n),\\
 \log \bar a_R
 &=\frac{(\beta-2/3)n}{\ell}
   +O(n/\ell^2+\log n),\\
 \log \mu_*(R)
 &=\frac{(\beta+1/3)n}{\ell}
   +O(n/\ell^2+\log n).
\end{align*}
Consequently the logarithm of the endpoint partition contribution after
the shell-wide Markov step is
\[
 \log\frac{P_R}{\bar a_R^{h_\kappa}}
 =\left[\log\sqrt{\frac43}
   -\kappa\left(\beta-\frac23\right)\right]n+o(n),
\]
which is strictly negative by~\eqref{eq:moment-buffer-condition}.  The target
partition exponent is
\[
 \log\frac{Q_R}{\mu_*(R)^{h_\kappa}}
 =\left[\log\sqrt{\frac43}
   -\kappa\left(\beta+\frac13\right)\right]n+o(n)
\]
and is smaller still.  The correction terms do not impose a new condition:
\[
 \frac{P_R}{\bar a_R}=p_\gamma(R)^{-1}(1+o(1))
 =\mathcal N_\star
   \exp\!\left(\frac{2n}{3\ell}+O(n/\ell^2+\log n)\right),
\]
so their strict $16/27$, $e/4$, and $2^{-n}$ margins are unchanged.

All uniformity checks survive the replacement $h\mapsto h_\kappa$.
Indeed $h_\kappa=O(\ell)$, so the Rogers validity conditions still hold,
the matrix-counting and Gram-density loss remains
$\exp(O(h_\kappa^2\log n))=e^{o(n)}$, and in the endpoint
Diophantine step $|z|=O(h_\kappa)<\ell^2$.  The displayed target-radius
formula follows by expanding $\log(1-1/\ell)$.
\end{proof}

\medskip
\begin{corollary}[Simultaneous nearby-radius queries]\label{fs-cor:slow-schedule}
For any predetermined sequence $R_{i+1}=\gamma R_i$ of $O(\log^3n)$
endpoint radii in the range of Theorem~\ref{fs-thm:main-pointwise}, all
pointwise representation, endpoint, and shell-count conclusions hold
simultaneously with probability $1-e^{-\Omega(n)}$.
\end{corollary}
\begin{proof}
There are $O(\log^3n)=\exp(o(n))$ fixed-shell queries.
\end{proof}

\begin{corollary}[Centered unweighted complete-shell interface]
\label{cor:centered-unweighted-interface}
For a fixed-shell difference query $(S_R,T_R)$ define
\[
 D_{\cL,R}(\mathbf x):=
 \sum_{\mathbf d\in\cL\cap T_R}\1_{S_R}(\mathbf x-\mathbf d),
 \qquad
 a_R(\mathbf x):=
 \int_{T_R}\1_{S_R}(\mathbf x-\mathbf d)\,d\mathbf d.
\]
Under the hypotheses of Theorem~\ref{fs-thm:main-pointwise}, with the same
probability and simultaneously over the entire predetermined family,
\begin{align}
 \frac12\mu_R(\mathbf d)
 &\le r_{\cL,R}(\mathbf d)
 \le\frac32\mu_R(\mathbf d)
 &&(\mathbf d\in\cL\cap T_R),\label{eq:unweighted-target-interface}\\
 \left(\frac12-o(1)\right)a_R(\mathbf x)
 &\le D_{\cL,R}(\mathbf x)
 \le\left(\frac32+o(1)\right)a_R(\mathbf x)
 &&(\mathbf x\in\cL\cap S_R).\label{eq:unweighted-endpoint-interface}
\end{align}
Both shell populations also have their continuum sizes up to a factor
$1+o(1)$.  In particular, for all sufficiently large $n$ the endpoint
constants in~\eqref{eq:unweighted-endpoint-interface} may be replaced by the
fully explicit constants $1/3$ and $2$.
\end{corollary}
\begin{proof}
The flattening identity is exact:
\[
 \omega(\mathbf d)\mu_R(\mathbf d)=\mu_*(R).
\]
Dividing the target inequalities of
Theorem~\ref{fs-thm:main-pointwise} by $\omega(\mathbf d)$ therefore proves
\eqref{eq:unweighted-target-interface} without loss.  By
\eqref{fs-eq:mu-uniform}, uniformly on $T_R$,
$\omega(\mathbf d)=1+O(1/n)$.  Positivity then gives
\[
 D_{\cL}^{\wt}(\mathbf x)
  =(1+O(1/n))D_{\cL,R}(\mathbf x),\qquad
 a_R^{\omega}(\mathbf x)=(1+O(1/n))a_R(\mathbf x),
\]
uniformly in $\mathbf x$.  Substitution into the endpoint inequalities of
Theorem~\ref{fs-thm:main-pointwise} proves
\eqref{eq:unweighted-endpoint-interface}; its shell-count assertion is
unchanged.
\end{proof}

\subsection{Critical-ball target bands}

\begin{proposition}[Critical-ball target regularity]\label{thm:critical-ball-lenses}
Put $R=\Rell$, $\Rcrit:=\gamma\Rell$, and let
\[
 \mathcal B_{\rm crit}
 :=\left\{\mathbf d:
 (1-1/\ell)\Rstar\le\|\mathbf d\|\le\Rcrit\right\}.
\]
With probability $1-e^{-\Omega(n)}$ over $\cL$, simultaneously for every
$\mathbf d\in\cL\cap\mathcal B_{\rm crit}$,
\begin{equation}
 r_{\cL,R}(\mathbf d)
 \ge\frac12\mu_R(\mathbf d),
 \qquad
 \log\mu_R(\mathbf d)
 \ge\frac{4n}{3\ell}
 -O\!\left(\frac n{\ell^2}+\log n\right).
 \label{eq:critical-ball-lens-lower}
\end{equation}
\end{proposition}

Zhang--Yang~\cite{ZhangYang2026} give a dimension-uniform tail bound for a
supremum of normalized ball counts over the radius.  The assertion here
controls the lens of every individual lattice target throughout a prescribed
critical band.  The deterministic cover below transfers the marked shell
estimates to that band without omitting boundary points.

We use the following deterministic micro-shell cover.
Let $[a,b]$ lie in a fixed compact subinterval of $(0,\infty)$; its endpoints
may depend on $n$.  For $R>0$ put
\[
 \vartheta_n:=\frac{1+w_n}{1-w_n},\qquad
 N:=\max\!\left\{1,\left\lceil
          \frac{\log(b/a)}{\log\vartheta_n}\right\rceil\right\},\qquad
 \rho_j:=\frac{aR}{1-w_n}\vartheta_n^j\quad(0\le j<N).
\]
The standard closed shell $S_{\rho_j}$ has radial interval
$[aR\vartheta_n^j,aR\vartheta_n^{j+1}]$.  Consequently $N=O(n^2)$ and
\begin{equation}\label{eq:standard-shell-cover}
 \{aR\le\|\mathbf d\|\le bR\}
 \subseteq\bigcup_{j<N}S_{\rho_j}
 \subseteq\{aR\le\|\mathbf d\|\le bR\vartheta_n\},
 \qquad \sum_{j<N}\1_{S_{\rho_j}}\le2.
\end{equation}
Indeed adjacent shells meet only on their common boundary sphere.  All
shells have the full prescribed relative half-width $w_n$, and the cover
contains every boundary point.  Their radii stay within $O(w_nR)$ of the
original band, while bounded overlap permits summing nonnegative integral
estimates.  We use this deterministic cover in both the centered and affine
applications below.

\begin{proof}[Proof of Proposition~\ref{thm:critical-ball-lenses}]
Cover $\mathcal B_{\rm crit}$ by the $O(n^2)$ standard closed shells of
\eqref{eq:standard-shell-cover}, taking
$a=(1-1/\ell)\Rstar/R$ and $b=\gamma$.  Their radii divided by $R$, as
well as the pointwise ratios $\alpha=\|\mathbf d\|/R$ throughout the
cover, lie in a fixed compact subinterval of $(0,2)$ and are at most
$\gamma+O(w_n)$.  The target proof in
Lemma~\ref{fs-thm:target-high-moment} and its supporting lemmas
is uniform over this larger family.  Indeed the height-one longitudinal identity
$2A+\sum B_j=1$ is independent of $\alpha$; the transverse radii remain
nondegenerate; and the $16/27$, $e/4$, and $2^{-n}$ estimates are uniform
on compact ratio intervals.  Flatten the lens mean separately in each
micro-shell with $\wt=\mu_*/\mu$.

No exact monotonicity assertion for intersections of two annuli is needed.
The uniform coarea estimate gives, with
$\alpha=\|\mathbf d\|/R$,
\[
 \log\mu_R(\mathbf d)
 =n\log(R/\Rstar)+\frac n2\log(1-\alpha^2/4)+O(\log n).
\]
The displayed main term decreases in $\alpha\in(0,2)$ and its error is
uniform.  Since $\alpha\le\gamma+O(w_n)$,
\eqref{fs-eq:mu-min} gives the second assertion in
\eqref{eq:critical-ball-lens-lower}.  For every target
micro-shell, its volume is at most
\[
 \exp\!\left(\tfrac12\log(4/3)n+o(n)\right)
 =\mathcal N_\star2^{o(n)}.
\]
After the shell-wide Markov step, the partition logarithm is therefore at
most
\[
 \frac12\log(4/3)n-\frac43n+o(n)<-\Omega(n).
\]
Every correction term has the common worst-case prefactor
$\mathcal N_\star\exp(O(n/\ell))$ and hence retains exactly the three strict margins
computed in the proof of Theorem~\ref{fs-thm:main-pointwise}.  Union-bounding
over $O(n^2)$ micro-shells proves
$\wt(\mathbf d)r_{\cL,R}(\mathbf d)\ge\mu_*/2$.  Since
$\wt(\mathbf d)\mu_R(\mathbf d)=\mu_*$ exactly, division by the positive
weight gives the displayed unweighted factor $1/2$ with no additional
slack.
\end{proof}

\begin{corollary}[Intrinsic critical-ball difference coverage]
\label{cor:critical-ball-coverage}
With $R=R_\ell$ and $R_{\mathrm{crit}}=\gamma R_\ell$, with probability
$1-e^{-\Omega(n/\log n)}$,
\begin{equation}\label{eq:critical-ball-inclusion}
 (\cL\cap B_{R_{\mathrm{crit}}})\setminus\{0\}
 \subseteq(\cL\cap S_R)-(\cL\cap S_R).
\end{equation}
The failure rate includes the event of a nonzero vector below
$(1-1/\ell)\Rstar$; the bandwise incidence estimate alone has exponentially
small failure in $n$.
Every target on the left lying in the band of
Proposition~\ref{thm:critical-ball-lenses} in fact has
$\exp(\Omega(n/\log n))$ ordered representations.
\end{corollary}

\begin{proof}
Proposition~\ref{thm:critical-ball-lenses} treats all lattice targets of length at
least $(1-1/\ell)\Rstar$ and at most $R_{\mathrm{crit}}$.  By Siegel and
Markov,
\[
 \Prb\bigl[(\cL\setminus\{0\})\cap
 B_{(1-1/\ell)\Rstar}\ne\varnothing\bigr]
 \le(1-1/\ell)^n
 =\exp\!\left(-\frac n\ell+O\!\left(\frac n{\ell^2}\right)\right).
\]
Outside this event there are no omitted nonzero targets.  The representation
lower bound follows from \eqref{eq:critical-ball-lens-lower}.
\end{proof}
\subsection{Random-coset full-shell regularity and coverage}

For the near-unit affine-difference band $T_R=S_{\gamma R}$, define the unweighted affine
endpoint degree
\begin{equation}\label{eq:affine-unweighted-degree}
 D^{\rm aff}_{\cL,\boldsymbol\tau,R}(\mathbf x)
 :=\sum_{\mathbf c\in\cL}
   \mathbf1_{S_R}(\mathbf c)\mathbf1_{T_R}(\mathbf x-\mathbf c),
 \qquad \mathbf x\in(\boldsymbol\tau+\cL)\cap S_R,
\end{equation}
and its Euclidean mean
\begin{equation}\label{eq:affine-unweighted-mean}
 a_R(\mathbf x):=\int_{S_R}\mathbf1_{T_R}(\mathbf x-\mathbf c)
 \,d\mathbf c.
\end{equation}

\begin{proposition}[Affine fixed-shell pointwise regularity]
\label{thm:affine-pointwise}
Fix $K>0$ and a predetermined family of at most $\exp(o(n))$ fixed-shell
difference queries $(S_R,T_R)$ satisfying
\[
 R_\ell\le R\le\exp(K\log^2n)\Rstar.
\]
With probability $1-e^{-\Omega(n)}$ over
$(\cL,\boldsymbol\tau)\sim\widetilde X_n$, simultaneously for every
query,
\begin{align}
 \frac12\mu_R(\mathbf d)
 &\le r^{\rm aff}_{\cL,\boldsymbol\tau,R}(\mathbf d)
 \le\frac32\mu_R(\mathbf d)
 &&\forall\mathbf d\in(\boldsymbol\tau+\cL)\cap T_R,
 \label{eq:affine-pointwise-target}\\
 \frac12a_R(\mathbf x)
 &\le D^{\rm aff}_{\cL,\boldsymbol\tau,R}(\mathbf x)
 \le\frac32a_R(\mathbf x)
 &&\forall\mathbf x\in(\boldsymbol\tau+\cL)\cap S_R.
 \label{eq:affine-pointwise-endpoint}
\end{align}
Moreover, every affine endpoint and difference shell in the family has population
$1+o(1)$ times its Euclidean volume.  Uniformly over the family,
\begin{align}
 \inf_{\mathbf d\in T_R}\log\mu_R(\mathbf d)
 &\ge\frac{4n}{3\ell}
   -O\!\left(\frac n{\ell^2}+\log n\right),
 \label{eq:affine-mu-lower}\\
 \inf_{\mathbf x\in S_R}\log a_R(\mathbf x)
 &\ge\frac n{3\ell}
   -O\!\left(\frac n{\ell^2}+\log n\right).
 \label{eq:affine-a-lower}
\end{align}
The largest tuple presented to the centered Rogers formula has order $2h$.
\end{proposition}

\begin{proof}
Choose a fixed compact interval $J\Subset(0,\sqrt2)$ containing
$\gamma=1-1/\ell$ for all sufficiently large $n$, and take $h=\ell$ in
Proposition~\ref{thm:affine-marked-moments}.  Its target and endpoint bounds both
apply.  Repeating the shell-wide Markov calculation in the proof of
Theorem~\ref{fs-thm:main-pointwise} gives the same partition exponents.  The
common correction contribution is at most
\[
 \mathcal N_\star
 \exp\!\left(O(n/\ell)+O(\ell^2\log n)\right)\Xi_n
 =e^{-\Omega(n)}.
\]
The subexponential overhead does not consume the smallest correction margin
$\tfrac12\log(3/e)$.  The radial variation estimates
\eqref{fs-eq:mu-uniform}--\eqref{fs-eq:a-uniform} imply
$\omega=1+O(1/n)$ and identify the unweighted means.  As in the centered
proof, use a fixed deviation threshold slightly below $1/2$ before removing
the weight; this yields
\eqref{eq:affine-pointwise-target}--\eqref{eq:affine-pointwise-endpoint}.

For an endpoint or difference shell $A$, Corollary~\ref{cor:affine-count-moments}
and Chebyshev's inequality give, at relative tolerance $n^{-1}$,
\[
 \Prb\bigl[\bigl|| (\boldsymbol\tau+\cL)\cap A|-\vol(A)\bigr|
              >n^{-1}\vol(A)\bigr]
 \le\frac{n^2}{\vol(A)}=e^{-\Omega(n)}.
\]
The smallest queried shell already has exponential volume.  A union bound
over the predetermined family proves all simultaneous claims.  Finally,
\eqref{eq:affine-mu-lower}--\eqref{eq:affine-a-lower} are the continuum
bounds of Lemma~\ref{fs-lem:continuum-scales}.
\end{proof}

The next elementary statement locates the inhomogeneous analogue of the
shortest-vector edge.  Its first- and second-moment proof belongs to the
classical affine void-probability framework; compare
Athreya~\cite[Theorem~1]{Athreya2015}.  The shell regularity above adds
simultaneous control of every affine target and endpoint for one
independent random coset, in the stated nonresonant range.

\begin{proposition}[Nearest-point tails in a random affine lattice]
\label{prop:affine-minimum}
Set
\[
 \rho_{\rm aff}(\cL,\boldsymbol\tau)
 :=\min_{\mathbf z\in\boldsymbol\tau+\cL}\|\mathbf z\|,
 \qquad
 U_{\rm aff}:=\kappa_n\rho_{\rm aff}(\cL,\boldsymbol\tau)^n.
\]
For every $u,t>0$,
\begin{equation}\label{eq:affine-minimum-tails}
 \Prb[U_{\rm aff}\le u]\le u,
 \qquad
 \Prb[U_{\rm aff}>t]\le\frac1t.
\end{equation}
In particular,
\begin{equation}\label{eq:affine-minimum-localization}
 \Prb\!\left[
 \left(1-\frac1\ell\right)\Rstar
 \le\rho_{\rm aff}(\cL,\boldsymbol\tau)
 \le\left(1+\frac1\ell\right)\Rstar\right]
 \ge1-e^{-\Omega(n/\log n)}.
\end{equation}
\end{proposition}

\begin{proof}
Let $B$ be the ball of volume $v$ and put
$X=|(\boldsymbol\tau+\cL)\cap B|$.  The first bound follows from
$\Prb[X\ge1]\le\E X=v$.  By
Corollary~\ref{cor:affine-count-moments}, $\E X=\Var X=v$; hence
\[
 \Prb[X=0]\le
 \Prb[|X-v|\ge v]\le\frac{\Var X}{v^2}=\frac1v.
\]
Take $v=u$ and $v=t$, respectively.  For the final assertion use
$u=(1-1/\ell)^n$ and $t=(1+1/\ell)^n$ and expand their logarithms.
\end{proof}

\begin{proposition}[Affine critical-band regularity]
\label{thm:affine-critical-ball-lenses}
Put $R=\Rell$, $\Rcrit=\gamma\Rell$, and
\[
 \mathcal B_{\rm crit}:=
 \left\{\mathbf d:
 (1-1/\ell)\Rstar\le\|\mathbf d\|\le\Rcrit\right\}.
\]
With probability $1-e^{-\Omega(n)}$ over
$(\cL,\boldsymbol\tau)\sim\widetilde X_n$, simultaneously for every
$\mathbf d\in(\boldsymbol\tau+\cL)\cap\mathcal B_{\rm crit}$,
\begin{equation}\label{eq:affine-critical-band-lower}
 r^{\rm aff}_{\cL,\boldsymbol\tau,R}(\mathbf d)
 \ge\frac12\mu_R(\mathbf d),
 \qquad
 \log\mu_R(\mathbf d)
 \ge\frac{4n}{3\ell}
 -O\!\left(\frac n{\ell^2}+\log n\right).
\end{equation}
\end{proposition}

\begin{proof}
Use the same $O(n^2)$ standard closed shells from
\eqref{eq:standard-shell-cover} as in the centered proof.  They cover
$\mathcal B_{\rm crit}$ pointwise, and all their ratios to $R$ lie in a
fixed compact subset of $(0,2)$ with upper endpoint $\gamma+O(w_n)$.
Apply the target half of
Proposition~\ref{thm:affine-marked-moments} with $h=\ell$ to each shell; no
endpoint gap is needed.  The longitudinal identity
\eqref{eq:affine-target-longitudinal} is uniform in the ratio, while all
transverse radii remain nondegenerate.  The same shell-wide Markov estimate
as in Proposition~\ref{thm:critical-ball-lenses} gives failure
$e^{-\Omega(n)}$, including all three terms of $\Xi_n$.

The uniform coarea identity
\[
 \log\mu_R(\mathbf d)
 =n\log(R/\Rstar)+\frac n2\log(1-\alpha^2/4)+O(\log n)
\]
has a decreasing main term in $\alpha=\|\mathbf d\|/R$.  Its uniform lower
bound for $\alpha\le\gamma+O(w_n)$ is the logarithmic estimate in
\eqref{eq:affine-critical-band-lower}; no exact monotonicity of annular lens
volumes is asserted.  Finally
$\omega(\mathbf d)\mu_R(\mathbf d)=m_*$, so division by the positive
flattening weight gives the unweighted factor $1/2$ exactly.
\end{proof}

\begin{corollary}[Random-coset critical-ball difference coverage]
\label{cor:affine-critical-ball-coverage}
With $R=\Rell$ and $R_{\rm crit}=\gamma\Rell$, with probability
$1-e^{-\Omega(n/\log n)}$ over $\widetilde X_n$,
\begin{equation}\label{eq:affine-critical-ball-inclusion}
 (\boldsymbol\tau+\cL)\cap B_{R_{\rm crit}}
 \subseteq
 \bigl((\boldsymbol\tau+\cL)\cap S_R\bigr)
       -(\cL\cap S_R).
\end{equation}
The failure rate includes the event of a coset point below
$(1-1/\ell)\Rstar$; the bandwise incidence estimate alone has exponentially
small failure in $n$.
Every coset point in the critical band has
$\exp(\Omega(n/\log n))$ ordered affine--centered representations.
\end{corollary}

\begin{proof}
By Proposition~\ref{prop:affine-minimum}, except with probability
$e^{-\Omega(n/\log n)}$ the random coset contains no point below
$(1-1/\ell)\Rstar$.  Proposition~\ref{thm:affine-critical-ball-lenses} treats all
remaining points up to $R_{\rm crit}$.  For every representation counted in
\eqref{eq:affine-critical-band-lower},
$\mathbf d=(\mathbf d+\mathbf c)-\mathbf c$, with
$\mathbf d+\mathbf c\in(\boldsymbol\tau+\cL)\cap S_R$ and
$\mathbf c\in\cL\cap S_R$, proving the inclusion.  The exponential
multiplicity follows from the same lower bound.
\end{proof}

The following fixed-ratio form is useful when an inhomogeneous nearest point
must be exposed as a difference at the terminal $\sqrt{4/3}$ shell.

\begin{lemma}[Affine terminal fixed-shift regularity]
\label{lem:affine-terminal-lens}
Let $R=(\sqrt{4/3}+o(1))\Rstar$, and let $A$ be a predetermined union of at
most $e^{o(n)}$ relative-width-$n^{-2}$ micro-shells, all at radii
$(1+o(1))\Rstar$, uniformly over the union.  With probability $1-e^{-\Omega(n)}$ over
$\widetilde X_n$, simultaneously for every
$\mathbf s\in(\boldsymbol\tau+\cL)\cap A$,
\begin{equation}\label{eq:affine-terminal-lens}
 \#\{\mathbf b\in\cL\cap S_R:\mathbf s+\mathbf b\in S_R\}
 \ge\frac12\left(\sqrt{\frac{13}{12}}+o(1)\right)^n.
\end{equation}
Only affine unfolding and Rogers order at most two are used.
\end{lemma}

\begin{proof}
The empty union is immediate.  Otherwise write $A=\bigcup_{j=1}^M T_j$
with $M=e^{o(n)}$, keeping the original, possibly overlapping micro-shells.
Define the two envelopes
\[
 \mu_{\min}:=\inf_{\mathbf s\in A}\mu_R(\mathbf s),\qquad
 \mu_{\max}:=\sup_{\mathbf s\in A}\mu_R(\mathbf s).
\]
The target/source ratios are $\sqrt3/2+o(1)$ uniformly, so
Lemma~\ref{lem:thin-shell-coarea} gives
\[
 \log\mu_{\min}=an+o(n),\qquad
 \log\mu_{\max}=an+o(n),\qquad
 a:=\frac12\log\frac{13}{12}>0.
\]
Also $\vol(A)=e^{o(n)}$.  For each $T_j$, put
$m_{*,j}:=\inf_{\mathbf s\in T_j}\mu_R(\mathbf s)$ and
$\omega_j(\mathbf s):=m_{*,j}/\mu_R(\mathbf s)$.
Apply~\eqref{eq:affine-target-moment} with $h=1$ on each micro-shell.
Since $\mu_{\min}\le m_{*,j}\le\mu_{\max}$ and $\mu_{\min}>1$ for
large $n$, the nonnegative sum
\[
 Z:=\sum_j\sum_{\mathbf s\in(\boldsymbol\tau+\cL)\cap T_j}
 \left|\omega_j(\mathbf s)
 r^{\rm aff}_{\cL,\boldsymbol\tau,R}(\mathbf s)-m_{*,j}\right|^2
\]
satisfies
\[
 \E Z\le e^{o(n)}\mu_{\max}\sum_j\vol(T_j)
 \le e^{o(n)}\vol(A)\mu_{\max}.
\]
Here the second inequality uses $\sum_j\vol(T_j)\le M\vol(A)$,
so all overlap is absorbed into the subexponential factor.
If any target has
$r^{\rm aff}_{\cL,\boldsymbol\tau,R}(\mathbf s)<\mu_R(\mathbf s)/2$,
then the exact identity
$\omega_j(\mathbf s)\mu_R(\mathbf s)=m_{*,j}$ on a containing shell
forces $Z\ge m_{*,j}^2/4\ge\mu_{\min}^2/4$.  Thus
\[
 \Prb[\text{some affine target is bad}]
 \le\frac{4\E Z}{\mu_{\min}^2}
 \le e^{o(n)}\vol(A)\frac{\mu_{\max}}{\mu_{\min}^2}
 =e^{-an+o(n)}.
\]
On the complementary event every count is at least $\mu_{\min}/2$,
which proves~\eqref{eq:affine-terminal-lens}.  Only the common exponential
rate of the two envelopes is used; no relative comparison between distinct
micro-shells is required.
\end{proof}

\section{The sharp same-shell coverage transition}
\label{sec:additive-phases}

We now apply the marked theory to intrinsic additive statistics of one random
lattice shell.  Fix $c>1$, set $R=c\Rstar$, and abbreviate
\[
 A_{\cL}(c):=\cL\cap S_R,\qquad N:=|A_{\cL}(c)|,
 \qquad P:=\vol(S_R)=c^{n+o(n)},
\]
where all $o(n)$ terms may depend on a fixed $c$.  For logarithms of
cardinalities and additive energies we use $\log0=-\infty$; identities
involving division by $N$ are restricted to $N>0$.  By
Lemma~\ref{lem:thin-shell-count}, this event has probability
$1-e^{-\Omega_c(n)}$.  Recall the continuum lens
kernel
\[
 K_R(\mathbf d):=\vol(S_R\cap(S_R+\mathbf d)).
\]
Define the represented fraction
\[
 \Theta_{\cL}(c):=
 \frac{|\{\mathbf d\in A_{\cL}(c):r_{\cL,R}(\mathbf d)>0\}|}
 {|A_{\cL}(c)|}\quad\text{when }|A_{\cL}(c)|>0,
 \qquad \Theta_{\cL}(c):=0\quad\text{otherwise}.
\]

\begin{theorem}[Sharp same-shell autocorrelation transition]
\label{thm:coverage}
For every fixed $c>1$ with $c\ne2/\sqrt3$, with probability
$1-e^{-\Omega_c(n)}$,
\[
 \Theta_{\cL}(c)\le e^{-\Omega_c(n)}\quad(1<c<2/\sqrt3),
 \qquad
 \Theta_{\cL}(c)=1\quad(c>2/\sqrt3).
\]
Above the threshold one has, simultaneously for every
$\mathbf d\in A_{\cL}(c)$, the pointwise lens comparison
\begin{equation}\label{eq:same-shell-lens-comparison}
 \left(1-\frac1{\log n}\right)K_R(\mathbf d)
 \le r_{\cL,R}(\mathbf d)
 \le\left(1+\frac1{\log n}\right)K_R(\mathbf d),
\end{equation}
and hence the exact inclusion
\begin{equation}\label{eq:exact-difference-inclusion}
 A_{\cL}(c)\subseteq A_{\cL}(c)-A_{\cL}(c).
\end{equation}
Moreover, for every $\mathbf x\in A_{\cL}(c)$,
\begin{equation}\label{eq:same-shell-endpoint-identity}
 \#\{\mathbf d\in A_{\cL}(c):
          \mathbf x-\mathbf d\in A_{\cL}(c)\}
 =r_{\cL,R}(\mathbf x),
\end{equation}
so the same comparison controls every same-shell endpoint degree.
\end{theorem}

The quantifier is over all target differences in one sampled shell.  In
comparison with the short-vector laws discussed in
Section~\ref{sec:growing-rogers}, the conclusion gives relative lens
multiplicities for the complete exponential population and a sharp
same-shell support threshold.  Its proof uses the target half of
Proposition~\ref{thm:intrinsic-marked-moments}; no generic endpoint gap is
needed at the same-shell ratio.

We first collect the continuum estimates, also used for the transitions in
Sections~\ref{sec:energy} and~\ref{sec:convolution}, and then prove the
coverage statement from the represented-target first moment and the marked
high moments.

\subsection{Continuum lens asymptotics}

The following lemma collects the Laplace calculations behind the three
thresholds.

\begin{lemma}[Lens, angular, and energy scales]\label{lem:additive-continuum}
Uniformly when $\|\mathbf d\|=\alpha R(1+O(n^{-2}))$ and $\alpha$ ranges in a
compact subset of $(0,2)$,
\begin{equation}\label{eq:additive-lens}
 K_R(\mathbf d)=
 \exp\!\left(n\log c+\frac n2\log\!\left(1-\frac{\alpha^2}{4}\right)
                 +O(\log n)\right).
\end{equation}
Moreover,
\begin{equation}\label{eq:kernel-mass}
 \int_{\R^n}K_R(\mathbf d)\,d\mathbf d=P^2.
\end{equation}
If $\mathbf X,\mathbf Y$ are independent uniform continuum points in $S_R$,
then $\|\mathbf X-\mathbf Y\|/R$ lies outside any fixed neighborhood of
$\sqrt2$ with probability $e^{-\Omega(n)}$.  Finally,
\begin{align}
 I_2(R)&:=\int_{\R^n}K_R(\mathbf d)^2d\mathbf d
 =\exp\!\left(n\log\frac{4c^3}{3\sqrt3}+O(\log n)\right),
 \label{eq:I2}\\
 I_4(R)&:=\int_{\R^n}K_R(\mathbf d)^4d\mathbf d
 =\exp\!\left(n\log\frac{32c^5}{25\sqrt5}+O(\log n)\right).
 \label{eq:I4}
\end{align}
For $p\in\{2,4\}$ put $\alpha_p:=2/\sqrt{p+1}$ and
$I_p(R):=\int_{\R^n}K_R(\mathbf d)^p\,d\mathbf d$.
If $J\Subset(0,2)$ is a fixed interval with $\alpha_p$ in its interior,
then for some $\delta_{p,J}>0$ and all sufficiently large $n$,
\begin{equation}\label{eq:lens-laplace-tail}
 \int_{\{\|\mathbf d\|/R\notin J\}}K_R(\mathbf d)^p\,d\mathbf d
 \le e^{-\delta_{p,J}n}I_p(R).
\end{equation}
\end{lemma}

\begin{proof}
Equation~\eqref{eq:additive-lens} is the coarea estimate
\eqref{eq:lens-coarea} with $R/\Rstar=c$.  Fubini gives
\[
 \int K_R(\mathbf d)d\mathbf d
 =\int_{S_R}\int_{\R^n}\mathbf1_{S_R}(\mathbf x-\mathbf d)
 \,d\mathbf d\,d\mathbf x=P^2.
\]
For two independent radial directions, their inner product has density
proportional to $(1-u^2)^{(n-3)/2}$; it is exponentially concentrated near
$u=0$, which is equivalent to relative difference length $\sqrt2$.

We next give a global bound, including separations near zero and the
support boundary where the compact coarea asymptotic is not uniform.
Put $b=R(1+w_n)$, $r=\|\mathbf d\|$, and let $\kappa_m$ be the unit-ball
volume in $\R^m$.  The intersection $B_b\cap(B_b+\mathbf d)$ has axial
length at most $2b$.  At axial coordinate $t$, the squared transverse radius
is at most $b^2-\max\{t^2,(t-r)^2\}\le b^2-r^2/4$.  Since $S_R\subset B_b$,
\begin{equation}\label{eq:global-ball-lens-bound}
 K_R(\mathbf d)
 \le\vol(B_b\cap(B_b+\mathbf d))
 \le2b\kappa_{n-1}(b^2-r^2/4)_+^{(n-1)/2}.
\end{equation}
For $p=2,4$, polar coordinates with $u=r/b$ therefore give
\begin{align}
 I_p(R)&\le
 2^p n\kappa_n\kappa_{n-1}^p b^{(p+1)n}
 \int_0^2\left[u(1-u^2/4)^{p/2}\right]^{n-1}\,du.
 \label{eq:global-laplace-upper}
\end{align}
Here $\kappa_n(\Rstar)^n=1$,
$\kappa_{n-1}/\kappa_n=e^{O(\log n)}$, and
$(1+w_n)^{(p+1)n}=e^{O(1/n)}$.  Thus the prefactor equals
$c^{(p+1)n}e^{O(\log n)}$.

The function $\phi_p(u):=u(1-u^2/4)^{p/2}$ is continuous on $[0,2]$,
vanishes at both endpoints, and has its unique maximum at
$\alpha_p=2/\sqrt{p+1}$.  Its maximum is
\[
 M_p=\frac2{\sqrt{p+1}}\left(\frac p{p+1}\right)^{p/2},
 \qquad M_2=\frac4{3\sqrt3},\quad M_4=\frac{32}{25\sqrt5}.
\]
Bounding the integral in~\eqref{eq:global-laplace-upper} by $2M_p^{n-1}$
gives the global upper estimate.  For the lower estimate, restrict the
polar integral in $\alpha=r/R$ to
$|\alpha-\alpha_p|\le n^{-1}$.  On this compact interval
\eqref{eq:additive-lens} gives
\[
 I_p(R)\ge
 c^{(p+1)n}e^{-O(\log n)}
 \int_{|\alpha-\alpha_p|\le n^{-1}}
        \frac{\phi_p(\alpha)^n}{\alpha}\,d\alpha
 =c^{(p+1)n}M_p^n e^{-O(\log n)}.
\]
The last equality uses the smooth nondegenerate maximum and an interval of
length $2/n$.  The upper and lower bounds prove
\eqref{eq:I2}--\eqref{eq:I4}, with the stated logarithmic errors.

Finally choose a closed neighborhood $J_0$ of $\alpha_p$ strictly inside
$J$.  Since $r/R=(1+w_n)u$, for large $n$ the condition $r/R\notin J$
implies $u\notin J_0$.  On that set
$\phi_p(u)\le M_p e^{-\eta_{p,J}}$ for some $\eta_{p,J}>0$.
Use this restricted maximum in~\eqref{eq:global-laplace-upper} and divide
by the lower bound just proved.  The resulting factor is
$e^{-\eta_{p,J}n+O(\log n)}$, which proves
\eqref{eq:lens-laplace-tail} with any fixed
$0<\delta_{p,J}<\eta_{p,J}$.
\end{proof}

\begin{remark}[Dependence on the shell width]\label{rem:shell-width}
The shell and lens prefactors make the role of the shell width explicit.
Temporarily let $w=w_n>0$ satisfy $nw=o(1)$.
Uniformly for $\alpha=\|\mathbf d\|/R$ in a fixed compact set
$J\Subset(0,2)$, the shell-volume formula and the bipolar formula give
\begin{align}
 P&=2nw\,c^n\bigl(1+O((nw)^2)\bigr),\notag\\
 K_R(\mathbf d)
 &=\frac{4(n-1)\kappa_{n-1}}{\alpha\kappa_n}\,
 w^2c^n(1-\alpha^2/4)^{(n-3)/2}
 \bigl(1+O_J(nw)\bigr).
 \label{eq:width-prefactors}
\end{align}
Indeed, the integration rectangle has area $4w^2R^2$, and the positive
bipolar integrand varies by a factor $1+O_J(nw)$ on it.
Since $\kappa_{n-1}/\kappa_n=\Theta(\sqrt n)$, every fixed choice
$w=n^{-1-\delta}$ with $\delta>0$ changes these logarithmic volume
estimates by only $O_\delta(\log n)$.
For an exponentially smaller width $w=e^{-\varepsilon n}$ with fixed
$\varepsilon>0$, the factor $w^2$ subtracts $2\varepsilon$ from the
continuum lens exponent. Its same-shell zero then occurs at
$c=(2/\sqrt3)e^{2\varepsilon}$.
The probabilistic statements in this paper retain the choice $w_n=n^{-2}$.
\end{remark}

\subsection{Proof of the coverage transition}

\begin{lemma}[First moment of represented targets]\label{lem:represented-first}
For fixed $c>1$, let
$m_c=\exp(n\log(c\sqrt3/2)+o(n))$ be the same-shell lens scale.  If
$Z_{\cL}=|\{\mathbf d\in A_{\cL}(c):r_{\cL,R}(\mathbf d)>0\}|$, then
\begin{equation}\label{eq:represented-first}
 \E Z_{\cL}\le P\bigl(m_c+e^{o(n)}\Xi_n\bigr)=Pm_ce^{o(n)},
\end{equation}
where $\Xi_n$ is defined in \eqref{eq:diagram-correction-budget}.
\end{lemma}

\begin{proof}
Since $\mathbf1_{\{r>0\}}\le r$, expand
$\sum_{\mathbf d\in\cL\cap S_R}r_{\cL,R}(\mathbf d)$ as a marked raw
one-occurrence Rogers sum.  The full-rank term is
$\int_{S_R}K_R(\mathbf d)d\mathbf d=Pm_c e^{o(n)}$.  By the
missing-fiber principle a nonidentity diagram has no free unmarked fiber at
this order.  The geometric and arithmetic estimates of
Proposition~\ref{prop:uniform-diagram-budget} therefore contribute at most
$Pe^{o(n)}\Xi_n$.  Finally,
\[
 c\frac{\sqrt3}{2}>\frac{\sqrt e}{2}
\]
for every fixed $c>1$, so each correction is exponentially smaller than
$m_c$; this proves the last equality.
\end{proof}

\begin{proof}[Proof of Theorem~\ref{thm:coverage}]
Let $m_*:=\inf_{\mathbf d\in S_R}K_R(\mathbf d)$ and
$\omega(\mathbf d):=m_*/K_R(\mathbf d)$, and write
$m_c=\exp(n\log(c\sqrt3/2)+o(n))$ for their common exponential scale.  Suppose first that
$c>2/\sqrt3$, so $m_c=e^{\Omega_c(n)}$.  Put $\delta_c:=\log(c\sqrt3/2)>0$ and choose any fixed integer
$h>\log c/\delta_c$.  Apply
Proposition~\ref{thm:intrinsic-marked-moments} with $\alpha=1$ and relative
deviation $\zeta_n:=1/\log n$.
The Markov threshold $(\zeta_nm_*)^{2h}$ costs
$\zeta_n^{-2h}=(\log n)^{2h}=e^{o(n)}$.
Since $\log m_*=\delta_cn+O(\log n)$, the partition contribution to the
probability of any violation is
\begin{equation}\label{eq:same-shell-partition-failure}
 \zeta_n^{-2h}Pe^{O(h\log(2h))}m_*^{-h}
 =\exp\bigl((\log c-h\delta_c)n+o(n)\bigr)
 =e^{-\Omega_c(n)}.
\end{equation}
The correction term has prefactor $P/m_*=(2/\sqrt3)^{n+o(n)}$.
At fixed dimension this ratio is independent of $R$, since both volumes
scale as $R^n$; thus the correction margins do not deteriorate as $c$
increases.  Multiplying the three parts of $\Xi_n$ gives respectively
\[
 \left(\frac89\right)^n,
 \qquad\left(\sqrt{\frac e3}\right)^n,
 \qquad\left(\frac1{\sqrt3}\right)^n.
\]
All bases are strictly below one.  Hence, except on an event of probability
$e^{-\Omega_c(n)}$, every $\mathbf d\in\cL\cap S_R$ satisfies
\[
 |\omega(\mathbf d)r_{\cL,R}(\mathbf d)-m_*|
 \le\frac{m_*}{\log n}.
\]
Since $\omega(\mathbf d)K_R(\mathbf d)=m_*$ exactly, division by
$\omega(\mathbf d)$ proves \eqref{eq:same-shell-lens-comparison}; positivity
then proves \eqref{eq:exact-difference-inclusion}.  Finally
$A_{\cL}(c)=-A_{\cL}(c)$, so the conditions
$\mathbf x-\mathbf d\in A_{\cL}(c)$ and
$\mathbf d-\mathbf x\in A_{\cL}(c)$ are equivalent.  This proves
\eqref{eq:same-shell-endpoint-identity} directly from the definition of
$r_{\cL,R}$.

If $1<c<2/\sqrt3$, then $m_c=e^{-\Omega_c(n)}$.
Lemma~\ref{lem:represented-first} and Markov's inequality give
$Z_{\cL}\le P\sqrt{m_c}$ outside an event of exponentially small
probability.  Lemma~\ref{lem:thin-shell-count} gives
$|A_{\cL}(c)|=(1+o(1))P$ with the same quality of probability.  Therefore
$\Theta_{\cL}(c)\le\sqrt{m_c}(1+o(1))=e^{-\Omega_c(n)}$.
\end{proof}

\begin{remark}[Why the same-shell ratio needs no endpoint gap]
At $\alpha=1$ the generic centered endpoint condition
\eqref{eq:generic-endpoint-gap} is not available: the same-shell geometry has
additional low-parent identities.  This creates no gap in
Theorem~\ref{thm:coverage}.  Its proof uses only the target-multiplicity half
of Proposition~\ref{thm:intrinsic-marked-moments}; central symmetry then converts
that estimate into the exact endpoint identity
\eqref{eq:same-shell-endpoint-identity}.
\end{remark}

\begin{corollary}[Affine same-shell autocorrelation]
\label{cor:affine-same-shell}
Fix $c>2/\sqrt3$ and put $R=c\Rstar$.  With probability
$1-e^{-\Omega_c(n)}$ over
$(\cL,\boldsymbol\tau)\sim\widetilde X_n$, simultaneously for every
$\mathbf d\in(\boldsymbol\tau+\cL)\cap S_R$,
\begin{equation}\label{eq:affine-same-shell-lens}
 \left(1-\frac1{\log n}\right)K_R(\mathbf d)
 \le r^{\rm aff}_{\cL,\boldsymbol\tau,R}(\mathbf d)
 \le\left(1+\frac1{\log n}\right)K_R(\mathbf d).
\end{equation}
Consequently
\begin{equation}\label{eq:affine-same-shell-inclusion}
 (\boldsymbol\tau+\cL)\cap S_R
 \subseteq
 \bigl((\boldsymbol\tau+\cL)\cap S_R\bigr)-(\cL\cap S_R).
\end{equation}
For every $\mathbf x\in(\boldsymbol\tau+\cL)\cap S_R$, the affine endpoint
degree
\[
 D^{\rm aff}_{R,1}(\mathbf x)
 :=\#\{\mathbf c\in\cL\cap S_R:\mathbf x-\mathbf c\in S_R\}
\]
satisfies, uniformly in $\mathbf x$,
\[
 D^{\rm aff}_{R,1}(\mathbf x)
 =\bigl(1+O(1/\log n)\bigr)K_R(\mathbf x).
\]
\end{corollary}

\begin{proof}
Apply Proposition~\ref{thm:affine-marked-moments} with $\alpha=1$ and any
fixed integer $h>\log c/\log(c\sqrt3/2)$.  The affine endpoint gap is
$|1-\alpha^2/2|=1/2$, while both the target lens mean and the endpoint cap
mean have exponential scale $\exp(n\log(c\sqrt3/2)+o(n))$.  The maximal Markov
calculation in the proof of Theorem~\ref{thm:coverage} therefore applies to
both estimates with deviation $1/\log n$: its additional
$(\log n)^{2h}=e^{o(n)}$ Markov cost leaves the same three strict
correction margins.  Dividing the
target estimate by the flattening weight gives
\eqref{eq:affine-same-shell-lens}; each represented target satisfies
$\mathbf d=(\mathbf d+\mathbf c)-\mathbf c$, which gives
\eqref{eq:affine-same-shell-inclusion}.  For endpoints, the continuum cap
equals $K_R(\mathbf x)$ by the substitution
$\mathbf y=\mathbf x-\mathbf c$, and removing the uniformly
$1+O(1/n)$ flattening weight gives the final relative error.
\end{proof}

\begin{remark}[Near-critical window]\label{rem:coverage-critical-window}
If
$c=(2/\sqrt3)\exp(\beta/\log n)$, then
$\log m_c=\beta n/\log n+o(n/\log n)$.  The proof above handles each fixed
$\beta>0$ after increasing the constant in $h\asymp\log n$, and each fixed
$\beta<0$ is on the sparse side.
At the exact radius $c=2/\sqrt3$, \eqref{eq:width-prefactors} with
$w=n^{-2}$ gives $K_R(\mathbf d)=\Theta(n^{-5/2})$ uniformly for
$\mathbf d\in S_R$.  An integer representation count cannot be relatively
close to this positive vanishing mean.  Moreover, the raw first-moment
calculation in Lemma~\ref{lem:represented-first}, retaining this prefactor
and its exponentially small remainder, gives
$\E Z_{\cL}/P=O(n^{-5/2})$.  Together with the shell count and Markov this
implies $\Theta_{\cL}(2/\sqrt3)\to0$ in probability.
The uniform dense-side argument fails here because its mean no longer
supplies exponential decay in the Markov bound.  Resolving a finer
same-shell critical window requires the polynomial prefactors and a
correspondingly finer analysis of the distribution of the parent counts.
\end{remark}

\section{Additive energy and collision entropy}
\label{sec:energy}

Fix $c>1$ and $R=c\Rstar$, with the notation of
Section~\ref{sec:additive-phases}.  The next result compares the forced
diagonal energy with the squared-lens contribution.

Define the ordered additive energy
\begin{equation}\label{eq:additive-energy-def}
 \mathcal E(A_{\cL}(c))
 :=\#\{(\mathbf x_1,\mathbf x_2,\mathbf x_3,\mathbf x_4)
       \in A_{\cL}(c)^4:
 \mathbf x_1-\mathbf x_2=\mathbf x_3-\mathbf x_4\}
 =\sum_{\mathbf d\in\cL}r_{\cL,R}(\mathbf d)^2.
\end{equation}

\begin{theorem}[Additive-energy exponent]\label{thm:energy}
For every fixed $c>1$, with probability tending to one,
\begin{equation}\label{eq:energy-exponent}
 \frac1n\log\mathcal E(A_{\cL}(c))
 =\max\left\{2\log c,
 3\log c+\log\frac4{3\sqrt3}\right\}+o(1).
\end{equation}
The continuum off-diagonal term becomes exponentially dominant for
$c>c_E$, where
\[
 c_E=\frac{3\sqrt3}{4}=1.299038\ldots.
\]
The exponent formula includes $c=c_E$.
More precisely, let $L(c)$ denote the maximum in
\eqref{eq:energy-exponent}.  For every fixed $\varepsilon>0$ there is
$b=b(c,\varepsilon)>0$ such that
\[
 \Prb\!\left[\mathcal E(A_{\cL}(c))\notin
 [e^{n(L(c)-\varepsilon)},e^{n(L(c)+\varepsilon)}]\right]\le e^{-bn}
\]
for all sufficiently large $n$.
\end{theorem}

The Euclidean constant $4/(3\sqrt3)=4\sqrt3/9$ already occurs in continuum
ball energy \cite{Shao2013}.  Long's Theorem~1.1 compares normalized energy
of integer-lattice points in Euclidean balls with that continuum rate when
the radius divided by the square root of the dimension tends to infinity
\cite{Long2026}.  The contribution here is the quenched exponent for thin
Haar shells at $R=c\Rstar$ and its competition with forced lattice
diagonals.  The isolated continuum constant is inherited.

Kaminaga's Gibbs-ensemble results concern soft quadratic-energy weights,
including sign-class edge statistics and primitive-direction thermal
concentration \cite{Kaminaga2026}.  These are different observables from
the hard-shell incidence and energy transition here.

The upper bound follows from a fixed-order remainder estimate.  For the
lower bound above the threshold, concentration of the squared lens kernel
is combined with an integrated approximation to the parent counts.  We
first prove the two second-moment estimates shared by the energy proof
and the total-variation comparison in Section~\ref{sec:convolution}.

\subsection{Shared second-moment estimates}

\begin{lemma}[An $L^2$ Rogers variance bound]
\label{lem:l2-rogers-variance}
Let $f:\R^n\to\R$ be bounded, compactly supported, and square integrable,
and redefine $f(0)=0$ if necessary.  For all sufficiently large $n$,
\begin{equation}\label{eq:l2-rogers-variance}
 \Var_{\cL}\left(\sum_{\mathbf v\in\cL\setminus\{0\}}f(\mathbf v)\right)
 \le\bigl(2+O(2^{-n/2})\bigr)
       \int_{\R^n}|f(\mathbf x)|^2d\mathbf x.
\end{equation}
\end{lemma}

\begin{proof}
Apply the order-two Rogers formula to $f(\mathbf x)f(\mathbf y)$, first to
its positive and negative parts if $f$ is signed.  The full-rank term is the
squared mean and cancels in the variance.  Rank one has
$\mathbf y=(a/q)\mathbf x$, with $q\ge1$, $a\ne0$, $(a,q)=1$, and
coefficient $q^{-n}$.  Cauchy--Schwarz and the substitution
$\mathbf z=(a/q)\mathbf x$ give
\[
 \left|\int f(\mathbf x)f((a/q)\mathbf x)d\mathbf x\right|
 \le\left(\frac q{|a|}\right)^{n/2}\|f\|_2^2.
\]
Therefore, even after dropping coprimality,
\begin{align*}
 \Var(\widehat f)
 &\le2\|f\|_2^2
   \sum_{q\ge1}\sum_{a\ge1}(qa)^{-n/2}\\
 &=2\zeta(n/2)^2\|f\|_2^2
  =\bigl(2+O(2^{-n/2})\bigr)\|f\|_2^2.
\end{align*}
Here $\zeta$ denotes the Riemann zeta function.
This enumeration includes both signs of $a$ and every rank-one diagram, so
there is no unsummed dilation remainder.
\end{proof}

\begin{lemma}[Bandwise integrated $L^2$ approximation]
\label{lem:bandwise-l2-approximation}
Fix $c>1$ and a compact interval $I\Subset(0,2)$ such that
\begin{equation}\label{eq:dense-window-condition}
 \inf_{\alpha\in I}c\sqrt{1-\alpha^2/4}>1.
\end{equation}
Put $R=c\Rstar$ and
$T_I:=\{\mathbf d:\|\mathbf d\|/R\in I\}$.  Then
\begin{equation}\label{eq:bandwise-l2-expectation}
 \E_{\cL}\sum_{\mathbf d\in\cL\cap T_I}
 |r_{\cL,R}(\mathbf d)-K_R(\mathbf d)|^2
 \le e^{o(n)}P^2.
\end{equation}
Consequently, for every fixed $\tau>0$,
\begin{equation}\label{eq:bandwise-l2-tail}
 \Prb\left[
  \sum_{\mathbf d\in\cL\cap T_I}
  |r_{\cL,R}(\mathbf d)-K_R(\mathbf d)|^2
  >e^{\tau n}P^2\right]
 \le e^{-\tau n+o(n)}.
\end{equation}
Both assertions remain valid for a predetermined union of a fixed number of
such intervals.
\end{lemma}

\begin{proof}
Write $I=[a,b]$ and cover $T_I$ by the $O(n^2)$ standard closed micro-shells
$T_i=S_{\rho_i}$ of~\eqref{eq:standard-shell-cover}.  For all sufficiently
large $n$, their union lies in a fixed slightly larger compact interval
$I^+\Subset(0,2)$ on which the strict density condition
\eqref{eq:dense-window-condition} still holds.  Every target is covered,
and the overlap multiplicity is at most two.  Put
$m_i:=\inf_{\mathbf d\in T_i}K_R(\mathbf d)$, and
$\omega_i(\mathbf d):=m_i/K_R(\mathbf d)$.  The coarea derivative bound in
Lemma~\ref{lem:thin-shell-coarea} gives, uniformly in $i$,
\begin{equation}\label{eq:micro-shell-weight-uniform}
 \omega_i(\mathbf d)=1+O(1/n)\qquad(\mathbf d\in T_i).
\end{equation}
The density condition on $I^+$ and the same coarea formula give
$m_i\ge e^{a_In}$ for some $a_I>0$ and all large $n$.

Apply the target half of
Proposition~\ref{thm:intrinsic-marked-moments} with $h=1$ to each full shell
$T_i$.  Since $m_i>1$, its two terms give, for a fixed $C>0$,
\begin{align*}
 &\E\sum_{\mathbf d\in\cL\cap T_i}
   |\omega_i(\mathbf d)r_{\cL,R}(\mathbf d)-m_i|^2\\
 &\qquad\le
 C\vol(T_i)m_i\bigl(1+n^C\Xi_n\bigr)
 \le e^{o(n)}\vol(T_i)m_i.
\end{align*}
Pointwise on $T_i$,
\[
 \omega_i(\mathbf d)r_{\cL,R}(\mathbf d)-m_i
 =\omega_i(\mathbf d)
   \bigl(r_{\cL,R}(\mathbf d)-K_R(\mathbf d)\bigr).
\]
Thus~\eqref{eq:micro-shell-weight-uniform} permits removal of the weight.
Since the errors are nonnegative, the covering inequality and
$m_i\vol(T_i)\le\int_{T_i}K_R(\mathbf d)d\mathbf d$ give
\begin{align*}
 &\E\sum_{\mathbf d\in\cL\cap T_I}
       |r_{\cL,R}(\mathbf d)-K_R(\mathbf d)|^2\\
 &\qquad\le e^{o(n)}\sum_i\int_{T_i}K_R(\mathbf d)d\mathbf d
 \le 2e^{o(n)}\int_{\R^n}K_R(\mathbf d)d\mathbf d
 =e^{o(n)}P^2.
\end{align*}
Here the second inequality is the bounded-overlap assertion in
\eqref{eq:standard-shell-cover}, and the last identity is
\eqref{eq:kernel-mass}.  This proves
\eqref{eq:bandwise-l2-expectation}; Markov's inequality proves
\eqref{eq:bandwise-l2-tail}.  A fixed union changes only the subexponential
factor.
\end{proof}

\subsection{Energy remainder and squared-kernel estimates}

\begin{lemma}[Fixed-order Rogers remainder for the energy]
\label{lem:energy-remainder}
For fixed $c>1$,
\begin{equation}\label{eq:energy-expectation}
 \E\mathcal E(A_{\cL}(c))
 \le I_2(R)+e^{o(n)}P^2.
\end{equation}
\end{lemma}

\begin{proof}
First interchange the dummy variables $\mathbf x_3$ and $\mathbf x_4$ in
\eqref{eq:additive-energy-def}; the relation then reads
$\mathbf x_1-\mathbf x_2=\mathbf x_4-\mathbf x_3$.  Eliminating
$\mathbf x_4$ gives the order-three lattice sum with kernel
\[
 F(\mathbf x_1,\mathbf x_2,\mathbf x_3)
 =\prod_{j=1}^3\mathbf1_{S_R}(\mathbf x_j)
  \mathbf1_{S_R}(\mathbf x_1-\mathbf x_2+\mathbf x_3).
\]
By Lemma~\ref{lem:rogers-pivot-facts}, the only rank-three diagram is the
$q=1$ identity, and its integral is precisely $I_2(R)$.  It remains to sum
all ranks $r\le2$.  Every such diagram has at most two free
shell-restricted pivots.  Dropping all dependent-column indicators bounds
its pivot integral by $P^r\le P^2$ because $P=e^{\Omega_c(n)}$.

We now justify that this $P^2$ majorant is summable over the infinitely many
admissible integer matrices.  There are only finitely many pivot patterns at
order three.  In rank one, every nonpivot column equals
$(a_j/q)\mathbf v_1$.  Comparable input and output norms force
$|a_j|\le Cq$, so the number of numerator arrays for fixed $q$ is
$O(q^3)$.  The $q\ge2$ contribution is therefore at most
\[
 P^2\sum_{q\ge2}q^{-n+3}=2^{-n+O(1)}P^2.
\]
For $q=1$ the same support bound leaves only $O(1)$ rank-one arrays; their
total pivot integral is $O(P)\le O(P^2)$.

Consider rank two and write the free pivots as
$\mathbf v_i=r_i\mathbf u_i$, $i=1,2$.  Under the product measure of their
retained shell factors, conditioning on $r_1,r_2$ leaves independent uniform
$\mathbf u_1,\mathbf u_2\in S^{n-1}$.  For each nonpivot column let
$b_j/q\in q^{-1}\mathbb Z^2$ be its coefficient vector.  Retaining its
output shell constraint gives
\[
 \left\|\sum_{i=1}^2 c_{ij}\mathbf u_i\right\|^2=1+O(w_n),
 \qquad c_{ij}:=\frac{b_{ij}r_i}{qR},
 \qquad \|c_j\|^2=(1+O(w_n))\|b_j/q\|^2.
\]
The last estimate follows from $r_i/R=1+O(w_n)$, uniformly in $q$ and the
coefficient height.  Thus Lemma~\ref{fs-lem:general-spherical-penalty}
applies with $d_\perp=n$.  Put a matrix in the bounded bin $H=1$ when
$\max_j\|b_j\|/q<2$; otherwise let $H\ge2$ be the unique dyadic number
satisfying
\[
 H\le\max_j\|b_j\|/q<2H.
\]
At fixed $(q,H)$ all entries have size $O(qH)$, so even after ignoring the
row-echelon zeros and primitivity there are at most $(CqH)^6$ matrices.  If
$H\ge2$ and the height is attained by a column with two nonzero
coefficients, Lemma~\ref{fs-lem:general-spherical-penalty} and
Lemma~\ref{fs-lem:fiber-product-majorization} bound its complete pivot
integral by
\[
 P^2e^{O(\log n)}
 \left(\frac{e+o(1)}{H^2}\right)^{(n-O(1))/2}.
\]
If the height-attaining column has only one nonzero coefficient, support is
impossible for $H\ge2$: a single shell pivot scaled by a factor at least two
cannot return to the same relative-width-$n^{-2}$ shell.  Consequently
\begin{align*}
 &\sum_{H=2,4,\ldots}(CH)^6
 \left(\frac{e+o(1)}{H^2}\right)^{(n-O(1))/2}\\
 &\hspace{25mm}\le
 e^{O(\log n)}\left(\frac e4+o(1)\right)^{n/2}.
\end{align*}

For $q\ge2$, the bounded-height bin has at most $O(q^6)$ matrices and the
exact Rogers coefficient satisfies $c_D\le q^{-n}$.  The large-height bins
have the same factor in addition to the preceding summable angular tail.
Thus their complete contribution is bounded by
\[
 P^2e^{O(\log n)}
 \sum_{q\ge2}q^{-n+6}
 \left[1+\left(\frac e4+o(1)\right)^{n/2}\right]
 =2^{-n+O(\log n)}P^2.
\]
For $q=1$, the bounded-height bin contains only finitely many signed
height-one matrices and contributes $O(P^2)$; the bins $H\ge2$ contribute
$e^{O(\log n)}(e/4+o(1))^{n/2}P^2$.  Combining ranks one and two gives an
$e^{o(n)}P^2$ lower-rank remainder, proving
\eqref{eq:energy-expectation}.
\end{proof}

\begin{lemma}[Concentration of the squared lens kernel]
\label{lem:squared-kernel-concentration}
Let $J$ be a sufficiently small fixed interval around $2/\sqrt3$.  Then,
with probability $1-e^{-\Omega_c(n)}$,
\begin{equation}\label{eq:squared-kernel-concentration-sharp}
 \sum_{\substack{\mathbf d\in\cL\\\|\mathbf d\|/R\in J}}
 K_R(\mathbf d)^2
 =I_2(R)\bigl(1+O(e^{-\eta_cn})\bigr)
\end{equation}
for some $\eta_c>0$.
\end{lemma}

\begin{proof}
Let
$f_J(\mathbf d):=\mathbf1_{\{\|\mathbf d\|/R\in J\}}K_R(\mathbf d)^2$.
The global tail bound~\eqref{eq:lens-laplace-tail} with $p=2$ gives a
$\delta_J>0$ such that
\[
 \int f_J(\mathbf d)d\mathbf d
 =I_2(R)\bigl(1+O(e^{-\delta_Jn})\bigr).
\]
This is the mean of the lattice sum by Siegel.  Applying
Lemma~\ref{lem:l2-rogers-variance} to $f_J$ gives the explicit variance
bound
\[
 \Var\left(\sum_{\mathbf d\in\cL\setminus\{0\}}f_J(\mathbf d)\right)
 \le(2+o(1))\int_JK_R(\mathbf d)^4d\mathbf d
 \le(2+o(1))I_4(R).
\]
Equations \eqref{eq:I2}--\eqref{eq:I4} give
\[
 \frac{I_4(R)}{I_2(R)^2}
 =\exp\!\left(n\log\frac{54}{25\sqrt5\,c}+O(\log n)\right),
\]
where $54/(25\sqrt5)<1$, so the leading exponent is negative for every
$c>1$.  Choose $\eta_c>0$ smaller than $\delta_J/2$ and than
$\tfrac14\log(25\sqrt5\,c/54)$.
Chebyshev at relative threshold $e^{-\eta_cn}$ then has failure
$e^{-\Omega_c(n)}$ and proves
\eqref{eq:squared-kernel-concentration-sharp}.  The zero lattice point is
outside $J$.
\end{proof}

\subsection{Proof of the energy exponent and its consequences}

\begin{proof}[Proof of Theorem~\ref{thm:energy}]
Fix $\varepsilon>0$.  Lemma~\ref{lem:energy-remainder} and
\eqref{eq:I2} give
$\E\mathcal E(A_{\cL}(c))\le e^{nL(c)+o(n)}$.
Markov at $e^{n(L(c)+\varepsilon)}$ gives upper-tail failure
$e^{-\varepsilon n+o(n)}$.  Since
$r_{\cL,R}(\mathbf d)$ is a nonnegative integer,
\[
 \mathcal E(A_{\cL}(c))=\sum_{\mathbf d}r_{\cL,R}(\mathbf d)^2
 \ge\sum_{\mathbf d}r_{\cL,R}(\mathbf d)=|A_{\cL}(c)|^2
 =c^{2n+o(n)}
\]
with high probability; this is the diagonal lower bound.

Suppose $c>c_E$.  Restrict targets to the interval $J$ in
Lemma~\ref{lem:squared-kernel-concentration}.  Since
\[
 c_E\sqrt{1-\frac{(2/\sqrt3)^2}{4}}
 =\frac{3\sqrt3}{4}\sqrt{\frac23}>1,
\]
$J$ may be chosen sufficiently small that the dense-window condition
\eqref{eq:dense-window-condition} also holds.  Put
\begin{align*}
 E_J&:=\sum_{\substack{\mathbf d\in\cL\\\|\mathbf d\|/R\in J}}
 |r_{\cL,R}(\mathbf d)-K_R(\mathbf d)|^2,\\
 S_J&:=\sum_{\substack{\mathbf d\in\cL\\\|\mathbf d\|/R\in J}}
 K_R(\mathbf d)^2.
\end{align*}
Lemma~\ref{lem:bandwise-l2-approximation} says that, for every fixed
$\tau>0$,
$E_J\le e^{\tau n}P^2$ except with probability $e^{-\tau n+o(n)}$.
Lemma~\ref{lem:squared-kernel-concentration} simultaneously gives
$S_J=I_2(R)(1+O(e^{-\eta_cn}))$ for some $\eta_c>0$.

Now
\begin{equation}\label{eq:energy-continuum-diagonal-ratio}
 \frac{I_2(R)}{P^2}
 =\exp\!\left(n\log\frac{4c}{3\sqrt3}+O(\log n)\right).
\end{equation}
For $c>c_E=3\sqrt3/4$, choose
$0<\tau<\tfrac12\log(4c/(3\sqrt3))$.  Then
$\sqrt{E_J}/\sqrt{S_J}=e^{-\Omega_c(n)}$.  The reverse triangle inequality
in $\ell^2(J\cap\cL)$ gives
\begin{align*}
 \left(\sum_{\substack{\mathbf d\in\cL\\\|\mathbf d\|/R\in J}}
 r_{\cL,R}(\mathbf d)^2\right)^{1/2}
 &\ge \sqrt{S_J}-\sqrt{E_J}\\
 &=\sqrt{I_2(R)}\bigl(1-e^{-\Omega_c(n)}\bigr).
\end{align*}
Squaring proves
$\mathcal E(A_{\cL}(c))\ge I_2(R)(1-e^{-\Omega_c(n)})$, and hence the
continuum lower exponent above $c_E$.  Below the threshold the diagonal
term already matches the upper exponent, and at equality the two exponents
agree.  These lower bounds have failure $e^{-\Omega_c(n)}$, so together
with the upper bound they prove the assertion for every fixed tolerance
$\varepsilon$.  Letting the tolerance tend to zero sufficiently slowly
gives \eqref{eq:energy-exponent} with probability tending to one.
The loss of a fixed exponential failure rate in that $o(1)$ formulation
comes from the vanishing Markov slack in the upper bound.
\end{proof}

\begin{corollary}[Collision entropy and effective difference support]
\label{cor:collision-entropy}
Let
\(
 \Gamma_{\cL,c}:=U_{\cL,c}*\widetilde U_{\cL,c}
\)
and define its order-two R\'enyi, or collision, entropy by
\[
 H_2(\Gamma_{\cL,c})
 :=-\log\sum_{\mathbf d\in\cL}\Gamma_{\cL,c}(\mathbf d)^2.
\]
For every fixed $c>1$, with probability tending to one,
\begin{equation}\label{eq:collision-entropy-exponent}
 \frac1nH_2(\Gamma_{\cL,c})
 =\min\left\{2\log c,\log\frac{3\sqrt3\,c}{4}\right\}+o(1).
\end{equation}
Thus the collision-effective number of differences is
\begin{equation}\label{eq:collision-effective-support}
 \exp(H_2(\Gamma_{\cL,c}))
 =\exp\!\left(n\min\left\{2\log c,
                 \log\frac{3\sqrt3\,c}{4}\right\}+o(n)\right).
\end{equation}
Moreover, on the event $N=|A_{\cL}(c)|>0$ one has the deterministic bound
\begin{equation}\label{eq:difference-set-lower}
 |A_{\cL}(c)-A_{\cL}(c)|
 \ge \exp(H_2(\Gamma_{\cL,c})),
\end{equation}
and hence, for $1<c\le3\sqrt3/4$, including equality,
\begin{equation}\label{eq:difference-set-exponent}
 \frac1n\log|A_{\cL}(c)-A_{\cL}(c)|=2\log c+o(1)
\end{equation}
with probability tending to one.
\end{corollary}

\begin{proof}
Work first on $N=|A_{\cL}(c)|>0$.  The exact identity
\(
 \Gamma_{\cL,c}(\mathbf d)=r_{\cL,R}(\mathbf d)/N^2
\)
gives
\begin{equation}\label{eq:collision-energy-identity}
 H_2(\Gamma_{\cL,c})=4\log N-\log\mathcal E(A_{\cL}(c)).
\end{equation}
By the thin-shell count, this event has probability
$1-e^{-\Omega_c(n)}$, and $\log N=n\log c+o(n)$ with high probability.
Substitution of Theorem~\ref{thm:energy} into
\eqref{eq:collision-energy-identity} proves
\eqref{eq:collision-entropy-exponent}.  For every finitely supported
probability measure $\nu$, Cauchy--Schwarz gives
\(
 1\le|\operatorname{supp}\nu|\sum_x\nu(x)^2
\).
On $N>0$ the support of $\Gamma_{\cL,c}$ is exactly
$A_{\cL}(c)-A_{\cL}(c)$, so this proves
\eqref{eq:difference-set-lower}.  The empty-shell convention
$\Gamma_{\cL,c}=\delta_0$ is not used in this deterministic inequality.
In the stated range its lower
exponent is $2\log c$, while the deterministic upper bound
$|A_{\cL}(c)-A_{\cL}(c)|\le|A_{\cL}(c)|^2$ has the same exponent; this proves
\eqref{eq:difference-set-exponent}.
\end{proof}

\begin{remark}[Interpretation and provenance]
Below the energy threshold the number of distinct differences has the
maximal exponential rate $2\log c$, as in~\eqref{eq:difference-set-exponent}.
This global support statement allows exponentially large individual
multiplicities: Theorem~\ref{thm:coverage} gives this for every same-shell
target when $2/\sqrt3<c<3\sqrt3/4$, and the zero difference has multiplicity
$N$.  Above the threshold,
continuum lens overlaps reduce the effective support exponent.  The
corollary is an algebraic consequence of Theorem~\ref{thm:energy}; it needs
no additional Rogers estimate and carries no separate novelty claim.
\end{remark}

\section{The quenched shell-convolution transition}
\label{sec:convolution}

Fix $c>1$ and $R=c\Rstar$, and retain $A_{\cL}(c)$, $N$, and $P$ from
Section~\ref{sec:additive-phases}.  We compare the complete difference law
with the lattice-sampled lens law of Section~\ref{sec:comparison-framework}.
Recall $U_{\cL,c}=U_{\cL,R}$ for $R=c\Rstar$: it is uniform on
$A_{\cL}(c)$ when the shell is nonempty, and equals $\delta_0$ otherwise.
Let $\widetilde U_{\cL,c}$ denote its reflection.  Define a continuum-kernel law
on the lattice by
\begin{equation}\label{eq:Pi-kernel}
 \Pi_{\cL,R}(\mathbf d):=\frac{K_R(\mathbf d)}{Z_{\cL,R}},
 \qquad Z_{\cL,R}:=\sum_{\mathbf z\in\cL}K_R(\mathbf z).
\end{equation}
On $|A_{\cL}(c)|>0$, the actual difference law assigns mass
$r_{\cL,R}(\mathbf d)/|A_{\cL}(c)|^2$ to $\mathbf d$; on the empty
event it is $\delta_0$.  The comparator is always defined because
$Z_{\cL,R}\ge K_R(0)=P>0$.

\begin{theorem}[Quenched shell-convolution transition]
\label{thm:convolution}
For every fixed $c>1$, $c\ne\sqrt2$, with probability
$1-e^{-\Omega_c(n)}$,
\[
 \TV\bigl(U_{\cL,c}*\widetilde U_{\cL,c},\Pi_{\cL,R}\bigr)
 =1-e^{-\Theta_c(n)}\quad(1<c<\sqrt2),
\]
while
\[
 \TV\bigl(U_{\cL,c}*\widetilde U_{\cL,c},\Pi_{\cL,R}\bigr)
 \le e^{-\Omega_c(n)}\quad(c>\sqrt2).
\]
\end{theorem}

El-Baz--Marklof--Vinogradov obtain mixed-moment convergence and Poissonian
two-point correlations for directions in a fixed planar affine lattice
under a Diophantine hypothesis
\cite[Theorem~2 and Corollary~4]{ElBazMarklofVinogradov2015}.
Kim--Marklof extend the pair-correlation theorem to every fixed dimension
at least three for every irrational shift
\cite[Theorem~1.1]{KimMarklof2025}.  Those limits send the radius to infinity
in fixed dimension and resolve shrinking angular separations.  The present
limit sends the dimension to infinity and compares the entire discrete
difference law in total variation, for a hard shell of relative half-width
$n^{-2}$.  Its comparator is sampled on the same lattice.

The proof combines the integrated $L^2$ estimate of
Lemma~\ref{lem:bandwise-l2-approximation} with normalization and
angular-tail bounds.  The shared variance estimate was also established
in Section~\ref{sec:energy}.

\subsection{Kernel normalization and angular tails}

\begin{lemma}[Kernel normalization and angular tails]
\label{lem:kernel-normalization}
For fixed $c>1$, there is $\eta_c>0$ such that, with probability
$1-e^{-\Omega_c(n)}$,
\begin{equation}\label{eq:kernel-normalization}
 Z_{\cL,R}=P^2\bigl(1+O(e^{-\eta_c n})\bigr).
\end{equation}
Moreover, for every sufficiently small fixed neighborhood $I$ of $\sqrt2$,
chosen after $c$ is fixed, both
\begin{align*}
 \frac1{P^2}\sum_{\mathbf x,\mathbf y\in A_{\cL}(c)}
 \mathbf1_{\{\|\mathbf x-\mathbf y\|/R\notin I\}},\qquad
 \frac1{Z_{\cL,R}}\sum_{\substack{\mathbf d\in\cL\\
                     \|\mathbf d\|/R\notin I}}K_R(\mathbf d)
\end{align*}
are $e^{-\Omega_c(n)}$.
\end{lemma}

\begin{proof}
Equation~\eqref{eq:kernel-mass} and Siegel's formula show that the nonzero
part of $Z_{\cL,R}$ has mean $P^2$; the zero term is $K_R(0)=P=o(P^2)$.
Apply Lemma~\ref{lem:l2-rogers-variance} to $f=K_R$, whose squared
$L^2$ norm is $I_2(R)$.  It gives the explicit bound
\begin{equation}\label{eq:kernel-normalization-variance}
 \Var(Z_{\cL,R})\le(2+O(2^{-n/2}))I_2(R).
\end{equation}
By \eqref{eq:I2},
\[
 \frac{I_2(R)}{P^4}
 =\left(\frac4{3\sqrt3\,c}\right)^{n+o(n)}=e^{-\Omega_c(n)}.
\]
For completeness, put
\[
 \delta_c:=-\log\left(\frac4{3\sqrt3\,c}\right)>0
\]
and choose $0<\theta_c<\min\{\delta_c/3,\log(c)/2\}$.  Chebyshev gives
\begin{align*}
 &\Prb\left[
  \left|Z_{\cL,R}-(P^2+P)\right|>e^{-\theta_cn}P^2\right]\\
 &\qquad\le e^{2\theta_cn}\frac{\Var(Z_{\cL,R})}{P^4}
 \le e^{-(\delta_c-2\theta_c)n+o(n)}.
\end{align*}
Since $P/P^2=c^{-n+o(n)}$, both the deterministic zero term and the random
deviation are $e^{-\Omega_c(n)}P^2$.  After decreasing the exponent to absorb
the $o(n)$ terms, one may take any
$\eta_c<\min\{\theta_c,\log c\}$.  This proves
\eqref{eq:kernel-normalization} with an exponential relative error and an
exponential failure probability, rather than merely an $o(1)$ error.

Let $T_{\rm out}$ denote the actual ordered pair count with
$\|\mathbf x-\mathbf y\|/R\notin I$.  Its full-rank Rogers term is the
continuum angular tail, at most $P^2e^{-\delta_In}$ for some $\delta_I>0$.
The rank-one terms include every rational dilation
$\mathbf y=(a/q)\mathbf x$, where $q\ge1$, $a\ne0$, and $(a,q)=1$.
For $q=1$ only $a=\pm1$ can meet both thin shells, contributing at most
$2P$.  For $q\ge2$, the enumeration in
Lemma~\ref{lem:thin-shell-count}, after dropping the tail indicator, gives
\[
 CP\sum_{q\ge2}q^{-n}(1+q/n^2)=O(2^{-n})P.
\]
Consequently
\[
 \E T_{\rm out}\le P^2e^{-\delta_In}+(2+O(2^{-n}))P.
\]
Choose $0<\eta<\tfrac12\min\{\delta_I,\log c\}$.  Since
$P=c^{n+o(n)}$, Markov at $P^2e^{-\eta n}$ gives both an exponentially
small tail mass and an exponentially small failure probability.
For the sampled kernel tail, Siegel gives a nonzero-part mean at most
$P^2e^{-\delta_In}$, and its zero term is at most $P$.
The same choice of $\eta$ and Markov therefore apply.
Combining with \eqref{eq:kernel-normalization} proves both tail assertions.
\end{proof}

\subsection{Proof of the convolution transition}

\begin{proof}[Proof of Theorem~\ref{thm:convolution}]
Suppose first that $c>\sqrt2$.  Choose a small fixed interval $I$ around
$\sqrt2$ so that
\eqref{eq:dense-window-condition} holds.  Put
\[
 M_I:=|\{\mathbf d\in\cL:\|\mathbf d\|/R\in I\}|,
 \qquad Q_I:=\vol\{\mathbf d:\|\mathbf d\|/R\in I\}.
\]
Fix a small $\tau>0$.  Lemma~\ref{lem:bandwise-l2-approximation} and Markov
give, except with probability $e^{-\tau n+o(n)}$,
\begin{equation}\label{eq:convolution-l2-event}
 \sum_{\substack{\mathbf d\in\cL\\\|\mathbf d\|/R\in I}}
 |r_{\cL,R}(\mathbf d)-K_R(\mathbf d)|^2
 \le e^{\tau n}P^2.
\end{equation}
Siegel's theorem and Markov's inequality independently give
$M_I\le e^{\tau n}Q_I$ with the same quality of failure.  Hence
Cauchy--Schwarz gives the fully normalized estimate
\begin{align}
 \frac1{P^2}
 \sum_{\substack{\mathbf d\in\cL\\\|\mathbf d\|/R\in I}}
 |r_{\cL,R}(\mathbf d)-K_R(\mathbf d)|
 &\le e^{\tau n}\frac{\sqrt{Q_I}}P.             \label{eq:convolution-l1-event}
\end{align}
If $I=[\sqrt2-\delta,\sqrt2+\delta]$, then
\[
 \frac{\sqrt{Q_I}}P
 \le\left(\sqrt{\frac{\sqrt2+\delta}{c}}+o(1)\right)^n.
\]
At $\delta=0$ this base is $2^{1/4}/\sqrt c<1$.  Choose first $\delta$ and
then $\tau$ sufficiently small; the right side of
\eqref{eq:convolution-l1-event} is $e^{-a_cn}$ for some $a_c>0$.  This
display is the precise source of the threshold $c=\sqrt2$.

Lemma~\ref{lem:kernel-normalization} puts exponentially small mass outside
$I$ for both the actual pair law and the sampled continuum kernel.  It
remains only to compare normalizers.  Choose
$0<\eta_N<\tfrac12\log c$.  Lemma~\ref{lem:thin-shell-count} and Chebyshev
give
\[
 \Prb[|N_R-P|>e^{-\eta_Nn}P]
 \le e^{2\eta_Nn}\frac{(2+o(1))P}{P^2}
 =e^{-(\log c-2\eta_N)n+o(n)}.
\]
Together with~\eqref{eq:kernel-normalization}, this yields
$N_R^2=P^2(1+O(e^{-\eta n}))$ and
$Z_{\cL,R}=P^2(1+O(e^{-\eta n}))$ for a common $\eta>0$.
On the common high-probability event,
\begin{align*}
 2\TV(\Gamma_{\cL,R},\Pi_{\cL,R})
 &\le \frac1{N_R^2}\sum_{\mathbf d\in\cL}|r_{\cL,R}(\mathbf d)-K_R(\mathbf d)|
       +\frac{|Z_{\cL,R}-N_R^2|}{N_R^2}\\
 &\le e^{-\Omega_c(n)}.
\end{align*}
The first sum is split into $I$ and its complement and bounded by
\eqref{eq:convolution-l1-event} and the two tail estimates.  This proves the
smooth-side upper bound.

Now assume $1<c<\sqrt2$ and choose $I$ so that
$K_R(\mathbf d)\le e^{-\delta n}$ whenever $\|\mathbf d\|/R\in I$.
The actual convolution is supported on
$B_{\cL}:=A_{\cL}(c)-A_{\cL}(c)$, and
$|B_{\cL}|\le|A_{\cL}(c)|^2=P^{2+o(1)}$.  Therefore
\[
 \Pi_{\cL,R}\bigl(B_{\cL}\cap\{\|\mathbf d\|/R\in I\}\bigr)
 \le\frac{P^{2+o(1)}e^{-\delta n}}{Z_{\cL,R}}
 =e^{-\Omega_c(n)}.
\]
The comparison mass outside $I$ is exponentially small by
Lemma~\ref{lem:kernel-normalization}.  Thus
$\Pi_{\cL,R}(B_{\cL})\le e^{-a_cn}$ for some $a_c>0$, whereas the actual convolution gives
$B_{\cL}$ mass one, proving the lower bound
$\TV\ge1-e^{-a_cn}$.  For the matching exponential scale of the
deficit, use the common atom at zero.  The actual difference law and the
comparison law assign it respectively the masses
\[
 \frac{r_{\cL,R}(0)}{|A_{\cL}(c)|^2}
 =\frac1{|A_{\cL}(c)|},
 \qquad
 \frac{K_R(0)}{Z_{\cL,R}}=\frac P{Z_{\cL,R}}.
\]
On the same high-probability event, $N_R=P(1+o(1))$ and
$Z_{\cL,R}=P^2(1+o(1))$.  Hence both displayed masses are at least
$e^{-C_cn}$ for some $C_c<\infty$.  Since for probability measures on a
countable set
\[
 1-\TV(\Gamma,\Pi)=\sum_{\mathbf d}\min\{\Gamma(\mathbf d),\Pi(\mathbf d)\},
\]
their common zero atom gives $1-\TV\ge e^{-C_cn}$.  This proves the stated
$1-e^{-\Theta_c(n)}$ granular behavior.
\end{proof}

\begin{remark}[The critical total-variation radius]\label{rem:tv-critical}
At $c=\sqrt2$, no fixed neighborhood of the typical lag $\sqrt2$
satisfies the dense-window condition, and the limiting base in
\eqref{eq:convolution-l1-event} is one.  The sparse-side proof likewise
has no fixed interval carrying almost all continuum mass on which the lens
is exponentially small.  In fact $g_{\sqrt2}(\sqrt2)=0$ and
$g'_{\sqrt2}(\sqrt2)=-1/\sqrt2$; the usual $n^{-1/2}$ angular fluctuations
therefore vary the logarithmic lens mean on a scale $\sqrt n$.
Determining the boundary law requires estimates in this shrinking angular
window, including the polynomial shell prefactors.  The present
fixed-window estimates do not determine that law.
\end{remark}

\section{Angular marginal universality}
\label{sec:angular-transfer}

The preceding sections analyzed parent counts and difference laws.  We
now return to the shell itself and prove complementary results about its
angular marginals.  Their proof uses the order-two theory, independently
of the growing-order moment argument.  Bounded pair-kernel identities and
one-sided upper tails are recorded separately in
Appendix~\ref{app:angular-kernels}.

Throughout this section assume $n\ge3$, so the centered order-two Rogers
integrals are finite.  Let $\sigma_{n-1}$ denote normalized rotation-invariant measure on
$S^{n-1}$.  For $R>0$ put
\[
 A_R(\cL):=\cL\cap S_R,\qquad
 N_R:=|A_R(\cL)|,\qquad P_R:=\vol(S_R),
 \qquad \widehat{\mathbf x}:=\frac{\mathbf x}{\|\mathbf x\|}.
\]
The shell is separated from the origin, so every displayed direction is
defined.

\subsection{Main angular results}

Define the directional scatter matrix
\[
 Q_R(\cL):=\sum_{\mathbf x\in A_R(\cL)}
 \widehat{\mathbf x}\widehat{\mathbf x}^{\mathsf T},
 \qquad
 C_R(\cL):=\frac n{N_R}Q_R(\cL)
\]
whenever $N_R>0$, and set $C_R=0$ when $N_R=0$.  Thus $C_R=I_n$
is exact isotropy on a nonempty shell.
We write $\|\cdot\|_{\rm op}$ for the Euclidean operator norm.

\begin{theorem}[Quenched dense-shell isometry]
\label{thm:dense-shell-isometry}
There is an absolute constant $C>0$ such that, for every $R>0$ and
$0<\varepsilon\le1/2$,
\begin{equation}
 \Prb_{\cL}\left[N_R=0\ \text{or}\
   \|C_R-I_n\|_{\rm op}>\varepsilon\right]
 \le \frac{C n^2}{\varepsilon^2P_R}.
 \label{eq:dense-shell-isometry-tail}
\end{equation}
The same assertion holds on $\widetilde X_n$ with the shell, $N_R$, $Q_R$,
and $C_R$ defined analogously for $(\boldsymbol\tau+\cL)\cap S_R$.

Equivalently, on the complementary event the analysis map
\[
 \Phi_R:\R^n\longrightarrow\R^{A_R(\cL)},
 \qquad
 (\Phi_R\mathbf u)_{\mathbf x}
 :=\sqrt{\frac n{N_R}}\,
   \langle\mathbf u,\widehat{\mathbf x}\rangle
\]
satisfies
\begin{equation}
 (1-\varepsilon)\|\mathbf u\|^2
 \le\|\Phi_R\mathbf u\|^2
 \le(1+\varepsilon)\|\mathbf u\|^2
 \qquad(\mathbf u\in\R^n).
 \label{eq:dense-shell-analysis-map}
\end{equation}
If $R=c_n\Rstar$ and $c_n\to c>1$, then
\begin{equation}
 P_R=\frac{2c_n^n}{n}\bigl(1+O(n^{-2})\bigr).
 \label{eq:thin-shell-volume-angular}
\end{equation}
Consequently, for every fixed $0<\beta<\tfrac12\log c$, taking
$\varepsilon=e^{-\beta n}$ in
\eqref{eq:dense-shell-isometry-tail} gives failure at most
\begin{equation}
 \exp\bigl(-[\log c-2\beta]n+o(n)\bigr).
 \label{eq:dense-shell-exponential-isometry}
\end{equation}
At $c=\sqrt{4/3}$, every
$\beta<\tfrac14\log(4/3)=0.0719205\ldots$ is admissible.
\end{theorem}

For $k\ge1$, equip $(\R^n)^{\otimes k}$ with its Hilbert--Schmidt norm and,
when $N_R>0$, define
\[
 M^{(k)}_{\cL,R}:=\frac1{N_R}
 \sum_{\mathbf x\in A_R(\cL)}
 \widehat{\mathbf x}^{\otimes k},
 \qquad
 M^{(k)}_{\sigma}:=\int_{S^{n-1}}
 \mathbf u^{\otimes k}\,d\sigma_{n-1}(\mathbf u).
\]

\begin{theorem}[Linear-degree spherical tensor moments]
\label{thm:linear-degree-tensor-moments}
For every even $k=2r\ge2$, put
\begin{equation}
 m_{k,n}:=\|M^{(k)}_{\sigma}\|_F^2
 =\E_{\mathbf U,\mathbf V}
   \langle\mathbf U,\mathbf V\rangle^k
 =\frac{(k-1)!!}{n(n+2)\cdots(n+k-2)},
 \label{eq:spherical-tensor-norm}
\end{equation}
where $\mathbf U,\mathbf V$ are independent and uniform on $S^{n-1}$.
There is an absolute $C>0$ such that, for $0<\varepsilon\le1/2$,
\begin{equation}
 \Prb_{\cL}\left[N_R=0\ \text{or}\
 \|M^{(k)}_{\cL,R}-M^{(k)}_{\sigma}\|_F
   >\varepsilon\|M^{(k)}_{\sigma}\|_F\right]
 \le\frac{C}{\varepsilon^2P_Rm_{k,n}}.
 \label{eq:tensor-moment-tail}
\end{equation}
The same estimate holds for a random-affine shell.  For odd $k$, both
$M^{(k)}_{\cL,R}$ and $M^{(k)}_{\sigma}$ vanish identically in the centered
case; the affine absolute-error estimate remains available directly from
Lemma~\ref{lem:hilbert-angular-second-moment}.

Suppose $R=c_n\Rstar$ with $c_n\to c>1$.  Define, with
$0\log0:=0$,
\begin{equation}
 \psi(\rho):=\rho\log\rho-
 \left(\rho+\frac12\right)\log\left(\rho+\frac12\right)
 -\frac12\log2.
 \label{eq:tensor-rate-function}
\end{equation}
If $\kappa>0$ satisfies
\begin{equation}
 \log c+\psi(\kappa/2)>0,
 \label{eq:tensor-degree-condition}
\end{equation}
then, for every
$0<\beta<\tfrac12[\log c+\psi(\kappa/2)]$, with failure
$e^{-\Omega(n)}$ one has simultaneously for every even
$2\le k\le\kappa n$,
\begin{equation}
 \frac{\|M^{(k)}_{\cL,R}-M^{(k)}_{\sigma}\|_F}
      {\|M^{(k)}_{\sigma}\|_F}
 \le e^{-\beta n}.
 \label{eq:simultaneous-linear-tensor-moments}
\end{equation}
This simultaneous conclusion also holds on $\widetilde X_n$.  At
$c=\sqrt{4/3}$, condition \eqref{eq:tensor-degree-condition} is equivalent
to
\begin{equation}
 \kappa<\kappa_*=0.080964856\ldots,
 \qquad
 \frac12\log\frac43+\psi(\kappa_*/2)=0.
 \label{eq:critical-tensor-degree}
\end{equation}
Moreover, throughout this range the empirical pair-angle moments obey
\begin{equation}
 \frac1{N_R^2}\sum_{\mathbf x,\mathbf y\in A_R(\cL)}
 \langle\widehat{\mathbf x},\widehat{\mathbf y}\rangle^k
 =\bigl(1+O(e^{-\beta n})\bigr)m_{k,n}
 \label{eq:pair-angle-moments}
\end{equation}
simultaneously for the same even degrees.

Conversely, if $\log c+\psi(\kappa/2)<0$ and
$k_n=2\lfloor\kappa n/2\rfloor$, then for some $a>0$, with probability
$1-e^{-\Omega(n)}$,
\begin{equation}\label{eq:tensor-degree-obstruction}
 \frac{\|M^{(k_n)}_{\cL,R}-M^{(k_n)}_\sigma\|_F}
      {\|M^{(k_n)}_\sigma\|_F}\ge e^{an}.
\end{equation}
This also holds for a random-affine shell.  Thus the strict degree boundary
is sharp at exponential scale for this relative tensor norm.
\end{theorem}

Figure~\ref{fig:tensor-capacity} displays the entropy condition in
Theorem~\ref{thm:linear-degree-tensor-moments} as a density--degree tradeoff.
The strict sides of this boundary distinguish relative tensor agreement
from an exponentially diverging relative error.  Denser shells permit
agreement through a larger linear degree.

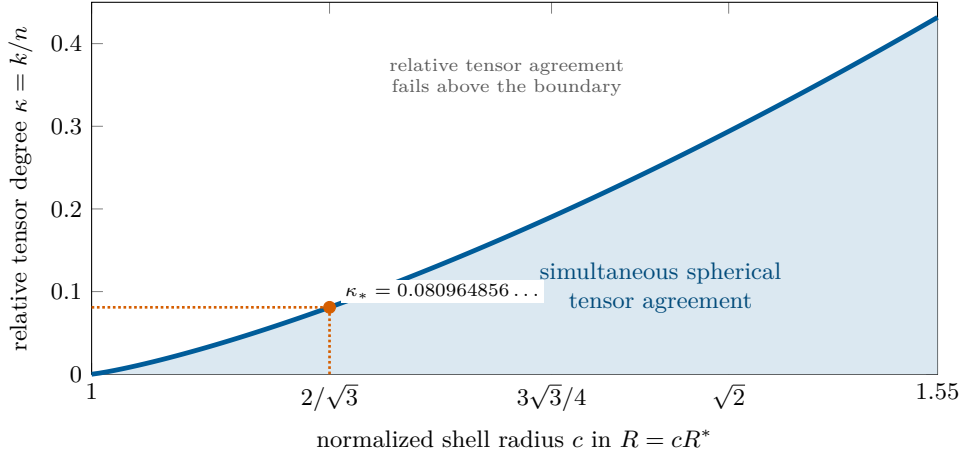
\begin{figure}[t]
\centering
\begin{tikzpicture}
\begin{axis}[
  width=.84\textwidth,
  height=6.5cm,
  xmin=1,xmax=1.55,
  ymin=0,ymax=.45,
  xlabel={normalized shell radius $c$ in $R=c\Rstar$},
  ylabel={relative tensor degree $\kappa=k/n$},
  xtick={1,1.1547005,1.2990381,1.4142136,1.55},
  xticklabels={$1$,$2/\sqrt3$,$3\sqrt3/4$,$\sqrt2$,$1.55$},
  ytick={0,.1,.2,.3,.4},
  tick label style={font=\small},
  label style={font=\small},
  clip=false
]
\addplot[draw=none,fill=figblue!14]
  coordinates {
    (1.000000000,0.000000000000)
    (1.005000000,0.001305302015)
    (1.010000000,0.002908834335)
    (1.015000000,0.004676392833)
    (1.020000000,0.006569579268)
    (1.025000000,0.008567860576)
    (1.030000000,0.010658064429)
    (1.035000000,0.012830873301)
    (1.040000000,0.015079275609)
    (1.045000000,0.017397763606)
    (1.050000000,0.019781872043)
    (1.055000000,0.022227892098)
    (1.060000000,0.024732683634)
    (1.065000000,0.027293546354)
    (1.070000000,0.029908128156)
    (1.075000000,0.032574357989)
    (1.080000000,0.035290395403)
    (1.085000000,0.038054591856)
    (1.090000000,0.040865460473)
    (1.095000000,0.043721652035)
    (1.100000000,0.046621935659)
    (1.105000000,0.049565183038)
    (1.110000000,0.052550355458)
    (1.115000000,0.055576492992)
    (1.120000000,0.058642705417)
    (1.125000000,0.061748164533)
    (1.130000000,0.064892097598)
    (1.135000000,0.068073781704)
    (1.140000000,0.071292538900)
    (1.145000000,0.074547731969)
    (1.150000000,0.077838760723)
    (1.154700538,0.080964856009)
    (1.155000000,0.081165058761)
    (1.160000000,0.084526090599)
    (1.165000000,0.087921349134)
    (1.170000000,0.091350353386)
    (1.175000000,0.094812646480)
    (1.180000000,0.098307793839)
    (1.185000000,0.101835381559)
    (1.190000000,0.105395014947)
    (1.195000000,0.108986317191)
    (1.200000000,0.112608928160)
    (1.205000000,0.116262503312)
    (1.210000000,0.119946712697)
    (1.215000000,0.123661240044)
    (1.220000000,0.127405781930)
    (1.225000000,0.131180047012)
    (1.230000000,0.134983755324)
    (1.235000000,0.138816637629)
    (1.240000000,0.142678434817)
    (1.245000000,0.146568897355)
    (1.250000000,0.150487784771)
    (1.255000000,0.154434865177)
    (1.260000000,0.158409914829)
    (1.265000000,0.162412717714)
    (1.270000000,0.166443065165)
    (1.275000000,0.170500755505)
    (1.280000000,0.174585593708)
    (1.285000000,0.178697391090)
    (1.290000000,0.182835965014)
    (1.295000000,0.187001138610)
    (1.299038106,0.190384319680)
    (1.300000000,0.191192740526)
    (1.305000000,0.195410604673)
    (1.310000000,0.199654570008)
    (1.315000000,0.203924480308)
    (1.320000000,0.208220183974)
    (1.325000000,0.212541533834)
    (1.330000000,0.216888386967)
    (1.335000000,0.221260604527)
    (1.340000000,0.225658051582)
    (1.345000000,0.230080596961)
    (1.350000000,0.234528113110)
    (1.355000000,0.239000475948)
    (1.360000000,0.243497564744)
    (1.365000000,0.248019261985)
    (1.370000000,0.252565453261)
    (1.375000000,0.257136027153)
    (1.380000000,0.261730875123)
    (1.385000000,0.266349891415)
    (1.390000000,0.270992972952)
    (1.395000000,0.275660019250)
    (1.400000000,0.280350932325)
    (1.405000000,0.285065616608)
    (1.410000000,0.289803978865)
    (1.414213562,0.293815373340)
    (1.415000000,0.294565928122)
    (1.420000000,0.299351375585)
    (1.425000000,0.304160234574)
    (1.430000000,0.308992420455)
    (1.435000000,0.313847850572)
    (1.440000000,0.318726444186)
    (1.445000000,0.323628122417)
    (1.450000000,0.328552808182)
    (1.455000000,0.333500426147)
    (1.460000000,0.338470902667)
    (1.465000000,0.343464165739)
    (1.470000000,0.348480144953)
    (1.475000000,0.353518771444)
    (1.480000000,0.358579977847)
    (1.485000000,0.363663698256)
    (1.490000000,0.368769868176)
    (1.495000000,0.373898424489)
    (1.500000000,0.379049305415)
    (1.505000000,0.384222450469)
    (1.510000000,0.389417800430)
    (1.515000000,0.394635297304)
    (1.520000000,0.399874884293)
    (1.525000000,0.405136505759)
    (1.530000000,0.410420107195)
    (1.535000000,0.415725635194)
    (1.540000000,0.421053037422)
    (1.545000000,0.426402262587)
    (1.550000000,0.431773260412)
    (1.550000000,0)
    (1.000000000,0)
  } \closedcycle;
\addplot[figblue,line width=1.8pt]
  coordinates {
    (1.000000000,0.000000000000)
    (1.005000000,0.001305302015)
    (1.010000000,0.002908834335)
    (1.015000000,0.004676392833)
    (1.020000000,0.006569579268)
    (1.025000000,0.008567860576)
    (1.030000000,0.010658064429)
    (1.035000000,0.012830873301)
    (1.040000000,0.015079275609)
    (1.045000000,0.017397763606)
    (1.050000000,0.019781872043)
    (1.055000000,0.022227892098)
    (1.060000000,0.024732683634)
    (1.065000000,0.027293546354)
    (1.070000000,0.029908128156)
    (1.075000000,0.032574357989)
    (1.080000000,0.035290395403)
    (1.085000000,0.038054591856)
    (1.090000000,0.040865460473)
    (1.095000000,0.043721652035)
    (1.100000000,0.046621935659)
    (1.105000000,0.049565183038)
    (1.110000000,0.052550355458)
    (1.115000000,0.055576492992)
    (1.120000000,0.058642705417)
    (1.125000000,0.061748164533)
    (1.130000000,0.064892097598)
    (1.135000000,0.068073781704)
    (1.140000000,0.071292538900)
    (1.145000000,0.074547731969)
    (1.150000000,0.077838760723)
    (1.154700538,0.080964856009)
    (1.155000000,0.081165058761)
    (1.160000000,0.084526090599)
    (1.165000000,0.087921349134)
    (1.170000000,0.091350353386)
    (1.175000000,0.094812646480)
    (1.180000000,0.098307793839)
    (1.185000000,0.101835381559)
    (1.190000000,0.105395014947)
    (1.195000000,0.108986317191)
    (1.200000000,0.112608928160)
    (1.205000000,0.116262503312)
    (1.210000000,0.119946712697)
    (1.215000000,0.123661240044)
    (1.220000000,0.127405781930)
    (1.225000000,0.131180047012)
    (1.230000000,0.134983755324)
    (1.235000000,0.138816637629)
    (1.240000000,0.142678434817)
    (1.245000000,0.146568897355)
    (1.250000000,0.150487784771)
    (1.255000000,0.154434865177)
    (1.260000000,0.158409914829)
    (1.265000000,0.162412717714)
    (1.270000000,0.166443065165)
    (1.275000000,0.170500755505)
    (1.280000000,0.174585593708)
    (1.285000000,0.178697391090)
    (1.290000000,0.182835965014)
    (1.295000000,0.187001138610)
    (1.299038106,0.190384319680)
    (1.300000000,0.191192740526)
    (1.305000000,0.195410604673)
    (1.310000000,0.199654570008)
    (1.315000000,0.203924480308)
    (1.320000000,0.208220183974)
    (1.325000000,0.212541533834)
    (1.330000000,0.216888386967)
    (1.335000000,0.221260604527)
    (1.340000000,0.225658051582)
    (1.345000000,0.230080596961)
    (1.350000000,0.234528113110)
    (1.355000000,0.239000475948)
    (1.360000000,0.243497564744)
    (1.365000000,0.248019261985)
    (1.370000000,0.252565453261)
    (1.375000000,0.257136027153)
    (1.380000000,0.261730875123)
    (1.385000000,0.266349891415)
    (1.390000000,0.270992972952)
    (1.395000000,0.275660019250)
    (1.400000000,0.280350932325)
    (1.405000000,0.285065616608)
    (1.410000000,0.289803978865)
    (1.414213562,0.293815373340)
    (1.415000000,0.294565928122)
    (1.420000000,0.299351375585)
    (1.425000000,0.304160234574)
    (1.430000000,0.308992420455)
    (1.435000000,0.313847850572)
    (1.440000000,0.318726444186)
    (1.445000000,0.323628122417)
    (1.450000000,0.328552808182)
    (1.455000000,0.333500426147)
    (1.460000000,0.338470902667)
    (1.465000000,0.343464165739)
    (1.470000000,0.348480144953)
    (1.475000000,0.353518771444)
    (1.480000000,0.358579977847)
    (1.485000000,0.363663698256)
    (1.490000000,0.368769868176)
    (1.495000000,0.373898424489)
    (1.500000000,0.379049305415)
    (1.505000000,0.384222450469)
    (1.510000000,0.389417800430)
    (1.515000000,0.394635297304)
    (1.520000000,0.399874884293)
    (1.525000000,0.405136505759)
    (1.530000000,0.410420107195)
    (1.535000000,0.415725635194)
    (1.540000000,0.421053037422)
    (1.545000000,0.426402262587)
    (1.550000000,0.431773260412)
  };
\addplot[figorange,densely dotted,line width=1pt]
  coordinates {(1.1547005,0) (1.1547005,0.080964856)};
\addplot[figorange,densely dotted,line width=1pt]
  coordinates {(1,0.080964856) (1.1547005,0.080964856)};
\addplot[figorange,only marks,mark=*,mark size=2.2pt]
  coordinates {(1.1547005,0.080964856)};
\node[font=\small,align=center,text=figblue!80!black]
  at (axis cs:1.37,.105)
  {simultaneous spherical\\tensor agreement};
\node[font=\scriptsize,anchor=south west,fill=white,inner sep=1.5pt]
  at (axis cs:1.162,.086)
  {$\kappa_*=0.080964856\ldots$};
\node[font=\scriptsize,align=center,text=gray!70!black]
  at (axis cs:1.27,.36)
  {relative tensor agreement\\fails above the boundary};
\end{axis}
\end{tikzpicture}
\caption[Tensor capacity of a dense shell.]{Tensor capacity of a dense shell.
The blue shading numerically illustrates the region
$\log c+\psi(\kappa/2)>0$; its boundary is drawn from densely sampled
numerical roots.  For every pair $(c,\kappa)$ satisfying this strict
inequality, every smaller even degree
$2\le k\le\kappa n$ then matches the spherical tensor moment with
exponentially small relative error and exponentially small failure
probability.  The marked point is the paper's near-critical specialization
$c=\sqrt{4/3}=2/\sqrt3$, where
$\kappa_*=0.080964856\ldots$.  Strictly above the boundary,
\eqref{eq:tensor-degree-obstruction} gives exponentially diverging relative
error at the largest indicated even degree.  No boundary claim is made
at equality.}
\label{fig:tensor-capacity}
\end{figure}

The short-vector angular limits of S\"odergren and Holm
\cite{SodergrenAngles2011,Holm2022} concern joint laws of sparse initial
vector families.  Here the population is the complete exponential shell:
the first theorem controls all test directions in one operator bound, and
the second controls even tensor degrees linear in the dimension.  The
spherical tensor means and the order-two Rogers formula are classical.
They yield the aggregate quenched bounds proved below.

Both theorems follow from the next Hilbert-valued estimate.  Its Rogers order
stays two even when the target Hilbert space is a high tensor power.

\subsection{A Hilbert-valued order-two estimate}

\begin{lemma}[Hilbert-valued angular second moment]
\label{lem:hilbert-angular-second-moment}
Let $\mathsf H$ be a finite-dimensional real Hilbert space and let
$h:S^{n-1}\to\mathsf H$ be Borel measurable with $\|h(\mathbf u)\|\le1$.
Write
\[
 \overline h:=\int_{S^{n-1}}h(\mathbf u)\,d\sigma_{n-1}(\mathbf u),
 \qquad
 Z_{h,R}(\cL):=\sum_{\mathbf x\in A_R(\cL)}h(\widehat{\mathbf x}).
\]
There is a deterministic $\delta_n=e^{-\Omega(n\log n)}$, independent of
$R$, $\mathsf H$, and $h$, such that
\begin{equation}
 \E_{\cL}\bigl\|Z_{h,R}-P_R\overline h\bigr\|^2
 \le(2+\delta_n)P_R.
 \label{eq:hilbert-angular-centered}
\end{equation}
For a random affine coset put
\[
 \widetilde Z_{h,R}:=
 \sum_{\mathbf z\in(\boldsymbol\tau+\cL)\cap S_R}
 h(\widehat{\mathbf z}).
\]
Then the exact identity
\begin{equation}
 \E_{\cL,\boldsymbol\tau}
 \bigl\|\widetilde Z_{h,R}-P_R\overline h\bigr\|^2
 =P_R\int_{S^{n-1}}\|h(\mathbf u)\|^2
 \,d\sigma_{n-1}(\mathbf u)
 \label{eq:hilbert-angular-affine}
\end{equation}
holds.
\end{lemma}

\begin{proof}
Choose an orthonormal basis of $\mathsf H$, expand the squared norm, and
apply the order-two Rogers formula componentwise.  Each resulting signed
kernel is handled by its positive and negative parts; bounded shell support
makes both applications finite.  The full-rank term equals
$P_R^2\|\overline h\|^2$ and cancels after centering.  A rank-one term has
$\mathbf y=(a/q)\mathbf x$, where $q\ge1$, $a\in\Z\setminus\{0\}$,
$(a,q)=1$, and coefficient $q^{-n}$.  For $(q,a)=(1,1)$ and $(1,-1)$ the
sum of the two integrals is
\[
 P_R\int_{S^{n-1}}
 \bigl(\|h(\mathbf u)\|^2+
       \langle h(\mathbf u),h(-\mathbf u)\rangle\bigr)
 \,d\sigma_{n-1}(\mathbf u)
 \le2P_R.
\]

Every other supported reduced ratio $b=|a|/q$ must satisfy
\[
 \frac{1-w_n}{1+w_n}\le b\le
 \frac{1+w_n}{1-w_n}.
\]
Since $w_n=n^{-2}$ and $a\ne\pm q$, this forces $q\ge c n^2$ for an
absolute $c>0$.  For a fixed $q$ there are $O(1+q/n^2)$ possible
numerators.  Taking absolute values of their Hilbert inner products bounds
their total contribution by
\[
 P_R\sum_{q\ge cn^2}q^{-n}O(1+q/n^2)
 =P_R e^{-\Omega(n\log n)}.
\]
This proves \eqref{eq:hilbert-angular-centered}.

For the affine identity, the diagonal contribution to the uncentered second
moment is
$P_R\int\|h\|^2d\sigma_{n-1}$.  Write two distinct affine points as
$\mathbf z$ and $\mathbf z+\mathbf v$ with
$\mathbf v\in\cL\setminus\{\mathbf0\}$.  Applying
Corollary~\ref{cor:affine-centered-campbell} to their Hilbert inner product
shows that the off-diagonal contribution is exactly
$P_R^2\|\overline h\|^2$.  Subtracting the squared mean proves
\eqref{eq:hilbert-angular-affine}.
\end{proof}

\subsection{Proofs of the angular theorems and the degree range}

\begin{proof}[Proof of Theorem~\ref{thm:dense-shell-isometry}]
Apply Lemma~\ref{lem:hilbert-angular-second-moment} first to $h\equiv1$ and
then to the Frobenius-Hilbert-space-valued map
$h(\mathbf u)=\mathbf u\mathbf u^{\mathsf T}$.  Since
$\|\mathbf u\mathbf u^{\mathsf T}\|_F=1$ and
$\int\mathbf u\mathbf u^{\mathsf T}d\sigma_{n-1}=I_n/n$, Markov's
inequality gives
\begin{align*}
 \Prb[|N_R-P_R|>\varepsilon P_R/4]
 &\le \frac{C}{\varepsilon^2P_R},\\
 \Prb[\|Q_R-P_RI_n/n\|_F>\varepsilon P_R/(4n)]
 &\le \frac{Cn^2}{\varepsilon^2P_R}.
\end{align*}
On the complement, $N_R\ge7P_R/8$ and
\[
 \|C_R-I_n\|_{\rm op}
 \le\frac n{N_R}\|Q_R-P_RI_n/n\|_F
      +\left|\frac{P_R}{N_R}-1\right|<\varepsilon.
\]
The affine proof is identical, using
\eqref{eq:hilbert-angular-affine}.  Equation
\eqref{eq:dense-shell-analysis-map} is the quadratic-form formulation of the
operator bound.  Finally,
\[
 P_R=c_n^n\bigl((1+n^{-2})^n-(1-n^{-2})^n\bigr)
 =\frac{2c_n^n}{n}(1+O(n^{-2})),
\]
which proves the exponential specialization.
\end{proof}

\begin{proof}[Proof of Theorem~\ref{thm:linear-degree-tensor-moments}]
Apply Lemma~\ref{lem:hilbert-angular-second-moment} to
$h(\mathbf u)=\mathbf u^{\otimes k}$, which has Hilbert--Schmidt norm one.
Write
\[
 Z_k:=\sum_{\mathbf x\in A_R(\cL)}\widehat{\mathbf x}^{\otimes k},
 \qquad M_k:=M^{(k)}_\sigma,
 \qquad \|M_k\|_F^2=m_{k,n}.
\]
There are two bad events.  Markov's inequality and the Hilbert-valued lemma
give
\[
 \Prb\left[\|Z_k-P_RM_k\|_F>
               \frac{\varepsilon}{4}P_R\sqrt{m_{k,n}}\right]
 \le\frac{C}{\varepsilon^2P_Rm_{k,n}},
\]
while Lemma~\ref{lem:thin-shell-count} gives
\[
 \Prb\left[|N_R-P_R|>\frac{\varepsilon}{4}P_R\right]
 \le\frac{C}{\varepsilon^2P_R}
 \le\frac{C}{\varepsilon^2P_Rm_{k,n}}.
\]
On their complement $N_R\ge(1-\varepsilon/4)P_R$ and
\begin{align*}
 \left\|\frac{Z_k}{N_R}-M_k\right\|_F
 &\le\frac{\|Z_k-P_RM_k\|_F}{N_R}
       +\left|\frac{P_R}{N_R}-1\right|\|M_k\|_F\\
 &<\varepsilon\sqrt{m_{k,n}},
\end{align*}
for $0<\varepsilon\le1/2$.  This proves
\eqref{eq:tensor-moment-tail} and also shows explicitly that the random
denominator is not absorbed into the exponentially small spherical tensor
norm.  The affine case uses the exact affine second moment and
Corollary~\ref{cor:affine-count-moments} in the same two-event argument.

For odd $k$, antipodal centered-shell points cancel pairwise.  For even $k$,
the first equality in
\eqref{eq:spherical-tensor-norm} follows from
\[
 \langle M^{(k)}_\sigma,M^{(k)}_\sigma\rangle
 =\E\langle\mathbf U,\mathbf V\rangle^k,
\]
and the last equality is the standard even moment of one coordinate of a
uniform spherical vector.

Writing $k=2r$ and using the gamma-function form
\[
 m_{2r,n}=
 \frac{\Gamma(n/2)\Gamma(r+1/2)}
      {\sqrt\pi\,\Gamma(n/2+r)},
\]
Stirling's formula gives, whenever $r/n\to\rho$,
\begin{equation}
 \frac1n\log m_{2r,n}=\psi(\rho)+o(1).
 \label{eq:tensor-stirling}
\end{equation}
Also $m_{k+2,n}/m_{k,n}=(k+1)/(n+k)<1$.  A union bound in
\eqref{eq:tensor-moment-tail}, using the largest permitted even degree, is
therefore at most
\[
 \exp\bigl(-[\log c+\psi(\kappa/2)-2\beta]n+o(n)\bigr).
\]
This proves the simultaneous claim.  The numerical constant in
\eqref{eq:critical-tensor-degree} is the unique positive root of the
displayed strictly decreasing equation.  Finally,
\[
 \frac1{N_R^2}\sum_{\mathbf x,\mathbf y}
 \langle\widehat{\mathbf x},\widehat{\mathbf y}\rangle^k
 =\|M^{(k)}_{\cL,R}\|_F^2,
\]
and relative norm closeness changes the squared norm by at most
$2e^{-\beta n}+e^{-2\beta n}$.

For the converse, rotation invariance gives
$\langle\mathbf u^{\otimes k},M^{(k)}_\sigma\rangle=m_{k,n}$ for every
unit vector $\mathbf u$ and even $k$.  For $N>0$ unit directions, let $M^{(k)}_{\rm emp}$ be the average
of their $k$th tensor powers.  It satisfies the deterministic identity
and bound
\begin{equation}\label{eq:tensor-diagonal-obstruction}
 \frac{\|M^{(k)}_{\rm emp}-M^{(k)}_\sigma\|_F^2}{m_{k,n}}
 =\frac{\|M^{(k)}_{\rm emp}\|_F^2}{m_{k,n}}-1
 \ge\frac1{Nm_{k,n}}-1.
\end{equation}
Indeed all pair inner products raised to the even power $k$ are
nonnegative, and the $N$ diagonal terms give
$\|M^{(k)}_{\rm emp}\|_F^2\ge1/N$.
The shell count gives $N_R=c^{n+o(n)}$ with
failure $e^{-\Omega(n)}$.  At $k=k_n$, \eqref{eq:tensor-stirling} now
shows that $N_Rm_{k_n,n}=\exp(n[\log c+\psi(\kappa/2)]+o(n))$.
Under the negative-sign hypothesis, \eqref{eq:tensor-diagonal-obstruction}
proves \eqref{eq:tensor-degree-obstruction}, with any sufficiently small
fixed $a>0$.  The argument applies equally to affine shells.
\end{proof}

\begin{remark}[Sharpness and the role of diagonal pairs]
The union over $O(n)$ even degrees costs only a polynomial factor.
The exponential restriction is $P_Rm_{k,n}\to\infty$, and
\eqref{eq:tensor-diagonal-obstruction} shows why it is necessary for the
present statistic on the opposite strict side.
The constant $\kappa_*$ is consequently not merely an artefact of that
union bound.  At equality, polynomial factors and the rate at which
$c_n$ approaches $c$ matter; the strict-sign conclusions above do not
resolve all such boundary sequences.  Removing diagonal pairs would
define a different pair statistic and remove this particular obstruction.
\end{remark}

\appendix
\section{Bounded angular kernels and upper tails}\label{app:angular-kernels}
Throughout this appendix $n\ge3$, and $\sigma_{n-1}$ is normalized
spherical measure as in Section~\ref{sec:angular-transfer}.
The results below are consequences of the order-two formulas; they give
exact expectations and one-sided upper tails for bounded nonnegative
kernels.

Let $K:[-1,1]\to[0,\infty)$ be bounded and Borel.  Write
\begin{equation}
 \overline K_n:=
 \int_{S^{n-1}}\int_{S^{n-1}}
 K(\langle\mathbf u,\mathbf v\rangle)
 \,d\sigma_{n-1}(\mathbf u)d\sigma_{n-1}(\mathbf v),
 \qquad K_\infty:=\|K\|_\infty.
 \label{eq:angular-kernel-mean}
\end{equation}
For $A=S_R$, $B=S_{R'}$, with volumes $P,Q$, define
\begin{align*}
 W_{K;A,B}(\cL)
 &:=\sum_{\mathbf x\in\cL\cap A}
     \sum_{\mathbf y\in\cL\cap B}
 K(\langle\widehat{\mathbf x},\widehat{\mathbf y}\rangle),\\
 W^{\rm ind}_{K;A,B}(\cL)
 &:=\sum_{\substack{\mathbf x\in\cL\cap A,\ \mathbf y\in\cL\cap B\\
             \dim\Span_{\R}\{\mathbf x,\mathbf y\}=2}}
 K(\langle\widehat{\mathbf x},\widehat{\mathbf y}\rangle).
\end{align*}

\begin{proposition}[Exact independent-pair transfer and collinear correction]
\label{thm:weighted-angular-transfer}
For every $R,R'>0$,
\begin{equation}
 \E_\cL W^{\rm ind}_{K;A,B}=PQ\overline K_n.
 \label{eq:independent-pair-exact-transfer}
\end{equation}
For the workload containing all ordered pairs,
\begin{equation}
 \E_\cL W_{K;A,B}=PQ\overline K_n+\mathcal R_{K;A,B},
 \label{eq:all-pair-angular-transfer}
\end{equation}
where the nonnegative correction is exactly
\begin{equation}
 \mathcal R_{K;A,B}
 =\sum_{q\ge1}\ 
   \sum_{\substack{a\in\Z\setminus\{0\}\\(a,q)=1}}
 q^{-n}K(\operatorname{sgn}a)
 \int_{\R^n}\mathbf1_A(\mathbf x)
     \mathbf1_B\!\left(\frac aq\mathbf x\right)d\mathbf x.
 \label{eq:collinear-correction-exact}
\end{equation}
If $R'/R$ remains in a fixed compact subinterval $J\Subset(0,\infty)$,
then
\begin{equation}
 0\le\mathcal R_{K;A,B}
 \le C_J K_\infty\min(P,Q).
 \label{eq:collinear-correction-bound}
\end{equation}

For two points in the same random affine coset, let
\begin{equation*}
 \widetilde W^{\ne}_{K;A,B}:=
 \sum_{\substack{\mathbf z\in(\boldsymbol\tau+\cL)\cap A,\
                    \mathbf z'\in(\boldsymbol\tau+\cL)\cap B\\
                    \mathbf z'\ne\mathbf z}}
 K(\langle\widehat{\mathbf z},\widehat{\mathbf z'}\rangle).
\end{equation*}
Then
\begin{equation}
 \E_{\cL,\boldsymbol\tau}\widetilde W^{\ne}_{K;A,B}
 =PQ\overline K_n.
 \label{eq:affine-offdiagonal-transfer}
\end{equation}
The mixed affine--centered workload also satisfies the exact identity
\begin{equation}
 \E_{\cL,\boldsymbol\tau}
 \sum_{\mathbf z\in(\boldsymbol\tau+\cL)\cap A}
 \sum_{\mathbf v\in\cL\cap B}
 K(\langle\widehat{\mathbf z},\widehat{\mathbf v}\rangle)
 =PQ\overline K_n.
 \label{eq:mixed-affine-centered-transfer}
\end{equation}
\end{proposition}

\begin{proof}
Apply the two-vector Rogers formula to
\[
 F(\mathbf x,\mathbf y)=\mathbf1_A(\mathbf x)\mathbf1_B(\mathbf y)
 K(\langle\widehat{\mathbf x},\widehat{\mathbf y}\rangle).
\]
Set $F=0$ when either coordinate is zero.
Its full-rank term is the double Euclidean integral, which equals
$PQ\overline K_n$ by polar coordinates.  If $F$ is multiplied by the
indicator that $\mathbf x,\mathbf y$ are linearly independent, every
rank-one Rogers integral vanishes.  This proves
\eqref{eq:independent-pair-exact-transfer}.  Without that indicator, the
rank-one parametrization $\mathbf y=(a/q)\mathbf x$ gives exactly
\eqref{eq:collinear-correction-exact}.

For completeness, put $\xi=R'/R\in J$.  Shell support forces
$|a|/q=\xi(1+O(n^{-2}))$, so for a fixed $q$ there are
$O_J(1+q/n^2)$ supported absolute numerators.  Moreover, if the integral in
\eqref{eq:collinear-correction-exact} is denoted by $I_{a,q}$, then
\[
 I_{a,q}\le P,
 \qquad
 I_{a,q}\le(q/|a|)^nQ.
\]
When $\xi\ge1$, sum the first bound against $q^{-n}$; when $\xi\le1$,
re-index by $|a|$ and sum the second bound.  In either case
$\sum q^{-n}I_{a,q}\le C_J\min(P,Q)$, including both signs.  This proves
\eqref{eq:collinear-correction-bound}.

For \eqref{eq:affine-offdiagonal-transfer}, write
$\mathbf z'=\mathbf z+\mathbf v$ with
$\mathbf v\in\cL\setminus\{\mathbf0\}$ and apply
Corollary~\ref{cor:affine-centered-campbell}; the change of variables
$(\mathbf z,\mathbf v)\mapsto(\mathbf z,\mathbf z')$ has unit Jacobian.
The same corollary applied directly to the affine point and centered vector
proves \eqref{eq:mixed-affine-centered-transfer}.  The centered shell $B$
does not contain the origin.
\end{proof}

The annealed identity has the following one-sided quenched consequence.  It
is the precise statement that a dense random-lattice shell has no
anomalously large nonnegative aggregate angular workload relative to two
independent spherical directions, up to a chosen Markov slack.  Since no
variance estimate for $W_{K;A,B}$ is used, the result supplies neither a
matching lower tail nor convergence of the normalized workload.
Put $N_A:=|\cL\cap A|$ and $N_B:=|\cL\cap B|$.
In the normalized formulas below, a ratio with $N_AN_B=0$ is interpreted as
$+\infty$.

\begin{corollary}[Quenched upper-tail bound for weighted kernels]
\label{cor:quenched-weighted-kernel-transfer}
Assume $\overline K_n>0$ and $R'/R\in J\Subset(0,\infty)$.  For every
$t\ge1$,
\begin{equation}
 \Prb_\cL\left[\frac{W_{K;A,B}}{PQ}>t\overline K_n\right]
 \le\frac1t\left(1+
   \frac{C_JK_\infty}{\overline K_n\max(P,Q)}\right).
 \label{eq:quenched-kernel-star-volume}
\end{equation}
Moreover,
\begin{align}
 &\Prb_\cL\left[
 N_A<\frac P2\ \text{or}\ N_B<\frac Q2\ \text{or}\
 \frac{W_{K;A,B}}{N_AN_B}>4t\overline K_n\right]\notag\\
 &\qquad\le
 C_J\left(\frac1P+\frac1Q\right)
 +\frac1t\left(1+
   \frac{C_JK_\infty}{\overline K_n\max(P,Q)}\right),
 \label{eq:quenched-kernel-star-count}
\end{align}
For $W^{\rm ind}_{K;A,B}$ the last fraction inside parentheses is absent.
For the affine off-diagonal or mixed workload in
\eqref{eq:affine-offdiagonal-transfer}--
\eqref{eq:mixed-affine-centered-transfer}, the analogous bound also has no
collinear fraction.  The bounds hold simultaneously for a predetermined
finite family after summing their right-hand sides.
\end{corollary}

\begin{proof}
The shell-count instances of
Lemma~\ref{lem:hilbert-angular-second-moment} and Chebyshev give the first two
terms.  On $N_A\ge P/2$ and $N_B\ge Q/2$, the displayed normalized workload
is at most $4W_{K;A,B}/(PQ)$.  Markov's inequality and
\eqref{eq:all-pair-angular-transfer}--
\eqref{eq:collinear-correction-bound} give the remaining terms.  Use
\eqref{eq:independent-pair-exact-transfer},
\eqref{eq:affine-offdiagonal-transfer}, or
\eqref{eq:mixed-affine-centered-transfer} for the correction-free versions,
and then take a union bound.
\end{proof}

\begin{corollary}[Common-cap statistics of complete shells]
\label{cor:common-cap-statistics}
Let $A=S_R$, $B=S_{R'}$ be predetermined shells with volumes $P,Q$ and
$R'/R\in J\Subset(0,\infty)$.  For $a,b\in(0,1)$ and $s\in[-1,1]$ define
\begin{align*}
 \mathsf C_n(a)
 &:=\Prb_{\mathbf z\sim\sigma_{n-1}}
       [\langle\mathbf z,\mathbf u\rangle\ge a],\\
 K_{n;a,b}(s)
 &:=\Prb_{\mathbf z\sim\sigma_{n-1}}
       [\langle\mathbf z,\mathbf u\rangle\ge a,\
        \langle\mathbf z,\mathbf v\rangle\ge b],
 \qquad \langle\mathbf u,\mathbf v\rangle=s,
\end{align*}
where $\mathbf u,\mathbf v\in S^{n-1}$; rotational invariance makes these
definitions independent of their choice.  Put
$\Phi_{\cL}^{a,b}(A,B):=W_{K_{n;a,b};A,B}(\cL)/(PQ)$.
Then
\begin{equation}
 0\le \E_{\cL}\Phi_{\cL}^{a,b}(A,B)-\mathsf C_n(a)\mathsf C_n(b)
 \le \frac{C_J\min\{\mathsf C_n(a),\mathsf C_n(b)\}}{\max(P,Q)}.
 \label{eq:common-cap-mean}
\end{equation}
For every $t\ge1$,
\begin{equation}
 \Prb_{\cL}\!\left[
   \Phi_{\cL}^{a,b}(A,B)>t\mathsf C_n(a)\mathsf C_n(b)\right]
 \le\frac1t\left(1+
 \frac{C_J}{\max(P,Q)\max\{\mathsf C_n(a),\mathsf C_n(b)\}}\right).
 \label{eq:common-cap-tail}
\end{equation}
For the mixed affine--centered sum in
\eqref{eq:mixed-affine-centered-transfer}, with $K=K_{n;a,b}$ and
normalized by $PQ$, the expectation over $(\cL,\boldsymbol\tau)$ equals
$\mathsf C_n(a)\mathsf C_n(b)$ and the corresponding tail bound is $1/t$.
The conclusions hold simultaneously for any predetermined finite family
of shell pairs and thresholds after summing the failure bounds.
\end{corollary}

\begin{proof}
The kernel $K=K_{n;a,b}$ is bounded, nonnegative, and Borel, with
$K_\infty=\min\{\mathsf C_n(a),\mathsf C_n(b)\}$, attained at $s=1$.
Let $\mathbf U,\mathbf V,\mathbf z$ be independent uniform directions.
Conditioning on $\mathbf z$ makes the two cap events independent, so Fubini gives
$\overline K_n=\mathsf C_n(a)\mathsf C_n(b)$.
Substitute these identities into
Proposition~\ref{thm:weighted-angular-transfer} and
Corollary~\ref{cor:quenched-weighted-kernel-transfer}, and use
\eqref{eq:mixed-affine-centered-transfer} for the affine assertion.
\end{proof}

\begin{remark}[Scope of the kernel bounds]
The exact independent-pair identities are direct consequences of the
classical order-two Rogers formula and affine unfolding.  Their role here
is a reusable bounded-kernel interface with an explicit centered collinear
correction.  The resulting upper tails do not imply relative concentration
for every exponentially rare kernel or uniform discrepancy over all caps.
\end{remark}

\begin{table}[t]
\centering\footnotesize
\renewcommand{\arraystretch}{1.14}
\setlength{\tabcolsep}{4pt}
\begin{tabularx}{\textwidth}{@{}>{\raggedright\arraybackslash}p{.22\textwidth}
 X >{\raggedright\arraybackslash}p{.22\textwidth}@{}}
\toprule
Symbol & Meaning & Definition\\
\midrule
$X_n,\widetilde X_n$ & Centered and random-affine lattice probability spaces
 & Section~\ref{sec:setup}\\
$\kappa_n,R^*$ & Unit-ball volume and unit-volume-ball radius
 & Section~\ref{sec:setup}\\
$w_n,S_R,A_R$ & Relative half-width $n^{-2}$, continuum shell, and its lattice points
 & Sections~\ref{sec:setup}--\ref{sec:comparison-framework}\\
$N_R,P_R$; $P,Q$ & Shell count and volume; endpoint and target volumes
 & Sections~\ref{sec:comparison-framework}, \ref{sec:shell-moments}\\
$r_{\cL,R},K_R=\mu_R$ & Ordered representation count and continuum lens volume
 & \eqref{eq:representation-function}; Lemma~\ref{lem:thin-shell-coarea}\\
$U_{\cL,R},\Gamma_{\cL,R}$ & Uniform empirical shell law and its difference law
 & \eqref{eq:empirical-shell-measures}--\eqref{eq:empirical-difference-law}\\
$Z_{\cL,R},\Pi_{\cL,R}$ & Sampled lens mass and normalized comparator
 & \eqref{eq:framework-sampled-comparator}\\
$\ell,h,H$ & $\lceil\log n\rceil$; centered half-order; fixed order-range constant
 & Lemma~\ref{lem:growing-order-legal}\\
$\gamma,T_R$ & Near-unit ratio $1-1/\ell$ and band $S_{\gamma R}$
 & Section~\ref{sec:rogers-engine}\\
$T_{\alpha,R},m_*$ & Band $S_{\alpha R}$ and its minimal lens mean
 & Section~\ref{sec:shell-moments}\\
$\mu_*(R),\omega_R$ & Near-unit minimal lens mean and flattening weight; zero off the band
 & Before Lemma~\ref{lem:fiber-polar}\\
$a_R,a_R^\omega,\bar a_R,a^*$ & Unweighted/weighted endpoint means; near-unit spatial average and general supremum
 & Section~\ref{sec:shell-moments}; Corollary~\ref{cor:centered-unweighted-interface}\\
$D^\omega_{\cL,R,\alpha}$ & Weighted number of partners in a prescribed difference band
 & Proposition~\ref{thm:intrinsic-marked-moments}\\
$\mathcal I_D,q,r,k$ & Rogers summand including its coefficient; denominator, rank, tuple order
 & \eqref{eq:diagram-integral-definition}; Proposition~\ref{thm:rogers-general}\\
$\chi,f$ & Centered-mark indicator ($1$ centered, $0$ affine); free-fiber count $r-\chi$
 & Lemma~\ref{fs-lem:missing-fiber}\\
$\Delta_{\rm lp},\Xi_n$ & Low-parent gap and summed correction factor
 & \eqref{eq:low-parent-gap}, \eqref{eq:diagram-correction-budget}\\
$\mathcal N_\star$ & Shell-scale factor $(4/3)^{n/2}$
 & Section~\ref{sec:rogers-engine}\\
$R_\ell,h_\kappa$ & Standard buffered radius and adjustable logarithmic half-order
 & Theorem~\ref{fs-thm:main-pointwise}; Corollary~\ref{cor:moment-buffer}\\
$\mathcal B_{\rm crit},R_{\rm crit}$ & Critical target band and its outer radius
 & Proposition~\ref{thm:critical-ball-lenses}\\
$T_I$ & Targets whose relative lengths lie in $I$
 & Lemma~\ref{lem:bandwise-l2-approximation}\\
$I_p,\mathcal E,H_2$ & Integrated lens powers, additive energy, collision entropy
 & Lemma~\ref{lem:additive-continuum}; Section~\ref{sec:energy}\\
$M^{(k)}_{\cL,R},M^{(k)}_\sigma$ & Empirical and spherical directional moment tensors
 & Theorem~\ref{thm:linear-degree-tensor-moments}\\
$m_{k,n},\psi,\kappa_*$ & Spherical squared tensor norm, its rate function, and critical degree fraction
 & \eqref{eq:spherical-tensor-norm}--\eqref{eq:critical-tensor-degree}\\
\bottomrule
\end{tabularx}
\medskip
\begin{minipage}{\textwidth}
Parameters in a single proof or lemma are local to that statement:
in particular, $\beta$ is a radius buffer in
Corollary~\ref{cor:moment-buffer} and an error exponent in the angular
theorems; $\kappa$ controls moment order in that corollary and tensor degree
in Theorem~\ref{thm:linear-degree-tensor-moments}.
The Rayleigh quadratic form uses $\tau$, while $t$ counts parents in
the spherical multi-parent bound.
Generic errors $\delta_n,\eta_n$ in the angular inequalities have the
size specified in each statement; they are distinct from the fixed
sequence $\varepsilon_n$ used in $\Xi_n$.
\end{minipage}
\caption{Notation index.  All entries refer back to definitions in the text.}
\label{tab:notation-index}
\end{table}

\end{document}